\documentclass[12pt]{article}
\usepackage{mathtools}
\usepackage{amsthm}
\usepackage{csquotes}
\usepackage{amssymb}
\usepackage{graphicx}
\usepackage{subfig}
\usepackage{soul}  
\usepackage{xcolor}
\usepackage{empheq}
\usepackage{lipsum}
\usepackage{fullpage}
\usepackage{float} 

\usepackage[top= 2cm, bottom = 2 cm, left = 2.2 cm, right= 2.2 cm]{geometry}  
\usepackage[obeyspaces,hyphens,spaces]{url}
\usepackage[unicode]{hyperref}
\hypersetup{
    colorlinks=true,
    linkcolor=blue,
    citecolor=blue,
    filecolor=magenta,
    urlcolor=cyan
}
\usepackage{listings}

\definecolor{codegreen}{rgb}{0,0.6,0}
\definecolor{codegray}{rgb}{0.5,0.5,0.5}
\definecolor{codepurple}{rgb}{0.58,0,0.82}
\definecolor{backcolour}{rgb}{0.95,0.95,0.92}

\lstdefinestyle{mystyle}{
	backgroundcolor=\color{backcolour},   
	commentstyle=\color{codegreen},
	keywordstyle=\color{magenta},
	numberstyle=\footnotesize\color{codegray},
	stringstyle=\color{codepurple},
	basicstyle=\ttfamily\small,
	breakatwhitespace=false,         
	breaklines=true,                 
	captionpos=b,                    
	keepspaces=true,                 
	numbers=left,                    
	numbersep=5pt,                  
	showspaces=false,                
	showstringspaces=false,
	showtabs=false,                  
	tabsize=2
}

\definecolor{seagreen}{rgb}{0.18, 0.55, 0.34}
\definecolor{mediumviolet-red}{rgb}{0.78, 0.08, 0.52}
\definecolor{khaki}{rgb}{0.94, 0.9, 0.55}

\lstdefinelanguage{mypython}
{
	keywords=[1]{from, import, assert, not, print},
	keywordstyle=[1]{\color{mediumviolet-red}},
	keywords=[2]{surecr, torch, cp, lo, pl},
	keywordstyle=[2]{\color{seagreen}},
	numbers=none,
	upquote=true,
	showstringspaces=false,
	basicstyle=\ttfamily,
	columns=fullflexible,
	keepspaces=true,
	emph={True,False,as,def,return,float,class,match,switch,len},
	emphstyle={\color{seagreen}},
	frame=trBL,
	belowskip=1em,
	aboveskip=1em,
	captionpos=b
}

\usepackage[capitalize, nameinlink, noabbrev]{cleveref}
\crefname{equation}{}{}
\crefname{chapter}{Chapter}{Chapters}
\crefname{item}{item}{items}
\crefname{figure}{Figure}{Figures}
\crefname{theorem}{Theorem}{Theorems}
\crefname{lemma}{Lemma}{Lemmas}
\crefname{proposition}{Proposition}{Propositions}
\crefname{corollary}{Corollary}{Corollarys}
\crefname{definition}{Definition}{Definitions}
\crefname{assumption}{Assumption}{Assumptions}
\crefname{fact}{Fact}{Facts}
\crefname{example}{Example}{Examples}
\crefname{algorithm}{Algorithm}{Algorithms}
\crefname{remark}{Remark}{Remarks}
\crefname{note}{Note}{Notes}
\crefname{notation}{Notation}{Notations}
\crefname{case}{Case}{Cases}
\crefname{exercise}{Exercise}{Exercises}
\crefname{question}{Question}{Questions}
\crefname{claim}{Claim}{Claims}
\crefname{enumi}{}{}

\numberwithin{equation}{section}

\usepackage{amsmath,xparse}
\makeatletter
\NewDocumentCommand{\lplabel}{o m}{%
	\makebox[0pt][r]{#2\hspace*{2em}}%
	\IfNoValueF{#1}
	{\def\@currentlabel{#2}\ltx@label{#1}}
}
\makeatother

\theoremstyle{plain}
\newtheorem{theorem}{Theorem}[section]
\newtheorem{corollary}{Corollary}[section]
\newtheorem{fact}{Fact}[section]
\newtheorem{lemma}{Lemma}[section]
\newtheorem{proposition}{Proposition}[section]
\theoremstyle{definition}

\newtheorem{assumption}{Assumption}[section]
\newtheorem{remark}{Remark}[section]

\usepackage{enumitem}

\newcommand{\minimize}{\ensuremath{\operatorname{minimize}}}
\newcommand{\maximize}{\ensuremath{\operatorname{maximize}}}

\newcommand{\weakly}{\ensuremath{{\;\operatorname{\rightharpoonup}\;}}}

\newcommand{\dom}{\ensuremath{\operatorname{dom}}}
\newcommand{\Id}{\ensuremath{\operatorname{Id}}}
\newcommand{\Pro}{\ensuremath{\operatorname{P}}}
\newcommand{\Prox}{\ensuremath{\operatorname{Prox}}}
\newcommand{\argmin}{\mathop{\rm argmin}}

\providecommand{\abs}[1]{\left|#1\right|}
\providecommand{\norm}[1]{\left\lVert#1\right\rVert}
\providecommand{\innp}[1]{\left\langle#1\right\rangle}

\begin{document}

\title{Alternating Proximity Mapping Method for Strongly Convex-Strongly Concave 
Saddle-Point Problems}

\author{
	 Hui Ouyang\thanks{Department of Electrical Engineering, Stanford University.
		E-mail: \href{mailto:houyang@stanford.edu}{\texttt{houyang@stanford.edu}}.}
}

\date{October 30, 2023}

\maketitle

\begin{abstract}
This is a continuation  of our previous work entitled
\enquote{Alternating Proximity Mapping Method for Convex-Concave Saddle-Point 
Problems}, in which we 
proposed the alternating proximal mapping method 
and showed convergence results on the sequence of our iterates,
the sequence of averages of our iterates, and the sequence of function values
evaluated at the averages of the iterates
for solving convex-concave saddle-point problems.

In this work, we extend the application of the alternating proximal mapping method
to solve strongly convex-strongly concave saddle-point problems.
We demonstrate two sets of sufficient conditions and also their simplified versions, 
which guarantee 
the linear convergence of the sequence of iterates towards a desired saddle-point.
Additionally, we provide two sets of sufficient conditions, 
along with their simplified versions,
that ensure the linear convergence of the sequence of function values evaluated at  
the convex combinations of iteration points to the desired function value of a 
saddle-point.
\end{abstract}

{\small
	\noindent
	{\bfseries 2020 Mathematics Subject Classification:}
	{
		Primary 90C25, 47J25;  
		Secondary 47H05,  90C30.
	}

\noindent{\bfseries Keywords:}
Convex-Concave Saddle-Point Problems, 
Proximity Mapping,  
Strongly Convexity, 
Linear Convergence
}


 \section{Introduction}
 In the whole work,  
 $\mathcal{H}_{1}$ and $\mathcal{H}_{2}$ are Hilbert spaces, and   
 the Hilbert direct sum $\mathcal{H}_{1} \times \mathcal{H}_{2}$ of  $\mathcal{H}_{1}$ 
 and $\mathcal{H}_{2}$ is equipped with the inner product 
\[ 
\left(\forall (x, y) \in \mathcal{H}_{1} \times \mathcal{H}_{2}\right) 
\left(\forall (u,v) \in \mathcal{H}_{1} \times \mathcal{H}_{2}\right)
 \quad \innp{(x,y), (u,v)} = 
\innp{x,u} +\innp{y,v}
\]
 and the induced norm
\[ 
\left(\forall (x, y) \in \mathcal{H}_{1} \times \mathcal{H}_{2}\right) 
 \quad 
 \norm{(x, y) }^{2} =\innp{(x,y), (x,y)}=\innp{x,x}  +\innp{y,y} 
 =\norm{x}^{2} + \norm{y}^{2}. 
\]
Throughout this work,  
$K \in \mathcal{B}(\mathcal{H}_{1} , \mathcal{H}_{2} )$, i.e., $K : \mathcal{H}_{1} \to 
\mathcal{H}_{2}$ is a continuous linear operator with the operator norm
\begin{align} \label{eq:definenormK}
	\norm{K} := \sup \{ \norm{Kx} ~:~ x \in \mathcal{H}_{1} \text{ with } \norm{x} \leq 1 \},
\end{align}
and $g: \mathcal{H}_{1} \to \mathbf{R} \cup \{ +\infty\}$ and 
$h: \mathcal{H}_{2} \to  \mathbf{R} \cup \{ +\infty\}$ are proper and lower 
semicontinuous, and also satisfy the following  \cref{assumption:basic}.
\begin{assumption} \label{assumption:basic}
	Let $\mu$ and $\nu$ be in $\mathbf{R}_{++}$.
	\begin{itemize}
		\item $g$ is strongly convex with a modulus $\mu >0$, i.e., 
		\begin{align}\label{eq:gmu}
			(\forall \bar{x} \in \dom \partial g)(\forall u \in \partial g(\bar{x})) (\forall x \in 
			\mathcal{H}_{1}) \quad \innp{u, x-\bar{x}} + \frac{\mu}{2} \norm{x -\bar{x}}^{2}
			+g(\bar{x}) \leq g(x).
		\end{align}
		\item $h$ is strongly convex with a modulus $\nu >0$, i.e., 
		\begin{align}\label{eq:hnu}
			(\forall \bar{y} \in \dom \partial h)(\forall v \in \partial h(\bar{y})) (\forall y \in 
			\mathcal{H}_{2}) \quad \innp{v, y-\bar{y}} + \frac{\nu}{2} \norm{y -\bar{y}}^{2}
			+h(\bar{y}) \leq h(y).
		\end{align}
	\end{itemize}
\end{assumption}
Henceforth, unless stated otherwise,
$f : \mathcal{H}_{1} \times \mathcal{H}_{2} \to \mathbf{R} \cup \{ - \infty, + \infty \}$
is defined as 
\begin{align}  \label{eq:fspecialdefine}
	\left(\forall (x,y) \in \mathcal{H}_{1} \times \mathcal{H}_{2}\right) \quad 	f(x,y) 
	=\innp{Kx, y} + g(x)  -h(y).
\end{align}

Our goal in this work is to solve the following \emph{strongly convex-strongly concave 
saddle-point problem}
\begin{align}\label{eq:problem}
\underset{x \in  \mathcal{H}_{1} }{\minimize}
~\,\underset{y \in  \mathcal{H}_{2} }{\maximize}  ~f(x,y).
\end{align}
We assume  that  the solution set of \cref{eq:problem} is nonempty, 
 that is, there exists a \emph{saddle-point} $(x^{*}, y^{*}) \in \mathcal{H}_{1} \times 
 \mathcal{H}_{2}$ of $f$
 satisfying 
 \begin{align}\label{eq:solution}
 \left(\forall x \in \mathcal{H}_{1}\right) \left(\forall y \in \mathcal{H}_{2}\right) \quad 
 f(x^{*}, y)  \leq  f(x^{*}, y^{*})  \leq f(x, y^{*}).
 \end{align}
Our tool in this work is the alternating proximal mapping method proposed 
in \cite{Oy2023Proximity}, which will be presented in  
\cref{eq:fact:algorithmsymplity} below. 
To solve the problem \cref{eq:problem},
we study the linear convergence of both the sequence 
$\left((x^{k},y^{k})\right)_{k \in \mathbf{N}}$ and 
$\left(f\left( \hat{x}_{k},\hat{y}_{k} \right)\right)_{k \in \mathbf{N}}$ 
to $(x^{*},y^{*})$ and $f(x^{*},y^{*})$, respectively, 
where $((x^{k},y^{k}))_{k \in \mathbf{N}}$ is generated by 
the alternating proximal mapping method, 
and 
$(\forall k \in \mathbf{N})$ 
 $\hat{x}_{k} = \frac{1}{\sum^{k}_{j=0}\frac{1}{\xi^{j}}} \sum^{k}_{i=0} \frac{1}{\xi^{i}} 
x^{i+1}$    
and 
$\hat{y}_{k} = \frac{1}{\sum^{k}_{j=0}\frac{1}{\xi^{j}}}  \sum^{k}_{i=0} \frac{1}{\xi^{i}} 
y^{i+1}$, with $\xi \in \mathbf{R}_{++}$.
Below, we present our main results.
\begin{enumerate}
	\item 
	We provide sufficient conditions for the linear convergence 
	$\left((x^{k},y^{k})\right)_{k \in \mathbf{N}}$ to $(x^{*},y^{*})$
	in \Cref{theorem:linearconverg:K,theorem:linearconverg:Ksquare}. 
	Furthermore, given known parameters $\mu >0$, $\nu >0$, and 
	$\norm{K} \geq 0$ from a specific problem, we consider the involved parameters 
	$(\tau_{k})_{k \in \mathbf{N}}$,  $(\sigma_{k})_{k \in \mathbf{N}}$,  $(\alpha_{k})_{k 
	\in \mathbf{N}}$, and $(\beta_{k})_{k \in \mathbf{N}}$ 
in the iterate scheme \cref{eq:fact:algorithmsymplity} 
being constants
$\tau$, $\sigma$, $\alpha$, and $\beta$, respectively, 
	and show clear and simplified sufficient conditions on the choices of  
	the involved parameters $\tau$, $\sigma$, $\alpha$, and $\beta$
	for the linear convergence of
	$\left((x^{k},y^{k})\right)_{k \in \mathbf{N}}$ to $(x^{*},y^{*})$ in 
	\Cref{corollary:linearconverg:K,corollary:linearconverg:Ksquare}.
	
	\item  We present two set of sufficient conditions for the linear convergence of 
	$\left(f\left( \hat{x}_{k},\hat{y}_{k} \right)\right)_{k \in \mathbf{N}}$ 
	to $f(x^{*},y^{*})$ in
	\cref{theorem:linearconvergeineq}.
	Moreover, after knowing the parameters $\mu >0$, $\nu >0$, and 
	$\norm{K} \geq 0$ from a specific problem, under two set of restrictions 
	on the parameters 
	$(\forall k \in \mathbf{N})$ $\tau_{k} \equiv \tau$, $\sigma_{k} \equiv \sigma$, 
	$\alpha_{k} \equiv \alpha$, and $\beta_{k} \equiv \beta$ 
	involved in the iterate scheme \cref{eq:fact:algorithmsymplity}, 
	 we deduce the linear 
	convergence of $\left(f\left( \hat{x}_{k},\hat{y}_{k} \right)\right)_{k \in \mathbf{N}}$ 
	to $f(x^{*},y^{*})$  in
	\cref{corollary:linearconvergeineq}.
\end{enumerate}
In fact, based on our \cref{remark:ineqaulities} below, 
by using almost the same techniques applied in the proofs of 
\Cref{theorem:linearconverg:K,theorem:linearconverg:Ksquare,theorem:linearconvergeineq},
and employing the inequalities presented in \cref{remark:ineqaulities},
we can provide many other sufficient conditions for the linear convergence of
$\left((x^{k},y^{k})\right)_{k \in \mathbf{N}}$ to $(x^{*},y^{*})$,
and also the linear convergence of 
 $\left(f\left( \hat{x}_{k},\hat{y}_{k} \right)\right)_{k \in \mathbf{N}}$ 
 to $f(x^{*},y^{*})$.

\subsection{Related work}
This work is a subsequent work of \cite{Oy2023Proximity}, in which we proposed 
the  alternating proximal mapping method 
and  studied the (weak) convergence 
of the sequence of iterations
$\left((x^{k},y^{k})\right)_{k \in \mathbf{N}}$,
$\left( \left( \hat{x}_{k},\hat{y}_{k} \right)  \right)_{k \in \mathbf{N}}$, and 
$\left(f\left( \hat{x}_{k},\hat{y}_{k} \right)\right)_{k \in \mathbf{N}}$,
where   $\left((x^{k},y^{k})\right)_{k \in \mathbf{N}}$ is generated by 
the alternating proximal mapping method 
and $(\forall k \in \mathbf{N})$ $\hat{x}_{k} = \frac{1}{k+1} \sum^{k}_{i=0}  x^{i+1}$    
 and 
 $\hat{y}_{k} = \frac{1}{k+1}  \sum^{k}_{i=0}  y^{i+1}$.
In particular, inspired by \cite[Section~3]{ChambollePock2011} by Chambolle and Pock, 
we worked theoretically and numerically on 
the  alternating proximal mapping method,  and showed the convergence results 
on  $\left((x^{k},y^{k})\right)_{k \in \mathbf{N}} \to (x^{*},y^{*})$,
$\left( \left( \hat{x}_{k},\hat{y}_{k} \right)  \right)_{k \in \mathbf{N}} \weakly 
(x^{*},y^{*})$, and 
$\left(f\left( \hat{x}_{k},\hat{y}_{k} \right)\right)_{k \in \mathbf{N}} \to f(x^{*},y^{*})$ 
without any strong convexity/concavity assumption in \cite{Oy2023Proximity}.

Although in \cite[Section~5]{ChambollePock2011}, Chambolle and Pock
considered the 
acceleration of the first-order  primal-dual algorithm worked therein 
under some strongly convexity assumptions, 
and showed a sufficient condition for the linear convergence of the sequence of 
iterations generated 
by their algorithm. 
In this work, we not only explicitly present two set of sufficient conditions for the linear 
convergence of 
the sequence   $\left((x^{k},y^{k})\right)_{k \in \mathbf{N}}$ generated by our iterate 
scheme \cref{eq:fact:algorithmsymplity}
to $(x^{*},y^{*})$, 
but also 
provide two set of sufficient conditions for the linear convergence of the sequence 
$\left(f\left( \hat{x}_{k},\hat{y}_{k} \right)\right)_{k \in \mathbf{N}}$ 
to $f(x^{*},y^{*})$, where $(\forall k \in \mathbf{N})$ 
$\hat{x}_{k} = \frac{1}{\sum^{k}_{j=0}\frac{1}{\xi^{j}}} \sum^{k}_{i=0} \frac{1}{\xi^{i}} 
x^{i+1}$    
and 
$\hat{y}_{k} = \frac{1}{\sum^{k}_{j=0}\frac{1}{\xi^{j}}}  \sum^{k}_{i=0} \frac{1}{\xi^{i}} 
y^{i+1}$, with $\xi \in \mathbf{R}_{++}$.

In \cite{Oy2023ccspp}, we proposed  the alternating proximal point 
algorithm with gradient descent and ascent steps 
 for solving the convex-concave saddle-point problem \cref{eq:problem}
with $(\forall (x,y) \in \mathcal{H}_{1} \times \mathcal{H}_{2} )$ 
$f(x,y) =f_{1}(x) +\Phi (x,y) -f_{2}(y)$, 
where $f_{1}: \mathcal{H}_{1} \to (-\infty, +\infty]$ 
and $f_{2} : \mathcal{H}_{2} \to (-\infty, +\infty]$ 
are (strongly) convex and lower semicontinuous,  
$(\forall y \in \dom f_{2})$ $\Phi(\cdot, y) $ is convex and Fr\'echet 
differentiable, and $(\forall x \in \dom f_{1})$ $\Phi(x,\cdot) $ is concave and  Fr\'echet 
differentiable.
Although, similarly with this work, we presented weak and linearly convergence of 
the sequence of iterations, and  showed
the convergence and linearly convergence of function values evaluated at  convex 
combinations of iteration points
under convex and strongly convex assumptions, respectively,
in \cite{Oy2023ccspp}.
Note that the convex-concave saddle-point problem considered in this work is 
only a special case of the one worked in \cite{Oy2023ccspp}, and that the 
alternating proximal point algorithm with gradient descent and ascent steps worked in
\cite{Oy2023ccspp} is different from the aternating proximity mapping method
 considered in this work. 
Therefore, although goals and results in this work and \cite{Oy2023ccspp} are similar, 
assumptions, proof techniques,  and application areas are different. 

In  \cite{Oy2023fccspp}, we worked on  
the Optimistic Gradient Ascent-Proximal Point Algorithm (OGAProx)
introduced  by Bo{\c{t}}, Csetnek, and Sedlmayer
in the article \cite{BoctCsetnekSedlmayer2022accelerated}
 for solving the problem \cref{eq:problem}
with $\left(\forall (x,y)  \in \mathcal{H}_{1} \times \mathcal{H}_{2} \right)$
$f(x,y) :=\Phi(x,y) -g(y)$, where 
$(\forall y \in \dom g)$ $\Phi(\cdot, y) : \mathcal{H}_{1} \to \mathbf{R} \cup 
\{+\infty\}$ is proper, convex, and lower semicontinuous,
$\left( \forall x \in \Pro_{\mathcal{H}_{1}} (\dom \Phi) \right)$ 
$\Phi(x, \cdot) : \mathcal{H}_{2} \to \mathbf{R}$ is convex and Fr\'echet 
differentiable with 
$ \Pro_{\mathcal{H}_{1}} (\dom \Phi):= \{ u \in \mathcal{H}_{1} ~:~ \exists y \in 
\mathcal{H}_{2} \text{ such that } (u,y) \in \dom \Phi\}$,
and $g : \mathcal{H}_{2} \to (-\infty, +\infty]$ is proper, lower semicontinuous, and 
(strongly) convex. 
The authors in  \cite{BoctCsetnekSedlmayer2022accelerated} provided the
 (weak/strong/linear) convergence of the sequence of iterations generated by the  
 OGAProx.  
 Under the same assumptions used by Bo{\c{t}} et al.\,in 
 \cite{BoctCsetnekSedlmayer2022accelerated} for proving the
 convergence of the minimax gap function,
 we established in \cite{Oy2023fccspp} convergence results 
of function values evaluated at the ergodic sequences generated by the OGAProx
 with convergence rates of order 
 $\mathcal{O}\left(\frac{1}{k}\right)$, $\mathcal{O}\left(\frac{1}{k^{2}}\right)$,
 and $\mathcal{O}\left(\theta^{k}\right)$ with $\theta \in (0,1)$ 
when the associated convex-concave coupling function $\Phi$ is convex-concave, 
 convex-strongly concave, and strongly convex-strongly concave, respectively. 
Clearly, the problems and algorithms studied in this work are different from those 
considered in \cite{BoctCsetnekSedlmayer2022accelerated} and \cite{Oy2023fccspp}.

\subsection{Notation} 
We  point out some notation used in this work.  
$\mathbf{R}$, $\mathbf{R}_{+}$, $\mathbf{R}_{++}$, and $\mathbf{N}$  are the set of 
all 
real numbers, the set of all nonnegative real numbers, the set of all positive real 
numbers, 
and the set of all nonnegative integers, respectively.  
$\mathcal{B}(\mathcal{H}_{1} , \mathcal{H}_{2} ) 
:= \{ T: \mathcal{H}_{1} \to 
\mathcal{H}_{2} ~:~ T \text{ is linear and continuous} \}$
is the set of all linear and continuous operators from $\mathcal{H}_{1}$ to  
$\mathcal{H}_{2}$.
Let $T \in \mathcal{B}(\mathcal{H}_{1} , \mathcal{H}_{2} ) $. The \emph{adjoint}
of $T$ is the unique operator $T^{*} \in \mathcal{B}(\mathcal{H}_{2} , \mathcal{H}_{1} )$
that satisfies 
\begin{align*}
	(\forall x \in \mathcal{H}_{1})(\forall y \in \mathcal{H}_{2}) \quad
	\innp{Tx,y} =\innp{x, T^{*}y}.
\end{align*}

Let $\mathcal{H}$ be a Hilbert space. 
$\Id :\mathcal{H}\to \mathcal{H}: x \mapsto x$ is the \emph{identity operator}.
$2^{\mathcal{H}}$ is the \emph{power set of $\mathcal{H}$}, i.e., 
the family of all subsets of $\mathcal{H}$.
Let $g: \mathcal{H} \to \left(  -\infty, 
+\infty\right]$ be proper. The \emph{subdifferential of $g$} is the set-valued operator 
\[  
\partial g: \mathcal{H} \to 2^{\mathcal{H}}: x \mapsto \{ u \in \mathcal{H} 
~:~ (\forall y \in \mathcal{H})~ \innp{u, y-x} + f(x) \leq f(y) \}.
\]
Suppose that $g$ is proper, convex, and lower semicontinuous. 
The \emph{proximity mapping $\Prox_{g}$ of $g$} is defined by
\[ 
\Prox_{g}: \mathcal{H} \to \mathcal{H} : x \mapsto \argmin_{y \in \mathcal{H}} \left(  g(y) 
+\frac{1}{2} \norm{x-y}^{2} \right).
\]

\subsection{Outline}
In \cref{section:Preliminaries}, we collect results that will simplify our proofs of
main results in the following sections.
Moreover, in this section, we also provide analysis for much more sufficient conditions 
of our linear convergence results  given in both  
\Cref{section:linearconvergeIterations,section:linearconvergf} except for the conditions
written clearly in \Cref{section:linearconvergeIterations,section:linearconvergf} 
(for details see \cref{remark:ineqaulities}). 
In \cref{section:linearconvergeIterations}, we present two set of sufficient conditions 
for the sequence of iterations generated by the  alternating proximal mapping method 
for converging linearly to the desired saddle-point. 
Based on these results, we provide simplified and clearly practical assumptions for the 
linear convergence of the sequence of iterations. (See  
\Cref{theorem:linearconverg:K,theorem:linearconverg:Ksquare}, and 
\Cref{corollary:linearconverg:K,corollary:linearconverg:Ksquare}
for details.)
In \cref{section:linearconvergf}, we  demonstrate two set of sufficient conditions for 
the linear convergence of  function values evaluated at  convex 
combinations of iterations generated by the alternating proximal mapping method.
According to these results, simplified and clearly practical versions of 
assumptions on the involved parameters in constants are established. (See
\Cref{theorem:linearconvergeineq,corollary:linearconvergeineq} for details.)

\section{Preliminaries} \label{section:Preliminaries}

We first show some basic results. 
\begin{lemma}\label{lemma:partialstronglyconvex}
	Let $(x^{*},y^{*})$ be a saddle-point of $f$ defined in \cref{eq:fspecialdefine}. We 
	have the following assertions.
	\begin{enumerate}
		\item \label{lemma:partialstronglyconvex:K}
		$-K^{*}y^{*} \in \partial g(x^{*})$ and $Kx^{*} \in \partial h (y^{*})$.
		\item \label{lemma:partialstronglyconvex:g}
		$(\forall x \in \mathcal{H}_{1})$ 
		$\innp{-K^{*}y^{*}, x-x^{*}} + \frac{\mu}{2} \norm{x -x^{*}}^{2}+g(x^{*}) \leq g(x)$.
		\item \label{lemma:partialstronglyconvex:h}
		$(\forall y \in \mathcal{H}_{2})$ 
		$ \innp{Kx^{*}, y-y^{*}} + \frac{\nu}{2} \norm{y -y^{*}}^{2}+h(y^{*}) \leq h(y)$.
		\item \label{lemma:partialstronglyconvex:f}
		For every $x \in \mathcal{H}_{1}$ and for every $y \in \mathcal{H}_{2}$,
		\begin{align*}
			f(x,y^{*}) -f(x^{*},y) 
			&= \innp{Kx, y^{*}} + g(x)  -h(y^{*}) - \left(\innp{Kx^{*}, y} + g(x^{*})  
			-h(y)\right)\\
			&\geq \frac{\mu}{2} \norm{x -x^{*}}^{2} + \frac{\nu}{2} \norm{y -y^{*}}^{2}.
		\end{align*}
	\end{enumerate}
\end{lemma}

\begin{proof}
	\cref{lemma:partialstronglyconvex:K}:
	Because $(x^{*},y^{*})$ is a saddle-point, 
	applying \cite[Fact~1.1]{Oy2023subgradient} 
	and bearing \cref{eq:fspecialdefine} in mind, we observe that
	\begin{align*}
		0 \in \partial_{x}f(x^{*},y^{*})=K^{*}y^{*} +\partial g(x^{*}) \quad \text{and} \quad
		0 \in \partial_{y}\left(-f(x^{*},y^{*}) \right)=- Kx^{*}+\partial h(y^{*}),
	\end{align*}
	which immediately implies our required result. 
	
	\cref{lemma:partialstronglyconvex:g}$\&$\cref{lemma:partialstronglyconvex:h}:
	The desired results are clear from 
	our assumptions stated in \cref{eq:gmu} and \cref{eq:hnu},
	and our result obtained in \cref{lemma:partialstronglyconvex:K}.
	
	\cref{lemma:partialstronglyconvex:f}: 
	Let $(x,y)$ be in  $\mathcal{H}_{1} \times \mathcal{H}_{2}$.
	Add the two inequalities presented in 
	\cref{lemma:partialstronglyconvex:g} and \cref{lemma:partialstronglyconvex:h} to
	deduce that
	\begin{align*}
		&\frac{\mu}{2} \norm{x -x^{*}}^{2} + \frac{\nu}{2} \norm{y -y^{*}}^{2}\\
		\leq & g(x) - g(x^{*}) + \innp{K^{*}y^{*}, x-x^{*}}  
		+h(y) -h(y^{*}) -\innp{Kx^{*}, y-y^{*}}  \\
		= &g(x) - g(x^{*}) + h(y) -h(y^{*}) + \innp{y^{*},Kx} -\innp{y^{*},Kx^{*}}
		-\innp{Kx^{*}, y} +\innp{Kx^{*}, y^{*}}\\
		=& \innp{Kx, y^{*}} + g(x)  -h(y^{*}) - \left(\innp{Kx^{*}, y} + g(x^{*})  -h(y)\right)\\
		=& f(x,y^{*}) -f(x^{*},y).
	\end{align*}
	where in the last equality we use \cref{eq:fspecialdefine}.
	
\end{proof}


\begin{fact} \label{fact:fxkykx*y*}
	{\rm  \cite[Lemma~2.5]{Oy2023ccspp}}
	Let $f: \mathcal{H}_{1} \times \mathcal{H}_{2} \to \mathbf{R} \cup \{-\infty, +\infty\}$ 
	satisfy that $(\forall y \in \mathcal{H}_{2} )$ $f(\cdot, y)$ is convex 
	and $(\forall x \in \mathcal{H}_{1})$ $f(x,\cdot)$ is concave. 
	Let $(x^{*},y^{*})$ be a saddle-point of $f$, let $((x^{k},y^{k}))_{k \in \mathbf{N}}$ 
	be in $\mathcal{H}_{1}\times \mathcal{H}_{2}$, and let $(\forall k \in 
	\mathbf{N})$ $t_{k} \in \mathbf{R}_{+}$ with $t_{0} \in \mathbf{R}_{++}$.
	Set 
	\begin{align*} 
		(\forall k \in \mathbf{N}) \quad 
		\hat{x}_{k} := \frac{1}{\sum^{k}_{i=0}t_{i}} 
		\sum^{k}_{j=0}t_{j}x^{j+1} 
		\quad \text{and} \quad 
		\hat{y}_{k} := \frac{1}{\sum^{k}_{i=0}t_{i}} 
		\sum^{k}_{j=0}t_{j}y^{j+1}.
	\end{align*}
	Then we have that for every $k \in \mathbf{N}$,
	\begin{subequations}
		\begin{align*}
			&f(\hat{x}_{k},\hat{y}_{k}) -f(x^{*},y^{*}) 
			\leq \frac{1}{\sum^{k}_{i=0}t_{i}} \sum^{k}_{j=0}t_{j} \left( f(x^{j+1}, \hat{y}_{k})  
			-f(x^{*},y^{j+1}) \right);\\
			&f(x^{*},y^{*}) -f(\hat{x}_{k},\hat{y}_{k})
			\leq \frac{1}{\sum^{k}_{i=0}t_{i}} \sum^{k}_{j=0}t_{j} \left( f(x^{j+1},y^{*})  
			-f(\hat{x}_{k},y^{j+1}) \right).
		\end{align*}
	\end{subequations}
\end{fact}

\subsection{Alternating Proximal Mapping Method }
From now on, we assume that 
\[ 
(\tau_{k})_{k \in \mathbf{N} \cup \{-1\}} \text{ and } 
(\sigma_{k})_{k \in \mathbf{N} \cup \{-1\}} 
\text{	are in } \mathbf{R}_{++}, \text{and } 
(\alpha_{k})_{k \in \mathbf{N} \cup \{-1\}} \text{ and } 
(\beta_{k})_{k \in \mathbf{N} \cup \{-1\}}
\text{ are in } \mathbf{R}_{+}.
\]

Let $(x^{0}, y^{0}) \in X \times Y$. 
Set $\bar{x}^{0}=x^{0}$ and $\bar{y}^{0} =y^{0}$ 
$($i.e., $x^{-1} =x^{0}$ and $y^{-1} =y^{0}$$)$.  
Recall from \cite[Lemma~4.1]{Oy2023Proximity} that 
the \emph{alternating proximal mapping method for solving the problem} 
\cref{eq:problem}
with $f$  defined as \cref{eq:fspecialdefine} is the following: 
for every $k \in \mathbf{N}$, 
\begin{subequations}\label{eq:fact:algorithmsymplity}
	\begin{align}
		&y^{k+1} =  \Prox_{ \sigma_{k} h} (\sigma_{k} K\bar{x}^{k}  +  y^{k} )
		=(\Id + \sigma_{k} \partial  h )^{-1} (\sigma_{k} K\bar{x}^{k}  +  y^{k} ); 
		\label{eq:fact:algorithmsymplity:yk}\\
		& \bar{y}^{k+1} = y^{k+1} + \beta_{k} (y^{k+1} -y^{k}); 
		\label{eq:fact:algorithmsymplity:bary}\\
		&x^{k+1} = \Prox_{\tau_{k}  g} ( -\tau_{k} K^{*} \bar{y}^{k+1}+ x^{k} ) 
		=(\Id + \tau_{k}  \partial  g  )^{-1} ( -\tau_{k} K^{*} \bar{y}^{k+1}+ x^{k} ); 
		\label{eq:fact:algorithmsymplity:xk}\\
		& \bar{x}^{k+1} = x^{k+1} + \alpha_{k} (x^{k+1} -x^{k}). 
		\label{eq:fact:algorithmsymplity:barx} 
	\end{align}
\end{subequations}

In the rest of this work, 
$\left(\left(x^{k},y^{k}\right)\right)_{k \in \mathbf{N}}$ is generated by the
iterate scheme \cref{eq:fact:algorithmsymplity}.

The following result will facilitate our proofs in some main results, 
\Cref{corollary:linearconverg:K,corollary:linearconverg:Ksquare,corollary:linearconvergeineq},
 in this work below.
\begin{lemma} \label{lemma:infsupconstants}
	Recall that $\mu>0$, $\nu>0$, $\norm{K} \geq 0$ are known from 
	\cref{assumption:basic} and the function $f$ defined in \cref{eq:fspecialdefine}. 		
	Suppose that 
	\begin{align}  \label{eq:lemma:infsupconstants:constant}
		(\forall k \in \mathbf{N} \cup \{-1\}) \quad \tau_{k} \equiv \tau \in \mathbf{R}_{++}, 
		\sigma_{k} \equiv  \sigma \in \mathbf{R}_{++}, 
		\alpha_{k} \equiv \alpha \in \mathbf{R}_{++}, \text{ and } 
		\beta_{k} \equiv \beta \in \mathbf{R}_{+}.
	\end{align}
		Let $\zeta $ be in $\mathbf{R}_{++}$.
		The following statements hold. 
	
	\begin{enumerate}
		\item \label{lemma:infsupconstants:munu} There exist $\tau >0$ and $\sigma >0$   
		small enough such that 
		$\norm{K}^{2} <\frac{\zeta \mu +\frac{1}{\tau}}{\sigma}$ and 
		$\norm{K}^{2} < \frac{\zeta \nu +\frac{1}{\sigma}}{\tau}$. 
		\item  \label{lemma:infsupconstants:alpha} Suppose $\norm{K}^{2} <\frac{\zeta \mu 
		+\frac{1}{\tau}}{\sigma}$ and 
		$\norm{K}^{2} < \frac{\zeta \nu +\frac{1}{\sigma}}{\tau}$. Then
		 there exists $\alpha$ be in $\left( \frac{1}{ \zeta \mu \tau +1} , 1\right) \subseteq 
		(0,1) $ such that $\alpha \geq 
		\frac{1}{\zeta  \nu \sigma +1}$, and $\alpha \tau \sigma \norm{K}^{2} <1$.
		Consequently, in this case, we have that $\alpha \sup_{k \in \mathbf{N}} \tau_{k} 
		\sup_{k \in 
		\mathbf{N}} \sigma_{k} \norm{K}^{2}< 1$.
	
		\item \label{lemma:infsupconstants:beta} 
		Suppose that $\alpha > \frac{1}{ \zeta \mu \tau +1} $  and $\alpha \tau 
		\sigma \norm{K}^{2} <1$. 
		Then there exists $\beta \in \mathbf{R}_{+}$  
		such that $\sigma \norm{K}^{2}\beta^{2} < \left(\zeta \mu +\frac{1}{\tau} 
		-\frac{1}{\alpha \tau}\right) \left(1-\alpha \tau \sigma \norm{K}^{2}\right)$.
		
		\item  \label{lemma:infsupconstants:nu}  
		Let $\eta_{2}$ be in $\mathbf{R}_{++}$.
		Suppose that 
		$\alpha \geq \frac{1}{\zeta  \nu \sigma +1}$. Then $(\forall i \in \{1,2\})$
		$
		\zeta \nu \geq \frac{1}{\alpha \inf_{k \in 
				\mathbf{N}}\sigma_{k}} -\frac{1}{\sup_{k \in \mathbf{N}}\sigma_{k}}  
		+ \eta_{2}\norm{K}^{i} \sup_{k \in \mathbf{N}}(\alpha -\alpha_{k})^{2}
		$
	 
			\item \label{lemma:infsupconstants:eta4}
			Suppose $\alpha \tau \sigma \norm{K}^{2} <1$. Then
		there exists $\eta_{4} \in \mathbf{R}_{++}$ such that 
		$\eta_{4} \geq \tau$ and $\alpha \sigma \eta_{4} \norm{K}^{2} < 1$. Moreover, 
		in this case, $\eta_{4} \geq \sup_{i \in \mathbf{N}} \tau_{i}$  
		and $\frac{1}{\sup_{i \in \mathbf{N}}\sigma_{i}} -\alpha \eta_{4} \norm{K}^{2} >0$.
		
		\item \label{lemma:infsupconstants:eta3} Suppose that   $\alpha \tau \sigma 
		\norm{K}^{2} <1$ and $\sigma 
		\norm{K}^{2}\beta^{2} < \left(\zeta \mu +\frac{1}{\tau} 
		-\frac{1}{\alpha \tau}\right) \left(1-\alpha \tau \sigma \norm{K}^{2}\right)$. 
		The following assertions hold. 
		\begin{enumerate}
			\item \label{lemma:infsupconstants:eta3:a} There exists $\eta_{3} \in 
			\mathbf{R}_{++}$ such that 
			$\eta_{3} > \frac{\sigma \norm{K}}{ 1 -\alpha \tau \sigma \norm{K}^{2}}$ and
			$\eta_{3}\norm{K} \beta^{2} <\zeta \mu +\frac{1}{\tau} 
			-\frac{1}{\alpha \tau}$. 
			
			Consequently, we have that $
			\zeta \mu\geq \frac{1}{\alpha \inf_{k \in \mathbf{N}} 
				\tau_{k}} -\frac{1}{\sup_{k \in \mathbf{N}}\tau_{k}}+  \eta_{3} 
				\norm{K} \sup_{k 
				\in \mathbf{N}}\beta_{k}^{2}
			$.
			\item \label{lemma:infsupconstants:eta3:b} There exists $\eta_{3} \in 
			\mathbf{R}_{++}$ such that 
			$\eta_{3} > \frac{\sigma}{ 1 -\alpha \tau \sigma \norm{K}^{2}}$ and
			$\eta_{3}\norm{K}^{2} \beta^{2} < \zeta  \mu +\frac{1}{\tau} 
			-\frac{1}{\alpha \tau}$. 
			
			Consequently, we have that 
			$
			\zeta \mu\geq \frac{1}{\alpha \inf_{k \in \mathbf{N}} 
				\tau_{k}} -\frac{1}{\sup_{k \in \mathbf{N}}\tau_{k}}+  \eta_{3} 
				\norm{K}^{2}\sup_{k 
				\in \mathbf{N}}\beta_{k}^{2}
			$.
		\end{enumerate}
		
		
		\item \label{lemma:infsupconstants:eta1a} Suppose $1 -\alpha \tau \sigma 
		\norm{K}^{2} >0$ and $\eta_{3} \in 
		\mathbf{R}_{++}$ satisfying 
		$\eta_{3} > \frac{\sigma \norm{K}}{ 1 -\alpha \tau \sigma \norm{K}^{2}}$. Then 
		there exists $\eta_{1} \in \mathbf{R}_{++}$ such that $\eta_{1} \geq \frac{\sigma 
		\norm{K}\eta_{3}}{\eta_{3} -\sigma \norm{K}}$ and
		$\alpha \tau \eta_{1}\norm{K} <1$. 
		Consequently, we have that $\norm{K} \left(\frac{1}{\eta_{1}} 
		+\frac{1}{\eta_{3}}\right) \leq \frac{1}{\sup_{i \in 
				\mathbf{N}}\sigma_{i}}$. 
		
		\item  \label{lemma:infsupconstants:eta1b}  Suppose $1 -\alpha \tau \sigma 
		\norm{K}^{2} >0$ and $\eta_{3} \in 
		\mathbf{R}_{++}$ satisfying 
		$\eta_{3} > \frac{\sigma  }{ 1 -\alpha \tau \sigma \norm{K}^{2}}$. Then 
		there exists $\eta_{1} \in \mathbf{R}_{++}$ such that $\eta_{1} \geq \frac{\sigma 
		 \eta_{3}}{\eta_{3} -\sigma \norm{K}}$ and
		$\alpha \tau \eta_{1}\norm{K}^{2} <1$. 
		Consequently, we have that $\frac{1}{\eta_{1}} +\frac{1}{\eta_{3}} \leq 
		\frac{1}{\sup_{k \in \mathbf{N}} \sigma_{k}}$.
		
		\item  \label{lemma:infsupconstants:eta2a} Suppose that $\alpha >0$ and $\eta_{1} 
		\in \mathbf{R}_{++}$ satisfying 
			$\alpha \tau \eta_{1}\norm{K} <1$.  Then there exists $\eta_{2} \in 
			\mathbf{R}_{++}$  such that 
			$\eta_{2} \geq \frac{\tau \norm{K}}{\alpha - \tau \eta_{1} \norm{K}\alpha^{2}}$. 	
		Consequently, we have that $\norm{K} \left( \eta_{1} \alpha^{2} 
		+\frac{1}{\eta_{2}}\right) \leq \frac{\alpha}{\sup_{i \in 
				\mathbf{N}}\tau_{i}}$.
		\item  \label{lemma:infsupconstants:eta2b}  Suppose that $\alpha >0$ and 
		$\eta_{1} \in \mathbf{R}_{++}$ satisfying 
			$\alpha \tau \eta_{1}\norm{K}^{2} <1$.  Then there exists $\eta_{2} \in 
			\mathbf{R}_{++}$  such that 
			$\eta_{2} \geq \frac{\tau }{\alpha - \tau \eta_{1} \norm{K}\alpha^{2}}$.
			Consequently, we have that
			$\eta_{1} \norm{K}^{2}\alpha +\frac{1}{\eta_{2}\alpha} \leq \frac{1}{\sup_{k \in 	
			\mathbf{N}}\tau_{k}}$.
	\end{enumerate}
\end{lemma}

\begin{proof}
	\cref{lemma:infsupconstants:munu}:  Clearly,
	$ \frac{\zeta \mu +\frac{1}{\tau}}{\sigma}$ and 
	$  \frac{ \zeta \nu +\frac{1}{\sigma}}{\tau}$ will both go to positive infinity when 
	$\tau >0$ and $\sigma >0$ go to $0^{+}$.
	Hence, 
	we can always take $\tau >0$ and $\sigma >0$ small enough
	to satisfy the requirements  $\norm{K}^{2} <\frac{\zeta \mu +\frac{1}{\tau}}{\sigma}$ 
	and 
	$\norm{K}^{2} < \frac{ \zeta \nu +\frac{1}{\sigma}}{\tau}$. 
	
	\cref{lemma:infsupconstants:alpha}:
	Note that if $\norm{K} =0$, then $\alpha \tau \sigma \norm{K}^{2} <1$ holds 
	automatically; if $\norm{K} \neq 0$, then
	\begin{subequations}\label{eq:lemma:infsupconstants:alpha}
		\begin{align}
			&\frac{1}{\zeta \mu \tau +1} <\frac{1}{\tau \sigma \norm{K}^{2}} \Leftrightarrow
			\norm{K}^{2}<\frac{ \zeta \mu \tau +1}{\tau \sigma} 
			=\frac{\zeta \mu +\frac{1}{\tau}}{\sigma};\\
			&\frac{1}{\zeta   \nu \sigma +1} <\frac{1}{\zeta \tau \sigma \norm{K}^{2}} 
			\Leftrightarrow
			\norm{K}^{2}<\frac{ \zeta \nu \sigma +1}{\tau \sigma} 
			=\frac{\zeta \nu +\frac{1}{\sigma}}{\tau}.
		\end{align}
	\end{subequations}
	Because $\mu>0$, $\nu >0$, $\tau >0$, and $\sigma >0$, due to 
	\cref{eq:lemma:infsupconstants:alpha} and our assumptions 
	$\norm{K}^{2} <\frac{\zeta \mu +\frac{1}{\tau}}{\sigma}$ and 
	$\norm{K}^{2} < \frac{\zeta \nu +\frac{1}{\sigma}}{\tau}$, 
	we are able to choose
	$\alpha \in \left( \frac{1}{ \mu \tau +1} , 1\right) \subseteq (0,1)$
	such that $\alpha \geq \frac{1}{\nu \sigma +1}$,
	and $\alpha \tau \sigma \norm{K}^{2} <1$. 
	Clearly, via \cref{eq:lemma:infsupconstants:constant}, 
	our assumption $\alpha \tau \sigma \norm{K}^{2} <1$ ensures 
	the required condition 
	$\alpha  \sup_{k \in \mathbf{N}} \tau_{k} \sup_{k \in \mathbf{N}} \sigma_{k} 
	\norm{K}^{2}< 
	1$.
	
	\cref{lemma:infsupconstants:beta}:  Notice that 
	$\zeta \mu +\frac{1}{\tau} 	-\frac{1}{\alpha \tau} >0 \Leftrightarrow
	\alpha > \frac{1}{\zeta \mu \tau +1}$. 
	So bearing our assumptions $\alpha > \frac{1}{\zeta \mu \tau +1}$ and 
	$\alpha \tau \sigma \norm{K}^{2} <1$ in mind, 
	we are able to take $\beta \in \mathbf{R}_{+}$ small enough such that 
	our assumption $\sigma \norm{K}^{2}\beta^{2} < \left(\zeta \mu +\frac{1}{\tau} 
	-\frac{1}{\alpha \tau}\right) \left(1-\alpha \tau \sigma \norm{K}^{2}\right)$.

	\cref{lemma:infsupconstants:nu}:
	Due to \cref{eq:lemma:infsupconstants:constant}, $ \sup_{k \in \mathbf{N}}(\alpha 
	-\alpha_{k})^{2}=0$.
	Because $\alpha \geq \frac{1}{\zeta \nu \sigma +1} \Leftrightarrow \zeta \nu \geq 
	\frac{1}{\alpha \sigma} -\frac{1}{\sigma}$,  
	our assumption $\alpha \geq \frac{1}{\zeta \nu \sigma +1}$ and 
	\cref{eq:lemma:infsupconstants:constant} 
	lead to the desired requirement 
	$
	\zeta \nu \geq \frac{1}{\alpha \inf_{k \in 
			\mathbf{N}}\sigma_{k}} -\frac{1}{\sup_{k \in \mathbf{N}}\sigma_{k}}  
	+ \eta_{2}\norm{K}^{i} \sup_{k \in \mathbf{N}}(\alpha -\alpha_{k})^{2}
	$. 
	
	\cref{lemma:infsupconstants:eta4}: 
	Because $\alpha \tau \sigma \norm{K}^{2} <1$ and $\tau >0$, we are able to 
	take $\eta_{4} >0$ such that $\eta_{4} \geq \tau$ and 
	$\alpha \eta_{4} \norm{K}^{2} < \frac{1}{\sigma}$, 
	which, combined  \cref{eq:lemma:infsupconstants:constant},
	satisfies the condition 
	$\eta_{4} \geq \sup_{i \in \mathbf{N}} \tau_{i}$,
	and $\frac{1}{\sup_{i \in \mathbf{N}}\sigma_{i}} -\alpha \eta_{4} \norm{K}^{2} >0$.
	
	\cref{lemma:infsupconstants:eta3:a}:
	Clearly, if $\norm{K}\beta \neq 0$, then
	\begin{align*}
		&\sigma \norm{K}^{2}\beta^{2} < \left(\zeta \mu +\frac{1}{\tau} 
		-\frac{1}{\alpha \tau}\right) \left(1-\alpha \tau \sigma \norm{K}^{2}\right)\\
		\Rightarrow &
		\sigma \norm{K} <
		\frac{\sigma \norm{K} }{1 -\alpha \tau \sigma \norm{K}^{2}} <  \frac{\zeta \mu 
			+\frac{1}{\tau} -\frac{1}{\alpha \tau}}{\norm{K} \beta^{2}}
	\end{align*}
	Based on our assumptions $\sigma \norm{K}^{2}\beta^{2} < \left(\zeta \mu 
	+\frac{1}{\tau} 
	-\frac{1}{\alpha \tau}\right) \left(1-\alpha \tau \sigma \norm{K}^{2}\right)$,
	we are able to take $\eta_{3} >0$ such that 
	$ \sigma \norm{K} \leq  \frac{\sigma  \norm{K}}{ 1 -\alpha \tau \sigma \norm{K}^{2}} < 
	\eta_{3}$, and
	$\eta_{3}\norm{K} \beta^{2} \leq \zeta  \mu +\frac{1}{\tau} 
	-\frac{1}{\alpha \tau}$.
	Hence, our assumption of $\eta_{3}$ together with  
	\cref{eq:lemma:infsupconstants:constant}
	guarantees the condition 
	$
	\zeta \mu \geq \frac{1}{\alpha \inf_{k \in \mathbf{N}} 
		\tau_{k}} -\frac{1}{\sup_{k \in \mathbf{N}}\tau_{k}}+  \eta_{3} \norm{K}^{2}\sup_{k 
		\in \mathbf{N}}\beta_{k}^{2}
	$.

	\cref{lemma:infsupconstants:eta3:b}:
	Note that $0<\alpha \tau \sigma \norm{K}^{2} <1$ leads to 
	$\sigma < \frac{\sigma}{1-\alpha \tau \sigma \norm{K}^{2} }$. 
	Because we assumed that $\norm{K}^{2} \beta^{2} \sigma < \left(\zeta \mu 
	+\frac{1}{\tau}-\frac{1}{\alpha 
		\tau} \right)\left( 1-\alpha \tau \sigma \norm{K}^{2}\right)
	$, 
	 we are able to take $\eta_{3} \in \mathbf{R}_{++}$ such that
	$\eta_{3} >\frac{\sigma}{1-\alpha \tau \sigma \norm{K}^{2}}$
	and $\norm{K}^{2}\beta^{2}\eta_{3} \leq \zeta \mu +\frac{1}{\tau} -\frac{1}{\alpha 
	\tau}$.
	Hence, via 	\cref{eq:lemma:infsupconstants:constant},  our assumption of $\eta_{3}$ 
	satisfies the required condition 
	$
	\zeta \mu \geq \frac{1}{\alpha \inf_{k \in \mathbf{N}} 
		\tau_{k}} -\frac{1}{\sup_{k \in \mathbf{N}}\tau_{k}}+  \eta_{3} \norm{K}^{2}\sup_{k 
		\in \mathbf{N}}\beta_{k}^{2}
	$. 
	 
	\cref{lemma:infsupconstants:eta1a}:
	It is easy to check that if $\norm{K} \neq 0$, then
	\begin{align*}
		\frac{\sigma \norm{K}}{1 -\alpha \tau \sigma \norm{K}^{2}} < \eta_{3} 
		\Leftrightarrow  
		\frac{\sigma \norm{K}\eta_{3}}{\eta_{3} -\sigma \norm{K}} < \frac{1}{\alpha \tau 
			\norm{K}}.	
	\end{align*}
	Based on our assumption
	$	\frac{\sigma \norm{K}}{1 -\alpha \tau \sigma \norm{K}^{2}} < \eta_{3}$, we are able 
	to 
	take $\eta_{1} >0$ such that 
	$\eta_{1} \geq \frac{\sigma \norm{K}\eta_{3}}{\eta_{3} -\sigma \norm{K}}$ and
	$\alpha \tau \eta_{1}\norm{K} <1$.	
	Note that under our assumption $\eta_{3} > \sigma \norm{K}$, we have that 
	$\eta_{1} \geq \frac{\sigma \norm{K}\eta_{3}}{\eta_{3} -\sigma \norm{K}}
	\Leftrightarrow
	\norm{K} \left(\frac{1}{\eta_{1}} +\frac{1}{\eta_{3}}\right) \leq \frac{1}{ \sigma}$. 
	So our assumption on $\eta_{1}$ yields  
	$\norm{K} \left(\frac{1}{\eta_{1}} +\frac{1}{\eta_{3}}\right) \leq \frac{1}{\sup_{i \in 
			\mathbf{N}}\sigma_{i}}$.

	\cref{lemma:infsupconstants:eta1b}:
	Clearly, if $\norm{K} \neq 0$, then
	$\eta_{3} >\frac{\sigma}{1-\alpha \tau \sigma \norm{K}^{2}} \Leftrightarrow
	\frac{\sigma \eta_{3}}{\eta_{3} -\sigma} <\frac{1}{\alpha \tau \norm{K}^{2}}$. 
	Furthermore, based on our assumptions  
	$\eta_{3} >\frac{\sigma}{1-\alpha \tau \sigma \norm{K}^{2}}>\sigma$,
  it is clear that
	$\frac{1}{\eta_{1}} +\frac{1}{\eta_{3}} \leq \frac{1}{\sigma} \Leftrightarrow
	\eta_{1} \geq \frac{\sigma \eta_{3}}{\eta_{3} -\sigma}$.
	Because we assumed 	
	$\eta_{3} >\frac{\sigma}{1-\alpha \tau \sigma \norm{K}^{2}}>\sigma$,
	we are able to take 
	$\eta_{1} \in \mathbf{R}_{++}$ such that 
	$\eta_{1} \geq \frac{\sigma \eta_{3}}{\eta_{3} -\sigma}$
	and $\alpha \tau \eta_{1}\norm{K}^{2} <1$,
	which, connecting with \cref{eq:lemma:infsupconstants:constant}, ensures
	$\frac{1}{\eta_{1}} +\frac{1}{\eta_{3}} \leq \frac{1}{\sup_{k \in \mathbf{N}} \sigma_{k}}$.

	\cref{lemma:infsupconstants:eta2a}:
	Using our assumptions $\alpha >0$ and $\alpha \tau \eta_{1}\norm{K} <1$, 
	we are able to take 
	$\eta_{2} >0$ such that $\eta_{2} \geq \frac{\tau \norm{K}}{\alpha - \tau \eta_{1} 
		\norm{K}\alpha^{2}}$. Notice that 
	$\eta_{2} \geq \frac{\tau \norm{K}}{\alpha - \tau \eta_{1} \norm{K}\alpha^{2}} 
	\Leftrightarrow
	\norm{K} \left( \eta_{1} \alpha^{2} +\frac{1}{\eta_{2}}\right) \leq \frac{\alpha}{\tau}$.
	Hence, via \cref{eq:lemma:infsupconstants:constant}, our assumption on $\eta_{2}$ 
	leads to  
	$\norm{K} \left( \eta_{1} \alpha^{2} +\frac{1}{\eta_{2}}\right) \leq \frac{\alpha}{\sup_{i 
	\in 	\mathbf{N}}\tau_{i}}$

	\cref{lemma:infsupconstants:eta2b}:
	Because $\alpha >0$ and $\alpha \tau \eta_{1} \norm{K}^{2} <1$,
 we are able to take 
	$\eta_{2} \in \mathbf{R}_{++}$ such that. 
	$\eta_{2} \geq \frac{\tau}{\alpha(1-\alpha \tau \eta_{1}\norm{K}^{2})}$.
	Because $\alpha \tau \eta_{1} \norm{K}^{2} <1$, it is easy to see that 
	$\eta_{2} \geq \frac{\tau}{\alpha(1-\alpha \tau \eta_{1}\norm{K}^{2})}
	\Leftrightarrow 
	\eta_{1} \norm{K}^{2}\alpha +\frac{1}{\eta_{2}\alpha} \leq \frac{1}{ \tau}$.
	Therefore, due to \cref{eq:lemma:infsupconstants:constant}, we conclude that
	$\eta_{1} \norm{K}^{2}\alpha +\frac{1}{\eta_{2}\alpha} \leq \frac{1}{\sup_{k \in 
			\mathbf{N}}\tau_{k}}$.
\end{proof}

\subsection{Properties of Iteration Points}
In this subsection, 
we provide useful properties of iteration points 
$\left(\left(x^{k},y^{k}\right)\right)_{k \in \mathbf{N}}$ 
generated by the alternating
proximal mapping method stated in \cref{eq:fact:algorithmsymplity} to facilitate our
main proofs later.
\begin{lemma} \label{lemma:partialinequalities}
Let $k \in \mathbf{N}$.	 
The following assertions hold. 
\begin{enumerate}
	\item \label{lemma:partialinequalities:innpx} 
	$(\forall x \in \mathcal{H}_{1})$ 
	$\innp{\frac{x^{k} -x^{k+1}}{\tau_{k}} -K^{*}\bar{y}^{k+1}, x -x^{k+1}} 
  +\frac{\mu}{2} \norm{x -x^{k+1}}^{2} + g(x^{k+1}) \leq g(x)$.
	\item \label{lemma:partialinequalities:innpy}
	$(\forall y \in \mathcal{H}_{2})$ 
	$\innp{\frac{y^{k}-y^{k+1}}{\sigma_{k}} +K\bar{x}^{k}, y - y^{k+1}} 
	+\frac{\nu}{2}\norm{y-y^{k+1}}^{2} +h(y^{k+1}) \leq h(y)$.
	\item  \label{lemma:partialinequalities:normx} 
	$(\forall x \in \mathcal{H}_{1})$ 
	$\frac{1}{2 \tau_{k}} \left( \norm{x^{k} -x^{k+1}}^{2} 
	+ \norm{x - x^{k+1}}^{2} -\norm{x-x^{k}}^{2} \right) 
	+\frac{\mu}{2} \norm{x -x^{k+1}}^{2}$ \\ 
	$  - 
	\innp{K^{*}\bar{y}^{k+1}, x -x^{k+1}}  +g(x^{k+1}) \leq g(x)$.
		\item \label{lemma:partialinequalities:normy} 
		$(\forall y \in \mathcal{H}_{2})$ 
	$\frac{1}{2 \sigma_{k}}\left( \norm{y^{k} -y^{k+1}}^{2} 
	+\norm{y-y^{k+1}}^{2} -\norm{y - y^{k}}^{2} \right) 
	+\frac{\nu}{2}\norm{y-y^{k+1}}^{2} \\
	+\innp{K\bar{x}^{k}, y - y^{k+1}} 
	+h(y^{k+1}) \leq h(y)$.
	\item  \label{lemma:partialinequalities:sum}
	For every $x \in \mathcal{H}_{1}$ and for every $y \in \mathcal{H}_{2}$,
	we have that
	\begin{align*}
		&\frac{1}{2 \tau_{k}} \norm{x -x^{k}}^{2} +\frac{1}{2\sigma_{k}} \norm{y 
		-y^{k}}^{2}\\
		\geq &\left(\frac{\mu}{2} +\frac{1}{2\tau_{k}}\right)\norm{x -x^{k+1}}^{2}+\left( 
		\frac{\nu}{2} +\frac{1}{2 \sigma_{k}}\right) \norm{y -y^{k+1}}^{2}
		+ \frac{1}{2\tau_{k}} \norm{x^{k} -x^{k+1}}^{2} 
		+ \frac{1}{2 \sigma_{k}} \norm{y^{k} -y^{k+1}}^{2}\\
		&+\innp{K \bar{x}^{k}, y-y^{k+1}}-\innp{\bar{y}^{k+1}, K(x-x^{k+1})}+g(x^{k+1})
		-h(y) +h(y^{k+1})-g(x).
	\end{align*}
\item  \label{lemma:partialinequalities:fK}
For every $x \in \mathcal{H}_{1}$ and for every $y \in \mathcal{H}_{2}$,
we have that
\begin{align*}
	&\innp{K \bar{x}^{k}, y-y^{k+1}}-\innp{\bar{y}^{k+1}, K(x-x^{k+1})}+g(x^{k+1})
	-h(y) +h(y^{k+1})-g(x)\\
	=&f(x^{k+1},y) -f(x,y^{k+1})+\innp{K(x^{k+1}-\bar{x}^{k}),y^{k+1}-y}
	-\innp{K(x^{k+1}-x),y^{k+1}-\bar{y}^{k+1}}.
\end{align*}
\item  \label{lemma:partialinequalities:f}
For every $x \in \mathcal{H}_{1}$ and for every $y \in \mathcal{H}_{2}$,
we have that
\begin{align*}
	&\innp{K \bar{x}^{k}, y-y^{k+1}}-\innp{\bar{y}^{k+1}, K(x-x^{k+1})}+g(x^{k+1})
	-h(y) +h(y^{k+1})-g(x)\\
	=& f(\bar{x}^{k},y) - f(\bar{x}^{k}, y^{k+1})   + f(x^{k+1}, \bar{y}^{k+1}) - 
	f(x,\bar{y}^{k+1}).
\end{align*}

\end{enumerate}
\end{lemma}

\begin{proof}
	\cref{lemma:partialinequalities:innpx}: 
	Based on \cref{eq:fact:algorithmsymplity:xk}, 
	we have that for every $k \in \mathbf{N}$,
	\begin{align*}
		 -\tau_{k}K^{*}\bar{y}^{k+1}+x^{k} \in x^{k+1} +\tau_{k}\partial g(x^{k+1})
		 \Leftrightarrow 
		 \frac{x^{k} -x^{k+1}}{\tau_{k}} -K^{*}\bar{y}^{k+1} \in \partial g(x^{k+1}).
	\end{align*}
	Combine this with \cref{eq:gmu}
	in \cref{assumption:basic} to deduce the required result.
	
	\cref{lemma:partialinequalities:innpy}:
	 In view of  \cref{eq:fact:algorithmsymplity:yk},
	 for every $k \in \mathbf{N}$,
	 \begin{align*}
	 	\sigma_{k}K\bar{x}^{k} +y^{k} \in y^{k+1} +\sigma_{k} \partial h(y^{k+1})
	 	\Leftrightarrow 
	 	\frac{y^{k}-y^{k+1}}{\sigma_{k}} +K\bar{x}^{k} \in  \partial h(y^{k+1}),
	 \end{align*} 
	 	which, connecting with
	 	\cref{eq:hnu} in \cref{assumption:basic},
	 	yields the desired result. 
	 	
	\cref{lemma:partialinequalities:normx}$\&$\cref{lemma:partialinequalities:normy}:
	It is easy to check that for every  Hilbert space $\mathcal{H}$
	 and for every  $(a, b,c) \in  \mathcal{H}^{3}$,
	\begin{align}\label{eq:lemma:partialinequalities:normx:abc}
		\innp{a-b, c-b} = \frac{1}{2} \left(\norm{a-b}^{2} + \norm{c-b}^{2} 
		-\norm{c-a}^{2}\right).
	\end{align}
	
	The inequality in 	\cref{lemma:partialinequalities:normx} 
	(resp.\,in \cref{lemma:partialinequalities:normy}) is from 
	\cref{eq:lemma:partialinequalities:normx:abc} and the result 
	obtained in 
	\cref{lemma:partialinequalities:innpx} (resp.\,\cref{lemma:partialinequalities:innpy}). 
	respectively.  
	
	\cref{lemma:partialinequalities:sum}: Add the inequalities in  
	\cref{lemma:partialinequalities:normx} and \cref{lemma:partialinequalities:normy}
	to derive the required result.
	
	\cref{lemma:partialinequalities:fK}: It is clear that 
	\begin{align*}
		&\innp{K \bar{x}^{k}, y-y^{k+1}}-\innp{\bar{y}^{k+1}, K(x-x^{k+1})}+g(x^{k+1})
		-h(y) +h(y^{k+1})-g(x)\\
		=&\innp{Kx^{k+1},y} +g(x^{k+1})-h(y) 
		-\left(\innp{Kx,y^{k+1}}+g(x)-h(y^{k+1})\right)\\
		&-\innp{Kx^{k+1},y} + \innp{Kx,y^{k+1}} 
		+\innp{K \bar{x}^{k}, y-y^{k+1}}-\innp{\bar{y}^{k+1}, K(x-x^{k+1})}\\
		=&f(x^{k+1},y) -f(x,y^{k+1})+\innp{K(x^{k+1}-\bar{x}^{k}),y^{k+1}-y}
		-\innp{K(x^{k+1}-x),y^{k+1}-\bar{y}^{k+1}}\\
		&-\innp{Kx^{k+1},y^{k+1}-y}+\innp{K(x^{k+1}-x),y^{k+1}}-\innp{Kx^{k+1},y}+
		\innp{Kx,y^{k+1}}\\
		=&f(x^{k+1},y) -f(x,y^{k+1})+\innp{K(x^{k+1}-\bar{x}^{k}),y^{k+1}-y}
		-\innp{K(x^{k+1}-x),y^{k+1}-\bar{y}^{k+1}},
	\end{align*}
where in the second equality, we use \cref{eq:fspecialdefine}.

	\cref{lemma:partialinequalities:f}: Applying the result stated in 
		\cref{lemma:partialinequalities:fK} above and employing 
		\cite[Lemma~4.4(i)]{Oy2023Proximity}, we have that
	\begin{align*}
		&\innp{K \bar{x}^{k}, y-y^{k+1}}-\innp{\bar{y}^{k+1}, K(x-x^{k+1})}+g(x^{k+1})
		-h(y) +h(y^{k+1})-g(x)\\
		=&f(x^{k+1},y) -f(x,y^{k+1})+\innp{K(x^{k+1}-\bar{x}^{k}),y^{k+1}-y}
		-\innp{K(x^{k+1}-x),y^{k+1}-\bar{y}^{k+1}}\\
		=&f(x^{k+1},y) -f(x,y^{k+1})\\
		&-\left(f(x^{k+1},y) -f(\bar{x}^{k},y) + f(\bar{x}^{k}, y^{k+1})   
		-f(x,y^{k+1}) - f(x^{k+1}, \bar{y}^{k+1}) + f(x,\bar{y}^{k+1})\right)\\
		=& f(\bar{x}^{k},y) - f(\bar{x}^{k}, y^{k+1})   + f(x^{k+1}, \bar{y}^{k+1}) - 
		f(x,\bar{y}^{k+1}).
	\end{align*}

\end{proof}

\begin{lemma} \label{lemma:inequalityx*y*}
	Let $(x^{*},y^{*})$ be a saddle-point of $f$ and let $k$ be in $\mathbf{N}$. Then
	\begin{align*}
		&\frac{1}{2 \tau_{k}} \norm{x^{*} -x^{k}}^{2} +\frac{1}{2\sigma_{k}} \norm{y^{*} 
			-y^{k}}^{2}\\
		\geq &\left(\mu +\frac{1}{2\tau_{k}}\right)\norm{x^{*} -x^{k+1}}^{2}+\left( 
		\nu +\frac{1}{2 \sigma_{k}}\right) \norm{y^{*} -y^{k+1}}^{2}
		+ \frac{1}{2\tau_{k}} \norm{x^{k} -x^{k+1}}^{2} 
		+ \frac{1}{2 \sigma_{k}} \norm{y^{k} -y^{k+1}}^{2}\\
		&+\innp{K(x^{k+1}-\bar{x}^{k}),y^{k+1}-y^{*}}
		-\innp{K(x^{k+1}-x^{*}),y^{k+1}-\bar{y}^{k+1}}.
	\end{align*}
\end{lemma}

\begin{proof}
	Applying \cref{lemma:partialinequalities:sum} and \cref{lemma:partialinequalities:fK}
	in \cref{lemma:partialinequalities},
	respectively in the following first inequality and the first equality below 
	with $(x,y)=(x^{*},y^{*})$, and employing 
 \cref{lemma:partialstronglyconvex}\cref{lemma:partialstronglyconvex:f}
 with $(x,y) =(x^{k+1},y^{k+1})$
  in the second inequality below, 
	we observe that 
	\begin{align*}
		&\frac{1}{2 \tau_{k}} \norm{x^{*} -x^{k}}^{2} +\frac{1}{2\sigma_{k}} \norm{y^{*}
			-y^{k}}^{2}\\
		\geq &\left(\frac{\mu}{2} +\frac{1}{2\tau_{k}}\right)\norm{x^{*} -x^{k+1}}^{2}+\left( 
		\frac{\nu}{2} +\frac{1}{2 \sigma_{k}}\right) \norm{y^{*} -y^{k+1}}^{2}
		+ \frac{1}{2\tau_{k}} \norm{x^{k} -x^{k+1}}^{2} 
		+ \frac{1}{2 \sigma_{k}} \norm{y^{k} -y^{k+1}}^{2}\\
		&+\innp{K \bar{x}^{k}, y^{*}-y^{k+1}}-\innp{\bar{y}^{k+1},
			K(x^{*}-x^{k+1})}+g(x^{k+1})
		-h(y^{*}) +h(y^{k+1})-g(x^{*})\\
		= & \left(\frac{\mu}{2} +\frac{1}{2\tau_{k}}\right)\norm{x^{*} -x^{k+1}}^{2}+\left( 
		\frac{\nu}{2} +\frac{1}{2 \sigma_{k}}\right) \norm{y^{*} -y^{k+1}}^{2}
		+ \frac{1}{2\tau_{k}} \norm{x^{k} -x^{k+1}}^{2} 
		+ \frac{1}{2 \sigma_{k}} \norm{y^{k} -y^{k+1}}^{2}\\
		&+f(x^{k+1},y^{*}) -f(x^{*},y^{k+1})+\innp{K(x^{k+1}-\bar{x}^{k}),y^{k+1}-y^{*}}
		-\innp{K(x^{k+1}-x^{*}),y^{k+1}-\bar{y}^{k+1}}\\
		\geq &\left(\mu +\frac{1}{2\tau_{k}}\right)\norm{x^{*} -x^{k+1}}^{2}+\left( 
		\nu +\frac{1}{2 \sigma_{k}}\right) \norm{y^{*} -y^{k+1}}^{2}
		+ \frac{1}{2\tau_{k}} \norm{x^{k} -x^{k+1}}^{2} 
		+ \frac{1}{2 \sigma_{k}} \norm{y^{k} -y^{k+1}}^{2}\\
		&+\innp{K(x^{k+1}-\bar{x}^{k}),y^{k+1}-y^{*}}
		-\innp{K(x^{k+1}-x^{*}),y^{k+1}-\bar{y}^{k+1}}.
	\end{align*}
\end{proof}

\subsection{Applications of Cauchy-Schwarz Inequality}
To derive results in this subsection, we mainly employ the Cauchy-Schwarz inequality
and properties of the norm of the operator and its adjoint in various ways.
Although the technique is fundamental and well-known, 
the results play an essential role to establish
our linear convergence results in the following sections. 
As explained in \cref{remark:ineqaulities} below, our techniques of proof in 
\Cref{section:linearconvergeIterations,section:linearconvergf} below together with
inequalities presented in \cref{remark:ineqaulities} can actually deduce  
many other similar linear convergence results  illustrated the following 
\Cref{section:linearconvergeIterations,section:linearconvergf}.

\begin{lemma}\label{lemma:inequalityCS}
	Let $k$ be in $\mathbf{N}$ and 
	let $(x,y)$ be in $\mathcal{H}_{1} \times \mathcal{H}_{2}$. 
	Let $\xi$, $\eta_{1}$, $\eta_{2}$, $\eta_{3}$, and $\eta_{4}$ be in $\mathbf{R}_{++}$.
	The following results hold. 
	\begin{enumerate}
		\item \label{lemma:inequalityCS:K}
		We have that 
			\begin{align*}
			&\innp{K(x^{k+1}-\bar{x}^{k}),y^{k+1}-y}
			-\innp{K(x^{k+1}-x),y^{k+1}-\bar{y}^{k+1}}\\
			\geq &\innp{K(x^{k+1}-x^{k}),y^{k+1}-y}
			-\xi \innp{K(x^{k}-x^{k-1}),y^{k}-y}\\
			&-\frac{ \norm{K}}{2}\left(\eta_{1}\xi^{2}\norm{x^{k}-x^{k-1}}^{2}
			+\frac{\norm{y^{k+1}-y^{k} }}{\eta_{1}}^{2}\right) \\
			&-\frac{\norm{K}}{2} \left( \eta_{2} \left(\xi -\alpha_{k-1} 
			\right)^{2}\norm{y^{k+1}-y}^{2} 
			+\frac{\norm{x^{k}-x^{k-1}}^{2}}{\eta_{2}}
			\right)\\
			&-\frac{\norm{K}}{2} \left( \eta_{3}\beta_{k}^{2}\norm{x^{k+1}-x}^{2} 
			+\frac{\norm{y^{k+1}-y^{k}}^{2} }{\eta_{3}}  \right).
		\end{align*}
	
		\item \label{lemma:inequalityCS:Ksquare}
	We have that 
	\begin{align*}
		&\innp{K(x^{k+1}-\bar{x}^{k}),y^{k+1}-y}
		-\innp{K(x^{k+1}-x),y^{k+1}-\bar{y}^{k+1}}\\
		\geq &\innp{K(x^{k+1}-x^{k}),y^{k+1}-y}
		-\xi \innp{K(x^{k}-x^{k-1}),y^{k}-y}\\
		&-\frac{1}{2}\left(\eta_{1} \norm{K}^{2}\xi^{2}\norm{x^{k}-x^{k-1}}^{2}
		+\frac{\norm{y^{k+1}-y^{k} }}{\eta_{1}}^{2}\right) \\
		&-\frac{1}{2} \left( \eta_{2} \norm{K}^{2}\left(\xi -\alpha_{k-1} 
		\right)^{2}\norm{y^{k+1}-y}^{2} 
		+\frac{\norm{x^{k}-x^{k-1}}^{2}}{\eta_{2}}
		\right)\\
		&-\frac{1}{2} \left( \eta_{3}\norm{K}^{2}\beta_{k}^{2}\norm{x^{k+1}-x}^{2} 
		+\frac{\norm{y^{k+1}-y^{k}}^{2} }{\eta_{3}}  \right).
	\end{align*}
\item \label{lemma:inequalityCS:Ksquareeta4}
We have that 
\begin{align*}
	 &\abs{\innp{K(x^{k+1}-x^{k}),y^{k+1}-y}}\\
	 \leq & \frac{1}{2} \left(\eta_{4} \norm{K}^{2}\norm{y^{k+1}-y}^{2}
	  +\frac{1}{\eta_{4}}\norm{x^{k+1}-x^{k}}^{2} \right).
\end{align*}
	\end{enumerate}
\end{lemma}

\begin{proof}
	\cref{lemma:inequalityCS:K}$\&$\cref{lemma:inequalityCS:Ksquare}:
	Recall from \cref{eq:fact:algorithmsymplity:barx} and  
	\cref{eq:fact:algorithmsymplity:bary} that
	$(\forall k \in \mathbf{N})$ 
	$\bar{x}^{k+1} = x^{k+1} + \alpha_{k} (x^{k+1} -x^{k})$ 
	and  $ \bar{y}^{k+1} = y^{k+1} + \beta_{k} (y^{k+1} -y^{k})$. 
	Hence, we know that 
	\begin{subequations}\label{eq:lemma:inequalityCS:innp}
		\begin{align}
			 &\innp{K(x^{k+1}-\bar{x}^{k}),y^{k+1}-y}
			 -\innp{K(x^{k+1}-x),y^{k+1}-\bar{y}^{k+1}}\\
			 =&\innp{K(x^{k+1}-x^{k}),y^{k+1}-y}
			 -\alpha_{k-1}\innp{K(x^{k}-x^{k-1}),y^{k+1}-y}
			 +\beta_{k} \innp{K(x^{k+1}-x),y^{k+1}-y^{k}}\\
			 =&\innp{K(x^{k+1}-x^{k}),y^{k+1}-y}
			 -\xi \innp{K(x^{k}-x^{k-1}),y^{k+1}-y}
			 +\xi \innp{K(x^{k}-x^{k-1}),y^{k+1}-y}\\
			 &-\alpha_{k-1}\innp{K(x^{k}-x^{k-1}),y^{k+1}-y}
			 +\beta_{k} \innp{K(x^{k+1}-x),y^{k+1}-y^{k}}\\
			 =&\innp{K(x^{k+1}-x^{k}),y^{k+1}-y}
			 -\xi \innp{K(x^{k}-x^{k-1}),y^{k}-y}
			 -\xi \innp{K(x^{k}-x^{k-1}),y^{k+1}-y^{k}}\\
			 &+\left(\xi -\alpha_{k-1} \right) \innp{K(x^{k}-x^{k-1}),y^{k+1}-y}
			 +\beta_{k} \innp{K(x^{k+1}-x),y^{k+1}-y^{k}}.
		\end{align}
	\end{subequations}

Note that 
\begin{subequations}\label{eq:lemma:inequalityCS:K}
	\begin{align}
		&
		-\xi \innp{K(x^{k}-x^{k-1}),y^{k+1}-y^{k}} 
		+\left(\xi -\alpha_{k-1} \right) \innp{K(x^{k}-x^{k-1}),y^{k+1}-y}\\
		&+\beta_{k} \innp{K(x^{k+1}-x),y^{k+1}-y^{k}}\\
		\geq 
		&-\frac{ \norm{K}}{2}\left(\eta_{1}\xi^{2}\norm{x^{k}-x^{k-1}}^{2}
		+\frac{\norm{y^{k+1}-y^{k} }}{\eta_{1}}^{2}\right) \\
		&-\frac{\norm{K}}{2} \left( \eta_{2} \left(\xi -\alpha_{k-1} 
		\right)^{2}\norm{y^{k+1}-y}^{2} 
		+\frac{\norm{x^{k}-x^{k-1}}^{2}}{\eta_{2}}
		\right)\\
		&-\frac{\norm{K}}{2} \left( \eta_{3}\beta_{k}^{2}\norm{x^{k+1}-x}^{2} 
		+\frac{\norm{y^{k+1}-y^{k}}^{2} }{\eta_{3}}  \right),
	\end{align}
\end{subequations}
where, in the last inequality, we use  the Cauchy-Schwarz inequality, 
\cref{eq:definenormK},  
and the fact that 
\begin{align}\label{eq:lemma:inequalityCS:ab}
	(\forall (a,b) \in \mathbf{R}_{+}) (\forall \kappa \in \mathbf{R}_{++}) \quad 
	2ab = 2 (\sqrt{\kappa}a ) \left(\frac{b}{\sqrt{\kappa}}\right) 
	\leq \kappa a^{2} + \frac{b^{2}}{\kappa}.
\end{align}
Combine \cref{eq:lemma:inequalityCS:innp} and \cref{eq:lemma:inequalityCS:K}
to derive the required result in \cref{lemma:inequalityCS:K}. 

Applying the Cauchy-Schwarz inequality, 
\cref{eq:definenormK},  and the fact \cref{eq:lemma:inequalityCS:ab}, 
we have that 
\begin{subequations}\label{eq:lemma:inequalityCS:Ksquare}
	\begin{align}
		&
		-\xi \innp{K(x^{k}-x^{k-1}),y^{k+1}-y^{k}} 
		+\left(\xi -\alpha_{k-1} \right) \innp{K(x^{k}-x^{k-1}),y^{k+1}-y}\\
		&+\beta_{k} \innp{K(x^{k+1}-x),y^{k+1}-y^{k}}\\
\geq &-\frac{1}{2}\left(\eta_{1} \norm{K}^{2}\xi^{2}\norm{x^{k}-x^{k-1}}^{2}
+\frac{\norm{y^{k+1}-y^{k} }}{\eta_{1}}^{2}\right) \\
&-\frac{1}{2} \left( \eta_{2} \norm{K}^{2}\left(\xi -\alpha_{k-1} 
\right)^{2}\norm{y^{k+1}-y}^{2} 
+\frac{\norm{x^{k}-x^{k-1}}^{2}}{\eta_{2}}
\right)\\
&-\frac{1}{2} \left( \eta_{3}\norm{K}^{2}\beta_{k}^{2}\norm{x^{k+1}-x}^{2} 
+\frac{\norm{y^{k+1}-y^{k}}^{2} }{\eta_{3}}  \right),	 
	\end{align}
\end{subequations}
where, to get the inequality
\begin{align*}
	&\left(\xi -\alpha_{k-1} \right) \innp{K(x^{k}-x^{k-1}),y^{k+1}-y}\\
	 \geq
	&-\frac{1}{2} \left( \eta_{2} \norm{K}^{2}\left(\xi -\alpha_{k-1} 
	\right)^{2}\norm{y^{k+1}-y}^{2} 
	+\frac{\norm{x^{k}-x^{k-1}}^{2}}{\eta_{2}}
	\right),
\end{align*}  
we adopt the definition of the adjoint  $K^{*}$ of  $K$
and also the fact $\norm{K} =\norm{K^{*}}$ from \cite[Fact~2.25(ii)]{BC2017}.		
Combining \cref{eq:lemma:inequalityCS:innp} and 
 \cref{eq:lemma:inequalityCS:Ksquare}, we obtain the desired result in 
 \cref{lemma:inequalityCS:Ksquare}. 
 
 \cref{lemma:inequalityCS:Ksquareeta4}: Applying the Cauchy-Schwarz inequality,
 the definition of the adjoint  $K^{*}$ of  $K$,
 and also the fact $\norm{K} =\norm{K^{*}}$ from \cite[Fact~2.25(ii)]{BC2017} again, 
 we conclude the required result. 
\end{proof}

\begin{remark}\label{remark:ineqaulities}
	Let's revisit \cref{lemma:inequalityCS}.
	It is clear that using the techniques in the proof of
	 \cref{lemma:inequalityCS}\cref{lemma:inequalityCS:K}$\&$\cref{lemma:inequalityCS:Ksquare},
%
	 we have at least the following results, 
	 \begin{enumerate}
	 	\item $-\xi \innp{K(x^{k}-x^{k-1}),y^{k+1}-y^{k}} \geq 
	 	-\frac{\xi \norm{K}}{2}\left(\eta_{1}\norm{x^{k}-x^{k-1}}^{2}
	 	+\frac{\norm{y^{k+1}-y^{k} }}{\eta_{1}}^{2}\right)$,\\
	 	$-\xi \innp{K(x^{k}-x^{k-1}),y^{k+1}-y^{k}} \geq 
	 	-\frac{ \norm{K}}{2}\left(\eta_{1}\xi^{2}\norm{x^{k}-x^{k-1}}^{2}
	 	+\frac{\norm{y^{k+1}-y^{k} }}{\eta_{1}}^{2}\right)$,\\
	 	$-\xi \innp{K(x^{k}-x^{k-1}),y^{k+1}-y^{k}} \geq 
	 	-\frac{ \norm{K}}{2}\left(\eta_{1}\norm{x^{k}-x^{k-1}}^{2}
	 	+\xi^{2} \frac{\norm{y^{k+1}-y^{k} }}{\eta_{1}}^{2}\right)$,\\
	 	$-\xi \innp{K(x^{k}-x^{k-1}),y^{k+1}-y^{k}} \geq 
	 	-\frac{1}{2}\left(\eta_{1} \norm{K}^{2}\xi^{2}\norm{x^{k}-x^{k-1}}^{2}
	 	+\frac{\norm{y^{k+1}-y^{k} }}{\eta_{1}}^{2}\right)$,\\
	 	and 
	 	$-\xi \innp{K(x^{k}-x^{k-1}),y^{k+1}-y^{k}} \geq 
	 	-\frac{1}{2}\left(\eta_{1} \norm{x^{k}-x^{k-1}}^{2}
	 	+\norm{K}^{2}\xi^{2}\frac{\norm{y^{k+1}-y^{k} }}{\eta_{1}}^{2}\right)$.
	 	
	 	\item $\left(\xi -\alpha_{k-1} \right) \innp{K(x^{k}-x^{k-1}),y^{k+1}-y} \geq 
	 	-\frac{\norm{K}}{2} \abs{\xi -\alpha_{k-1} }
	 	\left( \eta_{2} \norm{y^{k+1}-y}^{2} 
	 	+\frac{\norm{x^{k}-x^{k-1}}^{2}}{\eta_{2}}\right)$,\\
	 	$\left(\xi -\alpha_{k-1} \right) \innp{K(x^{k}-x^{k-1}),y^{k+1}-y} \geq 
	 	-\frac{\norm{K}}{2} \left( \eta_{2} \left(\xi -\alpha_{k-1} 
	 	\right)^{2}\norm{y^{k+1}-y}^{2} 
	 	+\frac{\norm{x^{k}-x^{k-1}}^{2}}{\eta_{2}}\right)$,\\
	 	$\left(\xi -\alpha_{k-1} \right) \innp{K(x^{k}-x^{k-1}),y^{k+1}-y} \geq 
	 	-\frac{\norm{K}}{2} \left( \eta_{2} \norm{y^{k+1}-y}^{2} 
	 	+\left(\xi -\alpha_{k-1} \right)^{2}\frac{\norm{x^{k}-x^{k-1}}^{2}}{\eta_{2}}\right)$,\\
	 	$\left(\xi -\alpha_{k-1} \right) \innp{K(x^{k}-x^{k-1}),y^{k+1}-y} \geq 
	 	-\frac{1}{2} \left( \eta_{2} \norm{K}^{2}\left(\xi 
	 	-\alpha_{k-1}\right)^{2}\norm{y^{k+1}-y}^{2} 
	 	+\frac{\norm{x^{k}-x^{k-1}}^{2}}{\eta_{2}}\right)$,\\
	 	and $\left(\xi -\alpha_{k-1} \right) \innp{K(x^{k}-x^{k-1}),y^{k+1}-y} \geq 
	 	-\frac{1}{2} \left( \eta_{2} \norm{y^{k+1}-y}^{2} 
	 	+\norm{K}^{2}\left(\xi -\alpha_{k-1} 
	 	\right)^{2}\frac{\norm{x^{k}-x^{k-1}}^{2}}{\eta_{2}}
	 	\right)$.
	 	
	 	\item $\beta_{k} \innp{K(x^{k+1}-x),y^{k+1}-y^{k}} \geq 
	 	-\frac{\norm{K}}{2}\beta_{k} \left( \eta_{3}\norm{x^{k+1}-x}^{2} 
	 	+\frac{\norm{y^{k+1}-y^{k}}^{2} }{\eta_{3}}  \right)$,\\
	 	$\beta_{k} \innp{K(x^{k+1}-x),y^{k+1}-y^{k}} \geq 
	 	-\frac{\norm{K}}{2} \left( \eta_{3}\beta_{k}^{2}\norm{x^{k+1}-x}^{2} 
	 	+\frac{\norm{y^{k+1}-y^{k}}^{2} }{\eta_{3}}  \right)$,\\
	 	$\beta_{k} \innp{K(x^{k+1}-x),y^{k+1}-y^{k}} \geq 
	 	-\frac{\norm{K}}{2} \left( \eta_{3}\norm{x^{k+1}-x}^{2} 
	 	+\beta_{k}^{2}\frac{\norm{y^{k+1}-y^{k}}^{2} }{\eta_{3}}  \right)$,\\
	 	$\beta_{k} \innp{K(x^{k+1}-x),y^{k+1}-y^{k}} \geq 
	 	-\frac{1}{2} \left( \eta_{3}\norm{K}^{2}\beta_{k}^{2}\norm{x^{k+1}-x}^{2} 
	 	+\frac{\norm{y^{k+1}-y^{k}}^{2} }{\eta_{3}}  \right)$,\\
	 	and $\beta_{k} \innp{K(x^{k+1}-x),y^{k+1}-y^{k}} \geq 
	 	-\frac{1}{2} \left( \eta_{3}\norm{x^{k+1}-x}^{2} 
	 	+\norm{K}^{2}\beta_{k}^{2}\frac{\norm{y^{k+1}-y^{k}}^{2} }{\eta_{3}}  \right)$.
	 \end{enumerate}
 In addition, we can also separate the coefficients $ \norm{K}^{2}\xi^{2}$, 
 $ \norm{K}^{2}\left(\xi -\alpha_{k-1}  \right)^{2}$, and 
 $\norm{K}^{2}\beta_{k}^{2}$, respectively, in the related inequality and 
  get six more inequalities:
  
$-\xi \innp{K(x^{k}-x^{k-1}),y^{k+1}-y^{k}} \geq 
-\frac{1}{2}\left(\eta_{1} \norm{K}^{2}\norm{x^{k}-x^{k-1}}^{2}
+\xi^{2}\frac{\norm{y^{k+1}-y^{k} }}{\eta_{1}}^{2}\right)$,

$-\xi \innp{K(x^{k}-x^{k-1}),y^{k+1}-y^{k}} \geq 
-\frac{1}{2}\left(\eta_{1} \xi^{2}\norm{x^{k}-x^{k-1}}^{2}
+\norm{K}^{2}\frac{\norm{y^{k+1}-y^{k} }}{\eta_{1}}^{2}\right)$,
and so on.

Therefore, by various combination of inequalities we get, 
we potentially have at least  $7^{3}$  lower bounds of 
$\innp{K(x^{k+1}-\bar{x}^{k}),y^{k+1}-y}
-\innp{K(x^{k+1}-x),y^{k+1}-\bar{y}^{k+1}}$
 including the two lower bounds
 presented in 
 \cref{lemma:inequalityCS}\cref{lemma:inequalityCS:K}$\&$\cref{lemma:inequalityCS:Ksquare}.
 For simplicity, we work on only the two inequalities given in 
 \cref{lemma:inequalityCS}\cref{lemma:inequalityCS:K}$\&$\cref{lemma:inequalityCS:Ksquare},
  in this work. 
\end{remark}

\section{Linear Convergence of Iteration Points} 
\label{section:linearconvergeIterations}
In this section, we present linear convergence of the sequence of iterations generated
by the alternating proximal mapping method updated as
 \cref{eq:fact:algorithmsymplity}.
 
 The following result will be used twice in our proofs below. 
 \begin{lemma}\label{lemma:ksum}
 Let $\xi$ be in $\mathbf{R}_{++}$.	Suppose that for every $k$ in $\mathbf{N}$,  
 	\begin{subequations}\label{eq:lemma:ksum}
 		\begin{align}
 			&\frac{1}{2 \tau_{k}} \norm{x^{*} -x^{k}}^{2} +\frac{1}{2\sigma_{k}} \norm{y^{*} 
 				-y^{k}}^{2}\\
 			\geq &\frac{1}{2\tau_{k+1}\xi}\norm{x^{*} -x^{k+1}}^{2}
 			+\frac{1}{2\sigma_{k+1}\xi}\norm{y^{*} -y^{k+1}}^{2}
 			+ \frac{1}{2\tau_{k}} \norm{x^{k} -x^{k+1}}^{2}
 			- \frac{\xi}{2\tau_{k-1}} \norm{x^{k} -x^{k+1}}^{2}\\
 			&+\innp{K(x^{k+1}-x^{k}),y^{k+1}-y^{*}}
 			-\xi \innp{K(x^{k}-x^{k-1}),y^{k}-y^{*}}.
 		\end{align}
 	\end{subequations}
 	Let $\eta_{4}$ be in $\mathbf{R}_{++}$ such that $(\forall k \in \mathbf{N})$
 	$\frac{1}{\tau_{k}} -\frac{1}{\eta_{4}} \geq 0$, 
 	and  let $N$ be in $\mathbf{N}$. Then
 \begin{align*}
 	&\xi^{N+1} \left( \frac{1}{  \tau_{0}} \norm{x^{*} -x^{0}}^{2} +\frac{1}{ \sigma_{0}} 
 	\norm{y^{*} -y^{0}}^{2} \right)\\	
 	\geq  & \frac{1}{ \tau_{N+1}}\norm{x^{*} -x^{N+1}}^{2}
 	+ \left(\frac{1}{\sigma_{N+1}} -\xi \eta_{4} \norm{K}^{2} \right)\norm{y^{*} 
 		-y^{N+1}}^{2}.
 \end{align*}
 \end{lemma}
 
 \begin{proof}
 	For every $k$ in $\mathbf{N}$, multiply both sides of 
 	the inequality  \cref{eq:lemma:ksum}  by $\frac{1}{\xi^{k}}$ 
 	to derive that
 	\begin{align*}
 		&\frac{1}{2 \tau_{k}\xi^{k}} \norm{x^{*} -x^{k}}^{2} +\frac{1}{2\sigma_{k}\xi^{k}} 
 		\norm{y^{*} -y^{k}}^{2}\\
 		\geq &\frac{1}{2\tau_{k+1}\xi^{k+1}}\norm{x^{*} -x^{k+1}}^{2}
 		+\frac{1}{2\sigma_{k+1}\xi^{k+1}}\norm{y^{*} -y^{k+1}}^{2}
 		+ \frac{1}{2\tau_{k}\xi^{k}} \norm{x^{k} -x^{k+1}}^{2}
 		- \frac{1}{2\tau_{k-1}\xi^{k-1}} \norm{x^{k} -x^{k-1}}^{2}\\
 		&+\frac{1}{\xi^{k}}\innp{K(x^{k+1}-x^{k}),y^{k+1}-y^{*}}
 		-\frac{1}{\xi^{k-1}} \innp{K(x^{k}-x^{k-1}),y^{k}-y^{*}}.
 	\end{align*}
 	Note that $x^{-1} =x^{0}$.	Add the inequality above from $k=0$ to $k=N$ to deduce 
 	that
 	\begin{align*}
 		&\frac{1}{2 \tau_{0}} \norm{x^{*} -x^{0}}^{2} +\frac{1}{2\sigma_{0}} 
 		\norm{y^{*} -y^{0}}^{2}\\
 		\geq &\frac{1}{2\tau_{N+1}\xi^{N+1}}\norm{x^{*} -x^{N+1}}^{2}
 		+\frac{1}{2\sigma_{N+1}\xi^{N+1}}\norm{y^{*} -y^{N+1}}^{2}
 		+ \frac{1}{2\tau_{N}\xi^{N}} \norm{x^{N} -x^{N+1}}^{2}\\
 		&- \frac{1}{2\tau_{-1}\xi^{-1}} \norm{x^{0} -x^{-1}}^{2}
 		+\frac{1}{\xi^{N}}\innp{K(x^{N+1}-x^{N}),y^{N+1}-y^{*}}
 		-\frac{1}{\xi^{-1}} \innp{K(x^{0}-x^{-1}),y^{0}-y^{*}}\\
 		=&\frac{1}{2\tau_{N+1}\xi^{N+1}}\norm{x^{*} -x^{N+1}}^{2}
 		+\frac{1}{2\sigma_{N+1}\xi^{N+1}}\norm{y^{*} -y^{N+1}}^{2}
 		+ \frac{1}{2\tau_{N}\xi^{N}} \norm{x^{N} -x^{N+1}}^{2}\\
 		&  +\frac{1}{\xi^{N}}\innp{K(x^{N+1}-x^{N}),y^{N+1}-y^{*}}\\
 		\geq &\frac{1}{2\tau_{N+1}\xi^{N+1}}\norm{x^{*} -x^{N+1}}^{2}
 		+\frac{1}{2\sigma_{N+1}\xi^{N+1}}\norm{y^{*} -y^{N+1}}^{2}
 		+ \frac{1}{2\tau_{N}\xi^{N}} \norm{x^{N} -x^{N+1}}^{2}\\
 		&-\frac{1}{2\xi^{N}} \left( \eta_{4}\norm{K}^{2} \norm{ y^{N+1}-y^{*}}^{2} 
 		+ \frac{1}{\eta_{4}}\norm{ x^{N+1}-x^{N}}^{2} \right)\\
 		=&\frac{1}{2\tau_{N+1}\xi^{N+1}}\norm{x^{*} -x^{N+1}}^{2}
 		+\frac{1}{2\xi^{N+1}} \left(\frac{1}{\sigma_{N+1}} -\xi \eta_{4}\norm{K}^{2} 
 		\right)\norm{y^{*} 	-y^{N+1}}^{2}\\
 		&+\frac{1}{2\xi^{N}} \left(\frac{1}{\tau_{N}} -\frac{1}{\eta_{4}}\right) \norm{x^{N} 
 			-x^{N+1}}^{2}\\
 		\geq &  \frac{1}{2\tau_{N+1}\xi^{N+1}}\norm{x^{*} -x^{N+1}}^{2}
 		+\frac{1}{2\xi^{N+1}} \left(\frac{1}{\sigma_{N+1}} -\xi \eta_{4} \norm{K}^{2} 
 		\right)\norm{y^{*} 
 			-y^{N+1}}^{2},
 	\end{align*}	
 where in the second inequality, we use
  \cref{lemma:inequalityCS}\cref{lemma:inequalityCS:Ksquareeta4},
  and in the last inequality, we adopt the assumption 
  $(\forall k \in \mathbf{N})$  $\frac{1}{\tau_{k}} -\frac{1}{\eta_{4}} \geq 0$.
 \end{proof}

 Employing \Cref{lemma:inequalityx*y*,lemma:ksum}, and applying the  inequality in 
   \cref{lemma:inequalityCS}\cref{lemma:inequalityCS:K} (resp.\,in
   \cref{lemma:inequalityCS}\cref{lemma:inequalityCS:Ksquare}),
   we derive the following \cref{prop:inequalityx*y*K} (resp.\,\cref{prop:inequalityx*y*}).
 
\begin{proposition} \label{prop:inequalityx*y*K}
	Let $(x^{*},y^{*})$ be a saddle-point of $f$, 
	and let $\xi$, $\eta_{1}$, $\eta_{2}$, and $\eta_{3}$ be in 
	$\mathbf{R}_{++}$. Suppose that $(\forall k \in \mathbf{N})$
	$ \left(\mu +\frac{1}{2\tau_{k}}  -\frac{\norm{K}}{2} \eta_{3}\beta_{k}^{2} \right) 
	\geq  \frac{1}{2 \tau_{k+1} \xi}$,
	$\nu +\frac{1}{2 \sigma_{k}} -\frac{\eta_{2} \norm{K}\left(\xi -\alpha_{k-1} 
		\right)^{2}}{2} \geq \frac{1}{2\sigma_{k+1}\xi}$,
	$  \norm{K} \left(\eta_{1}\xi^{2}
	+ \frac{1}{\eta_{2}}\right) \leq \frac{\xi}{\tau_{k-1}}$, and
 $\frac{1}{ \sigma_{k}} - \norm{K} \left(\frac{1}{\eta_{1}} 
	+\frac{1}{\eta_{3}} \right) \geq 0$.

	\begin{enumerate}
		\item \label{prop:inequalityx*y*K:k} For every $k$ in $\mathbf{N}$, we have that 
			\begin{align*}
			&\frac{1}{2 \tau_{k}} \norm{x^{*} -x^{k}}^{2} +\frac{1}{2\sigma_{k}} \norm{y^{*} 
				-y^{k}}^{2}\\
		\geq & \frac{1}{2 \tau_{k+1} \xi} \norm{x^{*}-x^{k+1}}^{2}
		+\frac{1}{2 \sigma_{k+1} \xi}	 \norm{y^{*} -y^{k+1}}^{2}  
		+ \frac{1}{2\tau_{k}} \norm{x^{k} -x^{k+1}}^{2} - \frac{\xi}{2\tau_{k-1}} \norm{x^{k} 
			-x^{k-1}}^{2}\\
		&+\innp{K(x^{k+1}-x^{k}),y^{k+1}-y^{*}}
		-\xi \innp{K(x^{k}-x^{k-1}),y^{k}-y^{*}}.
	\end{align*}
		
		\item \label{prop:inequalityx*y*K:sum}  Let  $\eta_{4}$ be in 
		$\mathbf{R}_{++}$ such that $(\forall k \in \mathbf{N})$
		$\frac{1}{\tau_{k}} -\frac{1}{\eta_{4}} \geq 0$, 
		and let $N$ be in $\mathbf{N}$. Then
		\begin{align*}
			&\xi^{N+1} \left( \frac{1}{  \tau_{0}} \norm{x^{*} -x^{0}}^{2} +\frac{1}{ \sigma_{0}} 
			\norm{y^{*} -y^{0}}^{2} \right)\\	
			\geq  & \frac{1}{ \tau_{N+1}}\norm{x^{*} -x^{N+1}}^{2}
			+ \left(\frac{1}{\sigma_{N+1}} -\xi \eta_{4} \norm{K}^{2} \right)\norm{y^{*} 
				-y^{N+1}}^{2}.
		\end{align*}
	\end{enumerate}
\end{proposition}	

\begin{proof}
	\cref{prop:inequalityx*y*K:k}: Employing \cref{lemma:inequalityx*y*} in the first 
	inequality below, and 
	applying \cref{lemma:inequalityCS}\cref{lemma:inequalityCS:K} 
	with $(x,y) =(x^{*},y^{*})$, we observe that 
		\begin{align*}
		&\frac{1}{2 \tau_{k}} \norm{x^{*} -x^{k}}^{2} +\frac{1}{2\sigma_{k}} \norm{y^{*} 
			-y^{k}}^{2}\\
		\geq &\left(\mu +\frac{1}{2\tau_{k}}\right)\norm{x^{*} -x^{k+1}}^{2}+\left( 
		\nu +\frac{1}{2 \sigma_{k}}\right) \norm{y^{*} -y^{k+1}}^{2}
		+ \frac{1}{2\tau_{k}} \norm{x^{k} -x^{k+1}}^{2} 
		+ \frac{1}{2 \sigma_{k}} \norm{y^{k} -y^{k+1}}^{2}\\
		&+\innp{K(x^{k+1}-\bar{x}^{k}),y^{k+1}-y^{*}}
		-\innp{K(x^{k+1}-x^{*}),y^{k+1}-\bar{y}^{k+1}}\\
		\geq &\left(\mu +\frac{1}{2\tau_{k}}\right)\norm{x^{*} -x^{k+1}}^{2}+\left( 
		\nu +\frac{1}{2 \sigma_{k}}\right) \norm{y^{*} -y^{k+1}}^{2}
		+ \frac{1}{2\tau_{k}} \norm{x^{k} -x^{k+1}}^{2} 
		+ \frac{1}{2 \sigma_{k}} \norm{y^{k} -y^{k+1}}^{2}\\
		&+\innp{K(x^{k+1}-x^{k}),y^{k+1}-y^{*}}
		-\xi \innp{K(x^{k}-x^{k-1}),y^{k}-y^{*}}\\
		&-\frac{ \norm{K}}{2}\left(\eta_{1}\xi^{2}\norm{x^{k}-x^{k-1}}^{2}
		+\frac{\norm{y^{k+1}-y^{k} }^{2}}{\eta_{1}}\right) \\
		&-\frac{\norm{K}}{2} \left( \eta_{2} \left(\xi -\alpha_{k-1} 
		\right)^{2}\norm{y^{k+1}-y^{*}}^{2} 
		+\frac{\norm{x^{k}-x^{k-1}}^{2}}{\eta_{2}}
		\right)\\
		&-\frac{\norm{K}}{2} \left( \eta_{3}\beta_{k}^{2}\norm{x^{k+1}-x^{*}}^{2} 
		+\frac{\norm{y^{k+1}-y^{k}}^{2} }{\eta_{3}}  \right)\\
=& \left(\mu +\frac{1}{2\tau_{k}}  -\frac{\norm{K}}{2} \eta_{3}\beta_{k}^{2} \right)
\norm{x^{k+1}-x^{*}}^{2}
+\left( \nu +\frac{1}{2 \sigma_{k}} -\frac{\norm{K}}{2}\eta_{2} \left(\xi -\alpha_{k-1} 
\right)^{2} \right) \norm{y^{*} -y^{k+1}}^{2}\\
&+ \frac{1}{2\tau_{k}} \norm{x^{k} -x^{k+1}}^{2} +\innp{K(x^{k+1}-x^{k}),y^{k+1}-y^{*}}
-\xi \innp{K(x^{k}-x^{k-1}),y^{k}-y^{*}}\\
&-\frac{ \norm{K}}{2}\left(\eta_{1}\xi^{2}
+ \frac{1}{\eta_{2}}\right) \norm{x^{k}-x^{k-1}}^{2}
	+ \frac{1}{2} \left( \frac{1}{ \sigma_{k}} - \norm{K} \left(\frac{1}{\eta_{1}} 
	+\frac{1}{\eta_{3}} \right)\right)  \norm{y^{k} -y^{k+1}}^{2}\\
\geq & \frac{1}{2 \tau_{k+1} \xi} \norm{x^{*}-x^{k+1}}^{2}
+\frac{1}{2 \sigma_{k+1} \xi}	 \norm{y^{*} -y^{k+1}}^{2}  
+ \frac{1}{2\tau_{k}} \norm{x^{k} -x^{k+1}}^{2} - \frac{\xi}{2\tau_{k-1}} \norm{x^{k} 
-x^{k-1}}^{2}\\
&+\innp{K(x^{k+1}-x^{k}),y^{k+1}-y^{*}}
-\xi \innp{K(x^{k}-x^{k-1}),y^{k}-y^{*}}.
	\end{align*}

\cref{prop:inequalityx*y*K:sum}:  The result obtained in  \cref{prop:inequalityx*y*K:k} 
above and  \cref{lemma:ksum} guarantee the required result. 
\end{proof}

\begin{proposition} \label{prop:inequalityx*y*}
	Let $(x^{*},y^{*})$ be a saddle-point of $f$, 
	and let $\xi$, $\eta_{1}$, $\eta_{2}$, and $\eta_{3}$ be in 
	$\mathbf{R}_{++}$. 
	Suppose that $(\forall k \in \mathbf{N})$
	$\mu +\frac{1}{2\tau_{k}}-\frac{\eta_{3}\norm{K}^{2}\beta_{k}^{2}}{2}
	\geq \frac{1}{2 \tau_{k+1}\xi}$,
	$\nu +\frac{1}{2 \sigma_{k}} -\frac{\eta_{2} \norm{K}^{2}\left(\xi -\alpha_{k-1} 
		\right)^{2}}{2} \geq \frac{1}{2 \sigma_{k+1}\xi}$,
	$\eta_{1} \norm{K}^{2}\xi^{2}+\frac{1}{\eta_{2}} \leq \frac{\xi}{\tau_{k-1}}$, and
	$\frac{1}{\sigma_{k}}  -\frac{1}{\eta_{1}} -\frac{1}{\eta_{3}} \geq 0$.
	Then we have the following statements. 
	\begin{enumerate}
		\item  \label{prop:inequalityx*y*:k} For every $k$ in $\mathbf{N}$, we have that 
		\begin{align*}
			&\frac{1}{2 \tau_{k}} \norm{x^{*} -x^{k}}^{2} +\frac{1}{2\sigma_{k}} \norm{y^{*} 
				-y^{k}}^{2}\\
			\geq &\frac{1}{2\tau_{k+1}\xi}\norm{x^{*} -x^{k+1}}^{2}
			+\frac{1}{2\sigma_{k+1}\xi}\norm{y^{*} -y^{k+1}}^{2}
			+ \frac{1}{2\tau_{k}} \norm{x^{k} -x^{k+1}}^{2}
			- \frac{\xi}{2\tau_{k-1}} \norm{x^{k} -x^{k+1}}^{2}\\
			&+\innp{K(x^{k+1}-x^{k}),y^{k+1}-y^{*}}
			-\xi \innp{K(x^{k}-x^{k-1}),y^{k}-y^{*}}.
		\end{align*}
		
		\item \label{prop:inequalityx*y*:sum}
		Let  $\eta_{4}$ be in 
		$\mathbf{R}_{++}$ such that $(\forall k \in \mathbf{N})$
		$\frac{1}{\tau_{k}} -\frac{1}{\eta_{4}} \geq 0$, 
		and let $N$ be in $\mathbf{N}$. Then
		\begin{align*}
			&\xi^{N+1} \left( \frac{1}{  \tau_{0}} \norm{x^{*} -x^{0}}^{2} 
			+\frac{1}{ \sigma_{0}} 
			\norm{y^{*} -y^{0}}^{2} \right)\\	
			\geq  & \frac{1}{ \tau_{N+1}}\norm{x^{*} -x^{N+1}}^{2}
			+ \left(\frac{1}{\sigma_{N+1}} -\xi \eta_{4} \norm{K}^{2} 
			\right)\norm{y^{*} 
				-y^{N+1}}^{2}.
		\end{align*}
	\end{enumerate}
\end{proposition}

\begin{proof}
	\cref{prop:inequalityx*y*:k}: 
	Applying \cref{lemma:inequalityx*y*} in the first inequality and adopting 
	\cref{lemma:inequalityCS}\cref{lemma:inequalityCS:Ksquare} with 
	$(x,y)=(x^{*},y^{*})$ 
	in the second inequality,
	we observe that for every $k$ in $\mathbf{N}$
	\begin{align*}
		&\frac{1}{2 \tau_{k}} \norm{x^{*} -x^{k}}^{2} +\frac{1}{2\sigma_{k}} \norm{y^{*} 
			-y^{k}}^{2}\\
		\geq &\left(\mu +\frac{1}{2\tau_{k}}\right)\norm{x^{*} -x^{k+1}}^{2}+\left( 
		\nu +\frac{1}{2 \sigma_{k}}\right) \norm{y^{*} -y^{k+1}}^{2}
		+ \frac{1}{2\tau_{k}} \norm{x^{k} -x^{k+1}}^{2} 
		+ \frac{1}{2 \sigma_{k}} \norm{y^{k} -y^{k+1}}^{2}\\
		&+\innp{K(x^{k+1}-\bar{x}^{k}),y^{k+1}-y^{*}}
		-\innp{K(x^{k+1}-x^{*}),y^{k+1}-\bar{y}^{k+1}}\\
		\geq &\left(\mu +\frac{1}{2\tau_{k}}\right)\norm{x^{*} -x^{k+1}}^{2}+\left( 
		\nu +\frac{1}{2 \sigma_{k}}\right) \norm{y^{*} -y^{k+1}}^{2}
		+ \frac{1}{2\tau_{k}} \norm{x^{k} -x^{k+1}}^{2} 
		+ \frac{1}{2 \sigma_{k}} \norm{y^{k} -y^{k+1}}^{2}\\
		&+\innp{K(x^{k+1}-x^{k}),y^{k+1}-y^{*}}
		-\xi \innp{K(x^{k}-x^{k-1}),y^{k}-y^{*}}\\
		&-\frac{1}{2}\left(\eta_{1} \norm{K}^{2}\xi^{2}\norm{x^{k}-x^{k-1}}^{2}
		+\frac{\norm{y^{k+1}-y^{k} }}{\eta_{1}}^{2}\right) \\
		&-\frac{1}{2} \left( \eta_{2} \norm{K}^{2}\left(\xi -\alpha_{k-1} 
		\right)^{2}\norm{y^{k+1}-y^{*}}^{2} 
		+\frac{\norm{x^{k}-x^{k-1}}^{2}}{\eta_{2}}
		\right)\\
		&-\frac{1}{2} \left( \eta_{3}\norm{K}^{2}\beta_{k}^{2}\norm{x^{k+1}-x^{*}}^{2} 
		+\frac{\norm{y^{k+1}-y^{k}}^{2} }{\eta_{3}}  \right)\\
		=&\left(\mu +\frac{1}{2\tau_{k}}-\frac{\eta_{3}\norm{K}^{2}\beta_{k}^{2}}{2}\right)
		\norm{x^{*} -x^{k+1}}^{2}
		+\left( \nu +\frac{1}{2 \sigma_{k}} -\frac{\eta_{2} \norm{K}^{2}\left(\xi -\alpha_{k-1} 
			\right)^{2}}{2}\right) 
		\norm{y^{*} -y^{k+1}}^{2}\\
		&+\innp{K(x^{k+1}-x^{k}),y^{k+1}-y^{*}}
		-\xi \innp{K(x^{k}-x^{k-1}),y^{k}-y^{*}}
		+ \frac{1}{2\tau_{k}} \norm{x^{k} -x^{k+1}}^{2}\\
		&-\frac{1}{2}\left( \eta_{1} \norm{K}^{2}\xi^{2}+\frac{1}{\eta_{2}}  \right) \norm{x^{k} 
			-x^{k-1}}^{2}		
		+\frac{1}{2} \left( \frac{1}{\sigma_{k}}  -\frac{1}{\eta_{1}} -\frac{1}{\eta_{3}}\right) 
		\norm{y^{k} -y^{k+1}}^{2}\\
		\geq &\frac{1}{2\tau_{k+1}\xi}\norm{x^{*} -x^{k+1}}^{2}
		+\frac{1}{2\sigma_{k+1}\xi}\norm{y^{*} -y^{k+1}}^{2}
		+ \frac{1}{2\tau_{k}} \norm{x^{k} -x^{k+1}}^{2}
		- \frac{\xi}{2\tau_{k-1}} \norm{x^{k} -x^{k-1}}^{2}\\
		&+\innp{K(x^{k+1}-x^{k}),y^{k+1}-y^{*}}
		-\xi \innp{K(x^{k}-x^{k-1}),y^{k}-y^{*}}.
	\end{align*}
	
	\cref{prop:inequalityx*y*:sum}: This is clear from the result obtained in  
	 \cref{prop:inequalityx*y*:k}
	above and \cref{lemma:ksum}.
\end{proof}

Recall that in our algorithm \cref{eq:fact:algorithmsymplity} generating the sequence of 
iterations
$\left((x^{k},y^{k})\right)_{k \in \mathbf{N}}$, the sequences $(\tau_{k})_{k \in 
	\mathbf{N} 
	\cup \{-1\}} \in \mathbf{R}_{++}$, 
$(\sigma_{k})_{k \in \mathbf{N} \cup \{-1\}} \in \mathbf{R}_{++}$, 
$(\alpha_{k})_{k \in \mathbf{N} \cup \{-1\}} \in \mathbf{R}_{+}$,
and $(\beta_{k})_{k \in \mathbf{N} \cup \{-1\}} \in \mathbf{R}_{+}$
are customized parameters by users of the algorithm.

\begin{theorem}  \label{theorem:linearconverg:K}
	Let $(x^{*},y^{*})$ be a saddle-point of $f$.
	Suppose $\inf_{k \in 	\mathbf{N}}\tau_{k} >0$ and $\inf_{k \in 	
	\mathbf{N}}\sigma_{k} >0$. 
	Let $\xi$ be in $(0,1)$ such that 
	$\xi \sup_{k \in \mathbf{N}} \tau_{k} \sup_{k \in \mathbf{N}} \sigma_{k} \norm{K}^{2}< 
	1$.
	Let $\eta_{1}$, $\eta_{2}$, $\eta_{3}$, and $\eta_{4}$ be in 
	$\mathbf{R}_{++}$ such that 
	$\norm{K} \left( \eta_{1} \xi^{2} +\frac{1}{\eta_{2}}\right) \leq \frac{\xi}{\sup_{i \in 
	\mathbf{N}}\tau_{i}}$,
	$\norm{K} \left(\frac{1}{\eta_{1}} +\frac{1}{\eta_{3}}\right) \leq \frac{1}{\sup_{i \in 
	\mathbf{N}}\sigma_{i}}$, $\eta_{4} \geq \sup_{i \in \mathbf{N}} \tau_{i}$,
	and $\frac{1}{\sup_{i \in \mathbf{N}}\sigma_{i}} -\xi \eta_{4} \norm{K}^{2} >0$.
	Suppose that $\mu $ and $\nu$ in $\mathbf{R}_{++}$ satisfying 
	$
	2\mu\geq \frac{1}{\xi \inf_{k \in \mathbf{N}} 
		\tau_{k}} -\frac{1}{\sup_{k \in \mathbf{N}}\tau_{k}}+  \eta_{3} \norm{K} \sup_{k 
		\in \mathbf{N}}\beta_{k}^{2}
	$
	and 
	$
	2 \nu \geq \frac{1}{\xi \inf_{k \in 
			\mathbf{N}}\sigma_{k}} -\frac{1}{\sup_{k \in \mathbf{N}}\sigma_{k}}  
	+ \eta_{2}\norm{K}  \sup_{k \in \mathbf{N}}(\xi -\alpha_{k})^{2}.
	$
	
	Then we have that  for every $k \in \mathbf{N}$,
	\begin{align*}
		&\norm{x^{k+1}-x^{*}} 
		\leq \xi^{k+1} \sup_{i \in \mathbf{N}}\tau_{i} 
		\left(\frac{1}{ \tau_{0}} \norm{x^{*} -x^{0}}^{2} +\frac{1}{\sigma_{0}} 
		\norm{y^{*} -y^{0}}^{2}\right);\\
		&\norm{y^{k+1}-y^{*}} 
		\leq \xi^{k+1} \left(\frac{1}{\sup_{i\in \mathbf{N}} \sigma_{i}} -\xi \eta_{4} 
		\norm{K}^{2}\right)
		\left(\frac{1}{ \tau_{0}} \norm{x^{*} -x^{0}}^{2} +\frac{1}{\sigma_{0}} 
		\norm{y^{*} -y^{0}}^{2}\right).
	\end{align*}
	Consequently,  the sequence of iterations $\left((x^{k},y^{k})\right)_{k \in 
		\mathbf{N}}$ 
	linearly converges to the saddle-point $(x^{*},y^{*})$ of $f$ 
	with a linear rate $\xi \in (0,1)$. 
\end{theorem}

\begin{proof}
	Because $\xi \sup_{k \in \mathbf{N}} \tau_{k} \sup_{k \in 
		\mathbf{N}} \sigma_{k} \norm{K}^{2}< 1$, 
we are able to take $\eta_{4} \in \mathbf{R}_{++}$ such that
$0<\sup_{k \in \mathbf{N}}\tau_{k} \leq \eta_{4}$ and $\eta_{4}\xi \norm{K}^{2} 
<\frac{1}{ \sup_{k \in \mathbf{N}}\sigma_{k}}$, which leads to
\begin{align*}
\frac{1}{\sup_{i\in \mathbf{N}} \sigma_{i}} -\xi \eta_{4} \norm{K}^{2}>0
\quad	\text{and} \quad
(\forall k \in \mathbf{N}) \quad  \frac{1}{\tau_{k}} -\frac{1}{\eta_{4}} \geq 0.
\end{align*}
Moreover, our assumptions 
$
2\mu\geq \frac{1}{\xi \inf_{k \in \mathbf{N}} 
\tau_{k}} -\frac{1}{\sup_{k \in \mathbf{N}}\tau_{k}}+  \eta_{3} \norm{K} \sup_{k 
\in \mathbf{N}}\beta_{k}^{2}
$
and 
$
2 \nu \geq \frac{1}{\xi \inf_{k \in 
	\mathbf{N}}\sigma_{k}} -\frac{1}{\sup_{k \in \mathbf{N}}\sigma_{k}}  
+ \eta_{2}\norm{K}  \sup_{k \in \mathbf{N}}(\xi -\alpha_{k})^{2}
$
clearly implies that 
\begin{align*}
(\forall k \in \mathbf{N}) \quad
\mu +\frac{1}{2\tau_{k}}-\frac{\eta_{3}\norm{K} \beta_{k}^{2}}{2}
\geq \frac{1}{2 \tau_{k+1}\xi}
\quad 
\text{and}
\quad 
\nu +\frac{1}{2 \sigma_{k}} -\frac{\eta_{2} \norm{K} \left(\xi -\alpha_{k-1} 
	\right)^{2}}{2} \geq \frac{1}{2 \sigma_{k+1}\xi}.	 
\end{align*}
In addition, it is clear that our assumptions
$\norm{K} \left( \eta_{1} \xi^{2} +\frac{1}{\eta_{2}}\right) \leq \frac{\xi}{\sup_{i \in 
		\mathbf{N}}\tau_{i}}$, and
$\norm{K} \left(\frac{1}{\eta_{1}} +\frac{1}{\eta_{3}}\right) \leq \frac{1}{\sup_{i \in 
		\mathbf{N}}\sigma_{i}}$
ensures
\begin{align*}
	(\forall k \in \mathbf{N}) \quad     \norm{K} \left(\eta_{1}\xi^{2}
	+ \frac{1}{\eta_{2}}\right) \leq \frac{\xi}{\tau_{k-1}} 
 \quad \text{and} \quad 
 \frac{1}{ \sigma_{k}} - \norm{K} \left(\frac{1}{\eta_{1}} +\frac{1}{\eta_{3}} \right) \geq 0. 
\end{align*}

	Hence, all assumptions in \cref{prop:inequalityx*y*K}\cref{prop:inequalityx*y*K:sum}
	 hold in our case. 
	Bearing \cref{prop:inequalityx*y*K}\cref{prop:inequalityx*y*K:sum} in mind,
	 we obtain that for every $k \in \mathbf{N}$,
	\begin{align*}
		&\xi^{k+1} \left( \frac{1}{\tau_{0}} \norm{x^{*} -x^{0}}^{2} +\frac{1}{\sigma_{0}} 
		\norm{y^{*} -y^{0}}^{2} \right)\\
		\geq&  \frac{1}{\tau_{k+1}}\norm{x^{*} -x^{k+1}}^{2}
		+ \left(\frac{1}{\sigma_{k+1}} -\xi \eta_{4} \norm{K}^{2} \right)\norm{y^{*} 
			-y^{k+1}}^{2}\\
		\geq &\frac{1}{\sup_{i \in 	\mathbf{N}} \tau_{i}}\norm{x^{*} -x^{k+1}}^{2}
		+ \left( \frac{1}{\sup_{i \in \mathbf{N}}\sigma_{i}} -\xi \eta_{4} \right) \norm{K}^{2}
		\norm{y^{*} -y^{k+1}}^{2},
	\end{align*}
	which, combined with the assumptions $
	0 < \inf_{k \in 	\mathbf{N}}\tau_{k} \leq \sup_{k \in 	\mathbf{N}}\tau_{k} \leq \eta_{4} 
	<\infty$ 
	and $\frac{1}{\sup_{i \in \mathbf{N}}\sigma_{i}} -\xi \eta_{4} \norm{K}^{2} >0$,
	yields the required linearly convergent result.
\end{proof}

\begin{theorem} \label{theorem:linearconverg:Ksquare}
	Let $(x^{*},y^{*})$ be a saddle-point of $f$.
	Suppose $\inf_{k \in 	\mathbf{N}}\tau_{k} >0$ and $\inf_{k \in 	
		\mathbf{N}}\sigma_{k} >0$.  
	Let $\xi$ be in $(0,1)$ 
	such that $\xi \sup_{k \in \mathbf{N}} \tau_{k} \sup_{k \in 
		\mathbf{N}} \sigma_{k} \norm{K}^{2}< 1$.
	Let $\eta_{1}$, $\eta_{2}$, $\eta_{3}$, and $\eta_{4}$ be in 
	$\mathbf{R}_{++}$ such that 
	$\eta_{1} \norm{K}^{2}\xi +\frac{1}{\eta_{2}\xi} \leq \frac{1}{\sup_{k \in 
			\mathbf{N}}\tau_{k}}$, 
	$\frac{1}{\eta_{1}} +\frac{1}{\eta_{3}} \leq \frac{1}{\sup_{k \in \mathbf{N}} \sigma_{k}}$,
	$ \sup_{k \in \mathbf{N}}\tau_{k} \leq \eta_{4}$ and $\eta_{4}\xi \norm{K}^{2} 
	<\frac{1}{ \sup_{k \in \mathbf{N}}\sigma_{k}}$.
	Suppose that $\mu $ and $\nu$ in $\mathbf{R}_{++}$ satisfying 
	$
	2\mu\geq \frac{1}{\xi \inf_{k \in \mathbf{N}} 
		\tau_{k}} -\frac{1}{\sup_{k \in \mathbf{N}}\tau_{k}}+  \eta_{3} \norm{K}^{2}\sup_{k 
		\in \mathbf{N}}\beta_{k}^{2}
	$
	and 
	$
	2 \nu \geq \frac{1}{\xi \inf_{k \in 
			\mathbf{N}}\sigma_{k}} -\frac{1}{\sup_{k \in \mathbf{N}}\sigma_{k}}  
	+ \eta_{2}\norm{K}^{2} \sup_{k \in \mathbf{N}}(\xi -\alpha_{k})^{2}.
	$
	
	Then we have that  for every $k \in \mathbf{N}$,
	\begin{align*}
		&\norm{x^{k+1}-x^{*}} 
		\leq \xi^{k+1} \sup_{i \in \mathbf{N}}\tau_{i} 
		\left(\frac{1}{ \tau_{0}} \norm{x^{*} -x^{0}}^{2} +\frac{1}{\sigma_{0}} 
		\norm{y^{*} -y^{0}}^{2}\right);\\
		&\norm{y^{k+1}-y^{*}} 
		\leq \xi^{k+1} \left(\frac{1}{\sup_{i\in \mathbf{N}} \sigma_{i}} -\xi \eta_{4} 
		\norm{K}^{2}\right)
		\left(\frac{1}{ \tau_{0}} \norm{x^{*} -x^{0}}^{2} +\frac{1}{\sigma_{0}} 
		\norm{y^{*} -y^{0}}^{2}\right).
	\end{align*}
	Consequently,  the sequence of iterations $\left((x^{k},y^{k})\right)_{k \in 
		\mathbf{N}}$ 
converges 	linearly  to the saddle-point $(x^{*},y^{*})$ of $f$ 
	with a  rate $\xi \in (0,1)$. 
\end{theorem}

\begin{proof}
	Based on our assumptions $\xi >0$,
	$\eta_{1} \norm{K}^{2}\xi +\frac{1}{\eta_{2}\xi} \leq \frac{1}{\sup_{k \in 
			\mathbf{N}}\tau_{k}}$,
	and
	$\frac{1}{\eta_{1}} +\frac{1}{\eta_{3}} \leq \frac{1}{\sup_{k \in \mathbf{N}} \sigma_{k}}$,
	we observe that 
	\begin{align*}
		(\forall k \in \mathbf{N}) \quad  \eta_{1} \norm{K}^{2}\xi^{2}+\frac{1}{\eta_{2}} \leq 
		\frac{\xi}{\tau_{k-1}} \quad \text{and} \quad 
		\frac{1}{\sigma_{k}}  -\frac{1}{\eta_{1}} -\frac{1}{\eta_{3}} \geq 0. 
	\end{align*}
	 Because $\xi \sup_{k \in \mathbf{N}} \tau_{k} \sup_{k \in 
			\mathbf{N}} \sigma_{k} \norm{K}^{2}< 1$, 
	we are able to take $\eta_{4} \in \mathbf{R}_{++}$ such that
	$0<\inf_{k \in \mathbf{N}}\tau_{k} \leq \sup_{k \in \mathbf{N}}\tau_{k} \leq \eta_{4}$ 
	and $\eta_{4}\xi \norm{K}^{2} 
	<\frac{1}{ \sup_{k \in \mathbf{N}}\sigma_{k}}$, leads to 
	\begin{align*}
		\frac{1}{\sup_{i\in \mathbf{N}} \sigma_{i}} -\xi \eta_{4} \norm{K}^{2}>0
		\quad	\text{and} \quad
		(\forall k \in \mathbf{N}) \quad  \frac{1}{\tau_{k}} -\frac{1}{\eta_{4}} \geq 0.
	\end{align*}
	Furthermore, our assumptions 
	$
	2\mu\geq \frac{1}{\xi \inf_{k \in \mathbf{N}} 
		\tau_{k}} -\frac{1}{\sup_{k \in \mathbf{N}}\tau_{k}}+  \eta_{3} \norm{K}^{2}\sup_{k 
		\in \mathbf{N}}\beta_{k}^{2}
	$
	and 
	$
	2 \nu \geq \frac{1}{\xi \inf_{k \in 
			\mathbf{N}}\sigma_{k}} -\frac{1}{\sup_{k \in \mathbf{N}}\sigma_{k}}  
	+ \eta_{2}\norm{K}^{2} \sup_{k \in \mathbf{N}}(\xi -\alpha_{k})^{2}
	$
	clearly guarantee  that 
	\begin{align*}
		(\forall k \in \mathbf{N}) \quad
		\mu +\frac{1}{2\tau_{k}}-\frac{\eta_{3}\norm{K}^{2}\beta_{k}^{2}}{2}
		\geq \frac{1}{2 \tau_{k+1}\xi}
		\quad 
		\text{and}
		\quad 
		\nu +\frac{1}{2 \sigma_{k}} -\frac{\eta_{2} \norm{K}^{2}\left(\xi -\alpha_{k-1} 
			\right)^{2}}{2} \geq \frac{1}{2 \sigma_{k+1}\xi}.	 
	\end{align*}
	Therefore, all assumptions in \cref{prop:inequalityx*y*}\cref{prop:inequalityx*y*:sum}
	 are satisfied. 
	Due to \cref{prop:inequalityx*y*}\cref{prop:inequalityx*y*:sum}, we obtain that 
	for every $k \in \mathbf{N}$,
	\begin{align*}
		&\xi^{k+1} \left( \frac{1}{\tau_{0}} \norm{x^{*} -x^{0}}^{2} +\frac{1}{\sigma_{0}} 
		\norm{y^{*} -y^{0}}^{2} \right)\\
		\geq&  \frac{1}{\tau_{k+1}}\norm{x^{*} -x^{k+1}}^{2}
		+ \left(\frac{1}{\sigma_{k+1}} -\xi \eta_{4} \norm{K}^{2} \right)\norm{y^{*} 
			-y^{k+1}}^{2}\\
		\geq &\frac{1}{\sup_{i \in 	\mathbf{N}} \tau_{i}}\norm{x^{*} -x^{k+1}}^{2}
		+ \left(\frac{1}{\sup_{i\in \mathbf{N}} \sigma_{i}} -\xi \eta_{4} 
		\norm{K}^{2}\right)
		\norm{y^{*} 
			-y^{k+1}}^{2},
	\end{align*}
	which, combined with the assumptions $0<\inf_{k \in \mathbf{N}}\tau_{k} \leq \sup_{k 
	\in \mathbf{N}}\tau_{k} \leq \eta_{4} <\infty$ 
	and $ \frac{1}{\sup_{i\in \mathbf{N}} \sigma_{i}} -\xi \eta_{4} 
	\norm{K}^{2} >0$,
	ensures the desired linearly convergent result.
\end{proof}

To apply our linear convergence results in 
\cref{theorem:linearconverg:K,theorem:linearconverg:Ksquare}  easily 
in practice, we present the following results to simplify assumptions in
\cref{theorem:linearconverg:K,theorem:linearconverg:Ksquare} by setting all involved 
parameters 
$(\tau_{k})_{k \in 
	\mathbf{N} 
	\cup \{-1\}} \in \mathbf{R}_{++}$, 
$(\sigma_{k})_{k \in \mathbf{N} \cup \{-1\}} \in \mathbf{R}_{++}$, 
$(\alpha_{k})_{k \in \mathbf{N} \cup \{-1\}} \in \mathbf{R}_{+}$,
and $(\beta_{k})_{k \in \mathbf{N} \cup \{-1\}} \in \mathbf{R}_{+}$
to be constants $\tau, \sigma, \alpha$, and $\beta$, respectively.
In the following results, as in practical situations,  $\mu >0$, $\nu >0$, $\norm{K} \geq 
0$ are known. 
We present assumptions about the involved parameters 
$\tau, \sigma, \alpha$, and $\beta$ 
of our algorithms and about the extra parameters $\eta_{1}$, $\eta_{2}$,
$\eta_{3}$, and $\eta_{4}$ 
used in the  proof of \cref{theorem:linearconverg:K,theorem:linearconverg:Ksquare}
to show the linear convergence of our algorithms.

\begin{corollary} \label{corollary:linearconverg:K}
 Let $(x^{*},y^{*})$ be a saddle-point of $f$. Suppose that 
\begin{align}  \label{eq:corollary:linearconverg:K:constant}
	(\forall k \in \mathbf{N} \cup \{-1\}) \quad \tau_{k} \equiv \tau \in \mathbf{R}_{++}, 
	\sigma_{k} \equiv  \sigma \in \mathbf{R}_{++}, 
	\alpha_{k} \equiv \alpha \in \mathbf{R}_{++}, \text{ and } 
	\beta_{k} \equiv \beta \in \mathbf{R}_{+}.
\end{align}
Recall that $\mu>0$, $\nu>0$, $\norm{K} \geq 0$ are known. 
Let $\tau >0$ and $\sigma >0$ be small enough such that 
$\norm{K}^{2} <\frac{2\mu +\frac{1}{\tau}}{\sigma}$ and 
$\norm{K}^{2} < \frac{2 \nu +\frac{1}{\sigma}}{\tau}$. 
Let $\alpha$ be in $\left( \frac{1}{ 2 \mu \tau +1} , 1\right)$ such that $\alpha \geq 
\frac{1}{2 \nu \sigma +1}$, and $\alpha \tau \sigma \norm{K}^{2} <1$.
Let $\beta$ satisfy that 
$\sigma \norm{K}^{2}\beta^{2} \leq 
\left(1-\alpha \tau \sigma \norm{K}^{2}\right)
\left(2\mu +\frac{1}{\tau} -\frac{1}{\alpha \tau}\right)$.
Let  $\eta_{4}$ be in $\mathbf{R}_{++}$ such that 
$\eta_{4} \geq \tau$ and $\alpha \sigma \eta_{4} \norm{K}^{2} < 1$.
Let  $\eta_{3}$ be in $\mathbf{R}_{++}$ such that 
$\eta_{3} > \frac{\sigma \norm{K}}{ 1 -\alpha \tau \sigma \norm{K}^{2}}$ and
$\eta_{3}\norm{K} \beta^{2} < 2 \mu +\frac{1}{\tau} 
-\frac{1}{\alpha \tau}$.
Let $\eta_{1}$ be in $\mathbf{R}_{++}$ such that 
$\eta_{1} \geq \frac{\sigma \norm{K}\eta_{3}}{\eta_{3} -\sigma \norm{K}}$ and
$\alpha \tau \eta_{1}\norm{K} <1$.
Let  $\eta_{2}$ be in $\mathbf{R}_{++}$ such that 
$\eta_{2} \geq \frac{\tau \norm{K}}{\alpha - \tau \eta_{1} \norm{K}\alpha^{2}}$. 	

	Then we have that  for every $k \in \mathbf{N}$,
	\begin{align*}
		&\norm{x^{k+1}-x^{*}} 
		\leq \alpha^{k+1}  \tau 
		\left(\frac{1}{ \tau_{0}} \norm{x^{*} -x^{0}}^{2} +\frac{1}{\sigma_{0}} 
		\norm{y^{*} -y^{0}}^{2}\right);\\
		&\norm{y^{k+1}-y^{*}} 
		\leq \alpha^{k+1} \left(\frac{1}{\sigma} -\alpha \eta_{4} \norm{K}^{2} \right)
		\left(\frac{1}{ \tau_{0}} \norm{x^{*} -x^{0}}^{2} +\frac{1}{\sigma_{0}} 
		\norm{y^{*} -y^{0}}^{2}\right).
	\end{align*}
	Consequently,  the sequence of iterations $\left((x^{k},y^{k})\right)_{k \in 
		\mathbf{N}}$ 
	linearly converges to the saddle-point $(x^{*},y^{*})$ of $f$ 
	with a  rate $\alpha \in (0,1)$. 
\end{corollary}

\begin{proof}
	Set $\xi =\alpha$ in \cref{theorem:linearconverg:K}. 
	
	Applying  
	\cref{lemma:infsupconstants}\cref{lemma:infsupconstants:munu}$\&$\cref{lemma:infsupconstants:alpha}$\&$\cref{lemma:infsupconstants:beta}
	 with $\zeta =2$, we are able to take 
	  $\tau >0$ and $\sigma >0$  small enough satisfying 
	 $\norm{K}^{2} <\frac{2\mu +\frac{1}{\tau}}{\sigma}$ and 
	 $\norm{K}^{2} < \frac{2 \nu +\frac{1}{\sigma}}{\tau}$, 
	 take $\alpha\in \left( \frac{1}{ 2 \mu \tau +1} , 1\right) \subseteq (0,1)$ satisfying 
	 $\alpha \geq  \frac{1}{2 \nu \sigma +1}$ and $\alpha \tau \sigma \norm{K}^{2} <1$,
	 take $\beta$ satisfying
	 $\sigma \norm{K}^{2}\beta^{2} \leq 
	 \left(1-\alpha \tau \sigma \norm{K}^{2}\right)
	 \left(2\mu +\frac{1}{\tau} -\frac{1}{\alpha \tau}\right)$.
	Moreover, because $\alpha =\xi$, 
	the required condition 
	$\xi \in (0,1)$ and
	$\xi \sup_{k \in \mathbf{N}} \tau_{k} \sup_{k \in \mathbf{N}} \sigma_{k} 
	\norm{K}^{2}< 
	1$ in \cref{theorem:linearconverg:K} hold in our case. 
	
	Employing \cref{lemma:infsupconstants}\cref{lemma:infsupconstants:nu} 
	with $\zeta =2$ and $i=1$,  
	and using $\xi 
	=\alpha$,
	we satisfy
	$
	2 \nu \geq \frac{1}{\xi \inf_{k \in 
			\mathbf{N}}\sigma_{k}} -\frac{1}{\sup_{k \in \mathbf{N}}\sigma_{k}}  
	+ \eta_{2}\norm{K} \sup_{k \in \mathbf{N}}(\xi -\alpha_{k})^{2}
	$
	in \cref{theorem:linearconverg:K}.

	In addition, applying 
	\cref{lemma:infsupconstants}\cref{lemma:infsupconstants:eta4}$\&$\cref{lemma:infsupconstants:eta3:a}$\&$\cref{lemma:infsupconstants:eta1a}$\&$\cref{lemma:infsupconstants:eta2a}
	 with $\zeta =2$, 
	we are able to take $\eta_{4}$, $\eta_{3}$, $\eta_{1}$, and $\eta_{2}$ in
	 $\mathbf{R}_{++}$  such that 
	 $\eta_{4} \geq \tau$, $\alpha \sigma \eta_{4} \norm{K}^{2} < 1$,  
	 $\eta_{3} > \frac{\sigma \norm{K}}{ 1 -\alpha \tau \sigma \norm{K}^{2}}$,
	 $\eta_{3}\norm{K} \beta^{2} < 2 \mu +\frac{1}{\tau} 
	 -\frac{1}{\alpha \tau}$,
	 $\eta_{1} \geq \frac{\sigma \norm{K}\eta_{3}}{\eta_{3} -\sigma \norm{K}}$, 
	 $\alpha \tau \eta_{1}\norm{K} <1$, and
	 $\eta_{2} \geq \frac{\tau \norm{K}}{\alpha - \tau \eta_{1} \norm{K}\alpha^{2}}$;
	 moreover, our choices of $\eta_{4}$, $\eta_{3}$, $\eta_{1}$, and $\eta_{2}$ in
	 $\mathbf{R}_{++}$  satisfy the conditions 
	 $\eta_{4} \geq \sup_{i \in \mathbf{N}} \tau_{i}$,
	 $\frac{1}{\sup_{i \in \mathbf{N}}\sigma_{i}} -\xi \eta_{4} \norm{K}^{2} >0$,
	 $
	 2\mu\geq \frac{1}{\xi \inf_{k \in \mathbf{N}} 
	 	\tau_{k}} -\frac{1}{\sup_{k \in \mathbf{N}}\tau_{k}}+  \eta_{3} \norm{K} \sup_{k 
	 	\in \mathbf{N}}\beta_{k}^{2}
	 $,
	 $\norm{K} \left(\frac{1}{\eta_{1}} +\frac{1}{\eta_{3}}\right) \leq \frac{1}{\sup_{i \in 
	 		\mathbf{N}}\sigma_{i}}$, and
	 $\norm{K} \left( \eta_{1} \xi^{2} +\frac{1}{\eta_{2}}\right) \leq \frac{\xi}{\sup_{i \in 
	 		\mathbf{N}}\tau_{i}}$.

Altogether, our assumptions implies all requirements of \cref{theorem:linearconverg:K}. 
Therefore, 
employing \cref{theorem:linearconverg:K} with $\xi =\alpha$,
we derive the desired linear convergence result.  
\end{proof}

\begin{corollary} \label{corollary:linearconverg:Ksquare}
 Let $(x^{*},y^{*})$ be a saddle-point of $f$. Suppose that 
 \begin{align}  \label{eq:corollary:linearconverg:Ksquare:constant}
 	(\forall k \in \mathbf{N} \cup \{-1\}) \quad \tau_{k} \equiv \tau \in \mathbf{R}_{++}, 
 	\sigma_{k} \equiv  \sigma \in \mathbf{R}_{++}, 
 	\alpha_{k} \equiv \alpha \in \mathbf{R}_{++}, \text{ and } 
 	\beta_{k} \equiv \beta \in \mathbf{R}_{+}.
 \end{align}
Recall that $\mu>0$, $\nu>0$, $\norm{K} \geq 0$ are known. 
Let $\tau >0$ and $\sigma >0$ be small enough such that 
$\norm{K}^{2} <\frac{2\mu +\frac{1}{\tau}}{\sigma}$ and 
$\norm{K}^{2} < \frac{2 \nu +\frac{1}{\sigma}}{\tau}$. 
Let $\alpha$ be in $\left( \frac{1}{ 2 \mu \tau +1} , 1\right)$ such that $\alpha \geq 
\frac{1}{2 \nu \sigma +1}$, and $\alpha \tau \sigma \norm{K}^{2} <1$.
Let $\beta$ satisfy  $\norm{K}^{2} \beta^{2} \sigma < \left(2\mu 
+\frac{1}{\tau}-\frac{1}{\alpha 
\tau} \right)\left( 1-\alpha \tau \sigma \norm{K}^{2}\right)
$.
Let $\eta_{3}$ be in $\mathbf{R}_{++}$ such that 
$\eta_{3} >\frac{\sigma}{1-\alpha \tau \sigma \norm{K}^{2}}$
and $\norm{K}^{2}\beta^{2}\eta_{3} \leq 2\mu +\frac{1}{\tau} -\frac{1}{\alpha \tau}$.
Let $\eta_{4}$ be in $\mathbf{R}_{++}$ such that 
$\tau \leq \eta_{4}$ and $\alpha \sigma \eta_{4}\norm{K}^{2} <1$.
Let $\eta_{1}$ be in $\mathbf{R}_{++}$ such that 
$\eta_{1} \geq \frac{\sigma \eta_{3}}{\eta_{3} -\sigma}$
and $\alpha \tau \eta_{1}\norm{K}^{2} <1$.
Let $\eta_{2}$ be in $\mathbf{R}_{++}$ such that 
$\eta_{2} \geq \frac{\tau}{\alpha(1-\alpha \tau \eta_{1}\norm{K}^{2})}$.
Then we have that  for every $k \in \mathbf{N}$,
\begin{align*}
	&\norm{x^{k+1}-x^{*}} 
	\leq \alpha^{k+1}  \tau 
	\left(\frac{1}{ \tau_{0}} \norm{x^{*} -x^{0}}^{2} +\frac{1}{\sigma_{0}} 
	\norm{y^{*} -y^{0}}^{2}\right);\\
	&\norm{y^{k+1}-y^{*}} 
	\leq \alpha^{k+1} \left(\frac{1}{\sigma} -\alpha \eta_{4}\norm{K}^{2}\right)
	\left(\frac{1}{ \tau_{0}} \norm{x^{*} -x^{0}}^{2} +\frac{1}{\sigma_{0}} 
	\norm{y^{*} -y^{0}}^{2}\right).
\end{align*}
Consequently,  the sequence of iterations $\left((x^{k},y^{k})\right)_{k \in 
	\mathbf{N}}$ 
converges linearly  to the saddle-point $(x^{*},y^{*})$ of $f$ 
with a  rate $\alpha \in (0,1)$. 
\end{corollary}

\begin{proof}
	Consider  $\xi=\alpha$ in \cref{theorem:linearconverg:Ksquare}. 
	
	Similarly with the beginning part of the proof of \cref{corollary:linearconverg:K}, 
	we are able to take 
	$\tau >0$ and $\sigma >0$  small enough satisfying 
	$\norm{K}^{2} <\frac{2\mu +\frac{1}{\tau}}{\sigma}$ and 
	$\norm{K}^{2} < \frac{2 \nu +\frac{1}{\sigma}}{\tau}$, 
	take $\alpha\in \left( \frac{1}{ 2 \mu \tau +1} , 1\right) \subseteq (0,1)$ satisfying 
	$\alpha \geq  \frac{1}{2 \nu \sigma +1}$ and $\alpha \tau \sigma \norm{K}^{2} <1$,
	take $\beta$ satisfying
	$\sigma \norm{K}^{2}\beta^{2} \leq 
	\left(1-\alpha \tau \sigma \norm{K}^{2}\right)
	\left(2\mu +\frac{1}{\tau} -\frac{1}{\alpha \tau}\right)$.
	Moreover, using $\alpha =\xi$, we derive
	the condition 
	$\xi \in (0,1)$ and
	$\xi \sup_{k \in \mathbf{N}} \tau_{k} \sup_{k \in \mathbf{N}} \sigma_{k} 
	\norm{K}^{2}< 
	1$ required  in \cref{theorem:linearconverg:Ksquare} hold in our case.

	Applying \cref{lemma:infsupconstants}\cref{lemma:infsupconstants:nu} 
	with $\zeta =2$ and $i=2$,  
	and using $\xi 
	=\alpha$,
	we get
	$
	2 \nu \geq \frac{1}{\xi \inf_{k \in 
			\mathbf{N}}\sigma_{k}} -\frac{1}{\sup_{k \in \mathbf{N}}\sigma_{k}}  
	+ \eta_{2}\norm{K}^{2} \sup_{k \in \mathbf{N}}(\xi -\alpha_{k})^{2}
	$
	required in \cref{theorem:linearconverg:Ksquare}. 
	
	Furthermore, 
	employing 
	\cref{lemma:infsupconstants}\cref{lemma:infsupconstants:eta4}$\&$\cref{lemma:infsupconstants:eta3:b}$\&$\cref{lemma:infsupconstants:eta1b}$\&$\cref{lemma:infsupconstants:eta2b},
with $\zeta =2$, we are able to take 
$\eta_{4}$, $\eta_{3}$, $\eta_{1}$, and $\eta_{2}$ in
$\mathbf{R}_{++}$  such that 
$\tau \leq \eta_{4}$, $\alpha \sigma \eta_{4}\norm{K}^{2} <1$,
$\eta_{3} >\frac{\sigma}{1-\alpha \tau \sigma \norm{K}^{2}}$,
$\norm{K}^{2}\beta^{2}\eta_{3} \leq 2\mu +\frac{1}{\tau} -\frac{1}{\alpha \tau}$,
$\eta_{1} \geq \frac{\sigma \eta_{3}}{\eta_{3} -\sigma}$,
$\alpha \tau \eta_{1}\norm{K}^{2} <1$,
and
$\eta_{2} \geq \frac{\tau}{\alpha(1-\alpha \tau \eta_{1}\norm{K}^{2})}$,
which ensures the conditions
$ \sup_{k \in \mathbf{N}}\tau_{k} \leq \eta_{4}$,
 $\eta_{4}\xi \norm{K}^{2} <\frac{1}{ \sup_{k \in \mathbf{N}}\sigma_{k}}$,
$
2\mu\geq \frac{1}{\xi \inf_{k \in \mathbf{N}} 
	\tau_{k}} -\frac{1}{\sup_{k \in \mathbf{N}}\tau_{k}}+  \eta_{3} \norm{K}^{2}\sup_{k 
	\in \mathbf{N}}\beta_{k}^{2}
$,
and 
$\eta_{1} \norm{K}^{2}\xi +\frac{1}{\eta_{2}\xi} \leq \frac{1}{\sup_{k \in 
		\mathbf{N}}\tau_{k}}$
	required in \cref{theorem:linearconverg:Ksquare}. 
	
Altogether, all requirements in  \cref{theorem:linearconverg:Ksquare} hold 
in our case. 
Applying \cref{theorem:linearconverg:Ksquare} with $\xi =\alpha$,
we obtain the required linear convergence result. 
\end{proof}

\section{Linear Convergence of  Function Values} 
\label{section:linearconvergf}
From now on, let $((x^{k},y^{k}))_{k \in \mathbf{N}}$ be generated by the iterate 
scheme \cref{eq:fact:algorithmsymplity}, set $\xi$ being in $\mathbf{R}_{++}$, and 
define
\begin{align}  \label{eq:x_ky_k}
	(\forall k \in \mathbf{N}) 
	~ \hat{x}_{k} := \frac{1}{\sum^{k}_{j=0}\frac{1}{\xi^{j}}} \sum^{k}_{i=0} \frac{1}{\xi^{i}} 
	x^{i+1}    
\quad	\text{and}  \quad 
\hat{y}_{k} := \frac{1}{\sum^{k}_{j=0}\frac{1}{\xi^{j}}}  \sum^{k}_{i=0} \frac{1}{\xi^{i}} 
y^{i+1}.
\end{align}

In this section, we present linear convergence of 
function values evaluated over the sequence  $\left( \left(\hat{x}_{k}, \hat{y}_{k}  
\right)\right)_{k \in \mathbf{N}}$ 
presented in \cref{eq:x_ky_k} above.

\begin{lemma} \label{lemma:stronglyconvex}
	Let $(x,y)$ be in $\mathcal{H}_{1} \times \mathcal{H}_{2}$, 
	and let $k$ be in $\mathbf{N}$. 
	We have the following results. 
	\begin{enumerate}
		\item \label{lemma:stronglyconvex:x} 
		$f(\hat{x}_{k},y) -\frac{\mu}{2} \norm{\hat{x}_{k}}^{2} \leq 
		 \frac{1}{\sum^{k}_{j=0}\frac{1}{\xi^{j}}}  \sum^{k}_{i=0} \frac{1}{\xi^{i}} 
		 \left(f(x^{i+1},y) -\frac{\mu}{2} \norm{x^{i+1}}^{2}\right)$.  
		\item \label{lemma:stronglyconvex:y} 
		$-f(x, \hat{y}_{k}) -\frac{\nu}{2} \norm{ \hat{y}_{k}}^{2} \leq 
		\frac{1}{\sum^{k}_{j=0}\frac{1}{\xi^{j}}}  \sum^{k}_{i=0} \frac{1}{\xi^{i}} 
		\left(-f(x,y^{i+1}) -\frac{\nu}{2} \norm{y^{i+1}}^{2}\right)$.  
		\item \label{lemma:stronglyconvex:normx} 
		$ \frac{1}{\sum^{k}_{j=0}\frac{1}{\xi^{j}}}  \sum^{k}_{i=0} \frac{1}{\xi^{i}} 
		\norm{x^{i+1}}^{2} -
		\norm{\hat{x}_{k}}^{2} 
		= \frac{1}{\sum^{k}_{j=0}\frac{1}{\xi^{j}}}  \sum^{k}_{i=0} \frac{1}{\xi^{i}} 
		\norm{x^{i+1} - \hat{x}_{k}}^{2}
		$.
	\item \label{lemma:stronglyconvex:normy}
		$ \frac{1}{\sum^{k}_{j=0}\frac{1}{\xi^{j}}}  \sum^{k}_{i=0} \frac{1}{\xi^{i}} 
		\norm{y^{i+1}}^{2} -
		\norm{\hat{y}_{k}}^{2} 
		= \frac{1}{\sum^{k}_{j=0}\frac{1}{\xi^{j}}}  \sum^{k}_{i=0} \frac{1}{\xi^{i}} 
		\norm{y^{i+1} - \hat{y}_{k}}^{2}
		$.
		\item \label{lemma:stronglyconvex:ineq} 
		$f(\hat{x}_{k},y) - f(x, \hat{y}_{k})  
		\leq \frac{1}{\sum^{k}_{j=0}\frac{1}{\xi^{j}}}  \sum^{k}_{i=0} \frac{1}{\xi^{i}} 
		\left( f(x^{i+1},y)-f(x,y^{i+1})   \right) 
		- \frac{\mu}{2\sum^{k}_{j=0}\frac{1}{\xi^{j}}}  \sum^{k}_{i=0} \frac{1}{\xi^{i}} 
		\norm{x^{i+1} - \hat{x}_{k}}^{2}
		- \frac{\nu}{2\sum^{k}_{j=0}\frac{1}{\xi^{j}}}  \sum^{k}_{i=0} \frac{1}{\xi^{i}} 
		\norm{y^{i+1} - \hat{y}_{k}}^{2}
		\leq \frac{1}{\sum^{k}_{j=0}\frac{1}{\xi^{j}}}  
		\sum^{k}_{i=0} \frac{1}{\xi^{i}} 
		\left( f(x^{i+1},y)-f(x,y^{i+1})   \right) $.
	\end{enumerate}
\end{lemma}

\begin{proof}
	\cref{lemma:stronglyconvex:x}: Based on our assumption \cref{eq:gmu} in 
	\cref{assumption:basic}, we know that 
	$g$ is strongly convex with a modulus $\mu >0$.
	Hence, it is immediate to see that   the function 
	 $f(\cdot,y) =\innp{K(\cdot), y} + 
	g(\cdot)  -h(y)$ is strongly convex with a modulus $\mu >0$, which,
	via \cref{eq:x_ky_k}, 
	 ensures the required inequality in 	\cref{lemma:stronglyconvex:x}.
	
	\cref{lemma:stronglyconvex:y}: 
	Similarly with our proof of \cref{lemma:stronglyconvex:x},
	as presented in \cref{eq:hnu} in \cref{assumption:basic}, we know that 
	$h$ is strongly convex with a modulus $\nu >0$. 
	So we have that the function 
	$-f(x,\cdot) =-\innp{Kx, \cdot} - 
	g(x)  +h(\cdot)$ is strongly convex with a modulus $\nu >0$,
	which, due to \cref{eq:x_ky_k},  immediately derive the desired result in 
	\cref{lemma:stronglyconvex:y}.
	
	\cref{lemma:stronglyconvex:normx}$\&$\cref{lemma:stronglyconvex:normy}:
	To prove \cref{lemma:stronglyconvex:normx} and 
	\cref{lemma:stronglyconvex:normy},
	it suffices to prove \cref{lemma:stronglyconvex:normx} because replacing 
	$(\forall i \in \{0,1,\cdots, k\})$ $x^{i+1}$ in \cref{lemma:stronglyconvex:normx}
	by $y^{i+1}$,
	we get \cref{lemma:stronglyconvex:normy}.
	
	It is clear that
	\begin{align*}
		 &\frac{1}{\sum^{k}_{j=0}\frac{1}{\xi^{j}}}  \sum^{k}_{i=0} \frac{1}{\xi^{i}} 
		\norm{x^{i+1}}^{2} - \norm{\hat{x}_{k}}^{2}\\
		 =& 
		 \frac{1}{\sum^{k}_{j=0}\frac{1}{\xi^{j}}}  \sum^{k}_{i=0} \frac{1}{\xi^{i}} 
		\norm{x^{i+1}}^{2} - 2 \norm{\hat{x}_{k}}^{2} +
		\norm{\hat{x}_{k}}^{2} \\
		\stackrel{\cref{eq:x_ky_k}}{=}&
		\frac{1}{\sum^{k}_{j=0}\frac{1}{\xi^{j}}}  \sum^{k}_{i=0} \frac{1}{\xi^{i}} 
		\norm{x^{i+1}}^{2} -
		2 \innp{ \frac{1}{\sum^{k}_{j=0}\frac{1}{\xi^{j}}}  \sum^{k}_{i=0} \frac{1}{\xi^{i}} 
		x^{i+1},  \hat{x}_{k}} +  \frac{1}{\sum^{k}_{j=0}\frac{1}{\xi^{j}}}  \sum^{k}_{i=0} 
		\frac{1}{\xi^{i}}\norm{\hat{x}_{k}}^{2} \\
		=& \frac{1}{\sum^{k}_{j=0}\frac{1}{\xi^{j}}}  \sum^{k}_{i=0} \frac{1}{\xi^{i}} 
	\left(\norm{x^{i+1}}^{2}  -2\innp{x^{i+1}, \hat{x}_{k}} +\norm{\hat{x}_{k}}^{2}\right)\\
	=&	 \frac{1}{\sum^{k}_{j=0}\frac{1}{\xi^{j}}}  \sum^{k}_{i=0} \frac{1}{\xi^{i}} 
	\norm{x^{i+1} - \hat{x}_{k}}^{2}.
	\end{align*}
	
	\cref{lemma:stronglyconvex:ineq}:    
	Add the two inequalities in \cref{lemma:stronglyconvex:x} and 
	\cref{lemma:stronglyconvex:y} to observe that
	\begin{align*}
		&f(\hat{x}_{k},y) - f(x, \hat{y}_{k})  \\
		\leq &\frac{1}{\sum^{k}_{j=0}\frac{1}{\xi^{j}}}  \sum^{k}_{i=0} \frac{1}{\xi^{i}} 
		\left(f(x^{i+1},y) -\frac{\mu}{2} \norm{x^{i+1}}^{2}\right)
		+\frac{1}{\sum^{k}_{j=0}\frac{1}{\xi^{j}}}  \sum^{k}_{i=0} \frac{1}{\xi^{i}} 
		\left(-f(x,y^{i+1}) -\frac{\nu}{2} \norm{y^{i+1}}^{2}\right)\\
		&+\frac{\mu}{2} \norm{\hat{x}_{k}}^{2}+\frac{\nu}{2} \norm{\hat{y}_{k}}^{2}\\
		=& \frac{1}{\sum^{k}_{j=0}\frac{1}{\xi^{j}}}  \sum^{k}_{i=0} \frac{1}{\xi^{i}} 
		\left( f(x^{i+1},y)-f(x,y^{i+1})   \right) 
		- \frac{\mu}{2}\left(  \frac{1}{\sum^{k}_{j=0}\frac{1}{\xi^{j}}}  \sum^{k}_{i=0} 
		\frac{1}{\xi^{i}} 
		\norm{x^{i+1}}^{2} -
		\norm{\hat{x}_{k}}^{2} \right)\\
		&- \frac{\nu}{2} \left(  \frac{1}{\sum^{k}_{j=0}\frac{1}{\xi^{j}}}  \sum^{k}_{i=0} 
		\frac{1}{\xi^{i}} 
		\norm{y^{i+1}}^{2} -
		\norm{\hat{y}_{k}}^{2}  \right)\\
		=&\frac{1}{\sum^{k}_{j=0}\frac{1}{\xi^{j}}}  \sum^{k}_{i=0} \frac{1}{\xi^{i}} 
		\left( f(x^{i+1},y)-f(x,y^{i+1})   \right) 
		- \frac{\mu}{2\sum^{k}_{j=0}\frac{1}{\xi^{j}}}  \sum^{k}_{i=0} \frac{1}{\xi^{i}} 
		\norm{x^{i+1} - \hat{x}_{k}}^{2}\\
		&- \frac{\nu}{2\sum^{k}_{j=0}\frac{1}{\xi^{j}}}  \sum^{k}_{i=0} \frac{1}{\xi^{i}} 
		\norm{y^{i+1} - \hat{y}_{k}}^{2},
	\end{align*} 
where in the second equality, we use 
\cref{lemma:stronglyconvex}\cref{lemma:stronglyconvex:normx}$\&$\cref{lemma:stronglyconvex:normy}.
\end{proof}

\begin{lemma} \label{lemma:fxkyfxyk}
	Let $(x,y)$ be in $\mathcal{H}_{1} \times \mathcal{H}_{2}$. 
	Then for every $k \in \mathbf{N}$, we have that
	\begin{align*}
		&f(x^{k+1},y) -f(x,y^{k+1})\\
	\leq & \frac{1}{2 \tau_{k}} \norm{x -x^{k}}^{2} +\frac{1}{2\sigma_{k}} \norm{y 
		-y^{k}}^{2} - \left(\frac{\mu}{2} +\frac{1}{2\tau_{k}}\right)\norm{x -x^{k+1}}^{2} 
	-\left( \frac{\nu}{2} +\frac{1}{2 \sigma_{k}}\right) \norm{y -y^{k+1}}^{2}\\
&- \frac{1}{2\tau_{k}} \norm{x^{k} -x^{k+1}}^{2} 
- \frac{1}{2 \sigma_{k}} \norm{y^{k} -y^{k+1}}^{2}\\ 
&-\innp{K(x^{k+1}-\bar{x}^{k}),y^{k+1}-y}
+\innp{K(x^{k+1}-x),y^{k+1}-\bar{y}^{k+1}}.
	\end{align*}
\end{lemma}

\begin{proof}
	Let $k$ be in $\mathbf{N}$. 
	According to 
	\cref{lemma:partialinequalities}\cref{lemma:partialinequalities:sum}$\&$\cref{lemma:partialinequalities:fK},
	 we observe that
		\begin{align*}
		&\frac{1}{2 \tau_{k}} \norm{x -x^{k}}^{2} +\frac{1}{2\sigma_{k}} \norm{y 
			-y^{k}}^{2}\\
		\geq &\left(\frac{\mu}{2} +\frac{1}{2\tau_{k}}\right)\norm{x -x^{k+1}}^{2}+\left( 
		\frac{\nu}{2} +\frac{1}{2 \sigma_{k}}\right) \norm{y -y^{k+1}}^{2}
		+ \frac{1}{2\tau_{k}} \norm{x^{k} -x^{k+1}}^{2} 
		+ \frac{1}{2 \sigma_{k}} \norm{y^{k} -y^{k+1}}^{2}\\
		&+\innp{K \bar{x}^{k}, y-y^{k+1}}-\innp{\bar{y}^{k+1}, K(x-x^{k+1})}+g(x^{k+1})
		-h(y) +h(y^{k+1})-g(x)\\
		= &\left(\frac{\mu}{2} +\frac{1}{2\tau_{k}}\right)\norm{x -x^{k+1}}^{2}+\left( 
		\frac{\nu}{2} +\frac{1}{2 \sigma_{k}}\right) \norm{y -y^{k+1}}^{2}
		+ \frac{1}{2\tau_{k}} \norm{x^{k} -x^{k+1}}^{2} 
		+ \frac{1}{2 \sigma_{k}} \norm{y^{k} -y^{k+1}}^{2}\\
		&+f(x^{k+1},y) -f(x,y^{k+1})+\innp{K(x^{k+1}-\bar{x}^{k}),y^{k+1}-y}
		-\innp{K(x^{k+1}-x),y^{k+1}-\bar{y}^{k+1}}.
	\end{align*}
After doing some rearrangement in the inequality derived above, 
we obtain the required inequality. 
\end{proof}

\begin{proposition} \label{prop:flinearconvergKsquare}
Let $(x,y)$ be in $\mathcal{H}_{1} \times \mathcal{H}_{2}$. 
	Let $\xi$, $\eta_{1}$, $\eta_{2}$, and $\eta_{3}$ be in $\mathbf{R}_{++}$ such that 
	one of the following condition holds. 
	\begin{enumerate}
		\item[(a)]   
		$(\forall k \in \mathbf{N})$ 
	$\mu +\frac{1}{\tau_{k}} - \eta_{3}\norm{K} \beta_{k}^{2}  \geq \frac{1}{\tau_{k+1} 
	\xi}$,
	$ \nu  +\frac{1}{\sigma_{k}} - \eta_{2} \norm{K}\left(\xi 
	-\alpha_{k-1} \right)^{2} \geq \frac{1}{\sigma_{k+1}\xi}$,
	$\norm{K} \left( \eta_{1} \xi^{2} + \frac{1}{\eta_{2}} 
	\right)  \leq \frac{\xi}{\tau_{k-1}}$,
	$\frac{1}{\sigma_{k}} -\frac{\norm{K}}{\eta_{1}} -\frac{\norm{K}}{\eta_{3}}  \geq 0$.
		\item[(b)]   
		$(\forall k \in \mathbf{N})$ 
		$\mu +\frac{1}{\tau_{k}} - \eta_{3}\norm{K}^{2}\beta_{k}^{2}  \geq 
		\frac{1}{\tau_{k+1} 
			\xi}$,
		$ \nu  +\frac{1}{\sigma_{k}} - \eta_{2} \norm{K}^{2}\left(\xi -\alpha_{k-1} 
		\right)^{2} \geq \frac{1}{\sigma_{k+1}\xi}$,
		$\eta_{1} \norm{K}^{2}\xi^{2} + \frac{1}{\eta_{2}}  \leq \frac{\xi}{\tau_{k-1}}$,
		and $\frac{1}{\sigma_{k}} -\frac{1}{\eta_{1}} -\frac{1}{\eta_{3}}  \geq 0$.
	\end{enumerate}
	We have the following statements. 
	\begin{enumerate}
		\item  \label{prop:flinearconvergKsquare:leq} 
		For every $k \in \mathbf{N}$, we have that
		\begin{align*}
			&f(x^{k+1},y) -f(x,y^{k+1})\\
\leq &
\frac{1}{2 \tau_{k}} \norm{x -x^{k}}^{2} +\frac{1}{2\sigma_{k}} \norm{y 
	-y^{k}}^{2} 
-\innp{K(x^{k+1}-x^{k}),y^{k+1}-y}
+\xi \innp{K(x^{k}-x^{k-1}),y^{k}-y} \\
&- \frac{1}{2 \tau_{k+1}\xi} \norm{x -x^{k+1}}^{2} - \frac{1}{2\sigma_{k+1}\xi} \norm{y 
	-y^{k+1}}^{2} 
- \frac{1}{2\tau_{k}} \norm{x^{k} -x^{k+1}}^{2} 
+ \frac{\xi}{2\tau_{k-1}} \norm{x^{k} -x^{k-1}}^{2}.
		\end{align*}
		\item  \label{prop:flinearconvergKsquare:sum} 
	Let $N$ be in $\mathbf{N}$, and let
		$\eta_{4}$ be in $\mathbf{R}_{++}$ such that $(\forall k \in \mathbf{N})$ 
		$\eta_{4} \geq \tau_{k}$ and $ \frac{1}{\sigma_{k+1}}- \xi \eta_{4}\norm{K}^{2} \geq 
		0$.
		Then  we have that
			\begin{align*}
			\sum^{ N}_{k=0} \frac{1}{\xi^{k}} \left( f(x^{k+1},y) -f(x,y^{k+1}) \right) 
			\leq   \frac{1}{2 \tau_{0}} \norm{x -x^{0}}^{2} 
			+\frac{1}{2\sigma_{0}} \norm{y -y^{0}}^{2}. 
		\end{align*}	
	\end{enumerate}
\end{proposition}

\begin{proof}
	\cref{prop:flinearconvergKsquare:leq}: 
	Let $k$ be in $\mathbf{N}$.
	 
	 Suppose that the condition (a) holds. 
	 Using  \cref{lemma:fxkyfxyk} in the first inequality, 
	 employing \cref{lemma:inequalityCS}\cref{lemma:inequalityCS:K} 
	 in the second inequality,
	 and applying our assumption (a) in the last inequality, we 
	observe that
	\begin{align*}
	&f(x^{k+1},y) -f(x,y^{k+1})\\
	\leq & \frac{1}{2 \tau_{k}} \norm{x -x^{k}}^{2} +\frac{1}{2\sigma_{k}} \norm{y 
		-y^{k}}^{2} - \left(\frac{\mu}{2} +\frac{1}{2\tau_{k}}\right)\norm{x -x^{k+1}}^{2} 
	-\left( \frac{\nu}{2} +\frac{1}{2 \sigma_{k}}\right) \norm{y -y^{k+1}}^{2}\\
	&- \frac{1}{2\tau_{k}} \norm{x^{k} -x^{k+1}}^{2} 
	- \frac{1}{2 \sigma_{k}} \norm{y^{k} -y^{k+1}}^{2}\\ 
	&-\innp{K(x^{k+1}-\bar{x}^{k}),y^{k+1}-y}
	+\innp{K(x^{k+1}-x),y^{k+1}-\bar{y}^{k+1}}\\
	\leq &	\frac{1}{2 \tau_{k}} \norm{x -x^{k}}^{2} +\frac{1}{2\sigma_{k}} \norm{y 
		-y^{k}}^{2} - \left(\frac{\mu}{2} +\frac{1}{2\tau_{k}}\right)\norm{x -x^{k+1}}^{2} 
	-\left( \frac{\nu}{2} +\frac{1}{2 \sigma_{k}}\right) \norm{y -y^{k+1}}^{2}\\
	&- \frac{1}{2\tau_{k}} \norm{x^{k} -x^{k+1}}^{2} 
	- \frac{1}{2 \sigma_{k}} \norm{y^{k} -y^{k+1}}^{2}\\ 
	&-\innp{K(x^{k+1}-x^{k}),y^{k+1}-y}
	+\xi \innp{K(x^{k}-x^{k-1}),y^{k}-y}\\
	&+\frac{ \norm{K}}{2}\left(\eta_{1}\xi^{2}\norm{x^{k}-x^{k-1}}^{2}
	+\frac{\norm{y^{k+1}-y^{k} }}{\eta_{1}}^{2}\right) \\
	&+\frac{\norm{K}}{2} \left( \eta_{2} \left(\xi -\alpha_{k-1} 
	\right)^{2}\norm{y^{k+1}-y}^{2} 
	+\frac{\norm{x^{k}-x^{k-1}}^{2}}{\eta_{2}}
	\right)\\
	&+\frac{\norm{K}}{2} \left( \eta_{3}\beta_{k}^{2}\norm{x^{k+1}-x}^{2} 
	+\frac{\norm{y^{k+1}-y^{k}}^{2} }{\eta_{3}}  \right)\\
	=&\frac{1}{2 \tau_{k}} \norm{x -x^{k}}^{2} +\frac{1}{2\sigma_{k}} \norm{y 
		-y^{k}}^{2} 
	-\innp{K(x^{k+1}-x^{k}),y^{k+1}-y}
	+\xi \innp{K(x^{k}-x^{k-1}),y^{k}-y} \\
	&- \frac{1}{2}\left(\mu +\frac{1}{\tau_{k}} - \eta_{3}\norm{K} \beta_{k}^{2} 
	\right)\norm{x -x^{k+1}}^{2} 
	-\frac{1}{2} \left(  \nu  +\frac{1}{\sigma_{k}} - \eta_{2} \norm{K}\left(\xi 
	-\alpha_{k-1} 
	\right)^{2}\right) \norm{y -y^{k+1}}^{2}\\
	&- \frac{1}{2\tau_{k}} \norm{x^{k} -x^{k+1}}^{2}
	+\frac{\norm{K}}{2}\left( \eta_{1} \xi^{2} + \frac{1}{\eta_{2}} 
	\right)\norm{x^{k}-x^{k-1}}^{2}\\
	&- \frac{1}{2} \left( \frac{1}{\sigma_{k}} -\frac{\norm{K}}{\eta_{1}} 
	-\frac{\norm{K}}{\eta_{3}} \right) 
	\norm{y^{k} -y^{k+1}}^{2}\\
	\leq& \frac{1}{2 \tau_{k}} \norm{x -x^{k}}^{2} +\frac{1}{2\sigma_{k}} \norm{y 
		-y^{k}}^{2} 
	-\innp{K(x^{k+1}-x^{k}),y^{k+1}-y}
	+\xi \innp{K(x^{k}-x^{k-1}),y^{k}-y} \\
	&- \frac{1}{2 \tau_{k+1}\xi} \norm{x -x^{k+1}}^{2} - \frac{1}{2\sigma_{k+1}\xi} \norm{y 
		-y^{k+1}}^{2} 
	- \frac{1}{2\tau_{k}} \norm{x^{k} -x^{k+1}}^{2} 
	+ \frac{\xi}{2\tau_{k-1}} \norm{x^{k} -x^{k-1}}^{2}, 
\end{align*}
which leads to the required inequality.

Similarly, under the condition (b), applying  \cref{lemma:fxkyfxyk} in the first inequality, 
	employing \cref{lemma:inequalityCS}\cref{lemma:inequalityCS:Ksquare} 
	in the second inequality,
	and adopting our assumptions in (b) in the last inequality, we 
	deduce that
	\begin{align*}
		&f(x^{k+1},y) -f(x,y^{k+1})\\
		\leq & \frac{1}{2 \tau_{k}} \norm{x -x^{k}}^{2} +\frac{1}{2\sigma_{k}} \norm{y 
			-y^{k}}^{2} - \left(\frac{\mu}{2} +\frac{1}{2\tau_{k}}\right)\norm{x -x^{k+1}}^{2} 
		-\left( \frac{\nu}{2} +\frac{1}{2 \sigma_{k}}\right) \norm{y -y^{k+1}}^{2}\\
		&- \frac{1}{2\tau_{k}} \norm{x^{k} -x^{k+1}}^{2} 
		- \frac{1}{2 \sigma_{k}} \norm{y^{k} -y^{k+1}}^{2}\\ 
		&-\innp{K(x^{k+1}-\bar{x}^{k}),y^{k+1}-y}
		+\innp{K(x^{k+1}-x),y^{k+1}-\bar{y}^{k+1}}\\
	\leq &	\frac{1}{2 \tau_{k}} \norm{x -x^{k}}^{2} +\frac{1}{2\sigma_{k}} \norm{y 
		-y^{k}}^{2} - \left(\frac{\mu}{2} +\frac{1}{2\tau_{k}}\right)\norm{x -x^{k+1}}^{2} 
	-\left( \frac{\nu}{2} +\frac{1}{2 \sigma_{k}}\right) \norm{y -y^{k+1}}^{2}\\
	&- \frac{1}{2\tau_{k}} \norm{x^{k} -x^{k+1}}^{2} 
	- \frac{1}{2 \sigma_{k}} \norm{y^{k} -y^{k+1}}^{2}\\ 
	&-\innp{K(x^{k+1}-x^{k}),y^{k+1}-y}
	+\xi \innp{K(x^{k}-x^{k-1}),y^{k}-y}\\
	&+\frac{1}{2}\left(\eta_{1} \norm{K}^{2}\xi^{2}\norm{x^{k}-x^{k-1}}^{2}
	+\frac{\norm{y^{k+1}-y^{k} }}{\eta_{1}}^{2}\right) \\
	&+\frac{1}{2} \left( \eta_{2} \norm{K}^{2}\left(\xi -\alpha_{k-1} 
	\right)^{2}\norm{y^{k+1}-y}^{2} 
	+\frac{\norm{x^{k}-x^{k-1}}^{2}}{\eta_{2}}
	\right)\\
	&+\frac{1}{2} \left( \eta_{3}\norm{K}^{2}\beta_{k}^{2}\norm{x^{k+1}-x}^{2} 
	+\frac{\norm{y^{k+1}-y^{k}}^{2} }{\eta_{3}}  \right)\\
=&\frac{1}{2 \tau_{k}} \norm{x -x^{k}}^{2} +\frac{1}{2\sigma_{k}} \norm{y 
	-y^{k}}^{2} 
-\innp{K(x^{k+1}-x^{k}),y^{k+1}-y}
+\xi \innp{K(x^{k}-x^{k-1}),y^{k}-y} \\
&- \frac{1}{2}\left(\mu +\frac{1}{\tau_{k}} - \eta_{3}\norm{K}^{2}\beta_{k}^{2} 
\right)\norm{x -x^{k+1}}^{2} 
-\frac{1}{2} \left(  \nu  +\frac{1}{\sigma_{k}} - \eta_{2} \norm{K}^{2}\left(\xi -\alpha_{k-1} 
\right)^{2}\right) \norm{y -y^{k+1}}^{2}\\
&- \frac{1}{2\tau_{k}} \norm{x^{k} -x^{k+1}}^{2}
+\frac{1}{2}\left( \eta_{1} \norm{K}^{2}\xi^{2} + \frac{1}{\eta_{2}} 
\right)\norm{x^{k}-x^{k-1}}^{2}\\
&- \frac{1}{2} \left( \frac{1}{\sigma_{k}} -\frac{1}{\eta_{1}} -\frac{1}{\eta_{3}} \right) 
\norm{y^{k} -y^{k+1}}^{2}\\
\leq& \frac{1}{2 \tau_{k}} \norm{x -x^{k}}^{2} +\frac{1}{2\sigma_{k}} \norm{y 
	-y^{k}}^{2} 
-\innp{K(x^{k+1}-x^{k}),y^{k+1}-y}
+\xi \innp{K(x^{k}-x^{k-1}),y^{k}-y} \\
&- \frac{1}{2 \tau_{k+1}\xi} \norm{x -x^{k+1}}^{2} - \frac{1}{2\sigma_{k+1}\xi} \norm{y 
	-y^{k+1}}^{2} 
- \frac{1}{2\tau_{k}} \norm{x^{k} -x^{k+1}}^{2} 
+ \frac{\xi}{2\tau_{k-1}} \norm{x^{k} -x^{k-1}}^{2}, 
	\end{align*}
	which immediately derives the desired inequality. 
	
	\cref{prop:flinearconvergKsquare:sum}: 
	For every $k$ in $\mathbf{N}$, multiply both sides of the inequality in 
	\cref{prop:flinearconvergKsquare:leq} above
	by $\frac{1}{\xi^{k}}$ to see that
	\begin{align*}
		&\frac{1}{\xi^{k}} \left( f(x^{k+1},y) -f(x,y^{k+1}) \right)\\
		\leq &
		\frac{1}{2 \tau_{k}\xi^{k}} \norm{x -x^{k}}^{2} 
		+\frac{1}{2\sigma_{k}\xi^{k}} \norm{y -y^{k}}^{2} 
		-\frac{1}{\xi^{k}} \innp{K(x^{k+1}-x^{k}),y^{k+1}-y}
		+\frac{1}{\xi^{k-1}} \innp{K(x^{k}-x^{k-1}),y^{k}-y} \\
		&- \frac{1}{2 \tau_{k+1}\xi^{k+1}} \norm{x -x^{k+1}}^{2} 
		- \frac{1}{2\sigma_{k+1}\xi^{k+1}} \norm{y -y^{k+1}}^{2} 
		- \frac{1}{2\tau_{k}\xi^{k}} \norm{x^{k} -x^{k+1}}^{2} 
		+ \frac{1}{2\tau_{k-1}\xi^{k-1}} \norm{x^{k} -x^{k-1}}^{2}.
	\end{align*}
	Then sum the inequality above from $k=0$ to $k=N$ to get that
		\begin{align*}
		&\sum^{k=N}_{k=0} \frac{1}{\xi^{k}} \left( f(x^{k+1},y) -f(x,y^{k+1}) \right)\\
		\leq &
		\frac{1}{2 \tau_{0}} \norm{x -x^{0}}^{2} 
		+\frac{1}{2\sigma_{0}} \norm{y -y^{0}}^{2} 
		-\frac{1}{\xi^{N}} \innp{K(x^{N+1}-x^{N}),y^{N+1}-y}
		+\frac{1}{\xi^{-1}} \innp{K(x^{0}-x^{-1}),y^{0}-y} \\
		&- \frac{1}{2 \tau_{N+1}\xi^{N+1}} \norm{x -x^{N+1}}^{2} 
		- \frac{1}{2\sigma_{N+1}\xi^{N+1}} \norm{y -y^{N+1}}^{2} 
		- \frac{1}{2\tau_{N}\xi^{N}} \norm{x^{N} -x^{N+1}}^{2} 
		+ \frac{1}{2\tau_{-1}\xi^{-1}} \norm{x^{0} -x^{-1}}^{2}\\
		=&\frac{1}{2 \tau_{0}} \norm{x -x^{0}}^{2} 
		+\frac{1}{2\sigma_{0}} \norm{y -y^{0}}^{2} 
		-\frac{1}{\xi^{N}} \innp{K(x^{N+1}-x^{N}),y^{N+1}-y}\\
		&- \frac{1}{2 \tau_{N+1}\xi^{N+1}} \norm{x -x^{N+1}}^{2} 
		- \frac{1}{2\sigma_{N+1}\xi^{N+1}} \norm{y -y^{N+1}}^{2} 
		- \frac{1}{2\tau_{N}\xi^{N}} \norm{x^{N} -x^{N+1}}^{2}\\
		\leq & \frac{1}{2 \tau_{0}} \norm{x -x^{0}}^{2} 
		+\frac{1}{2\sigma_{0}} \norm{y -y^{0}}^{2} 
		+\frac{1}{2\xi^{N}} \left(\eta_{4}\norm{K}^{2} \norm{y^{N+1}-y}^{2}
		+\frac{1}{\eta_{4}} \norm{x^{N+1}-x^{N}}^{2} \right)\\
		&- \frac{1}{2 \tau_{N+1}\xi^{N+1}} \norm{x -x^{N+1}}^{2} 
		- \frac{1}{2\sigma_{N+1}\xi^{N+1}} \norm{y -y^{N+1}}^{2} 
		- \frac{1}{2\tau_{N}\xi^{N}} \norm{x^{N} -x^{N+1}}^{2}\\
		=& \frac{1}{2 \tau_{0}} \norm{x -x^{0}}^{2} 
		+\frac{1}{2\sigma_{0}} \norm{y -y^{0}}^{2} 
		+\frac{1}{2\xi^{N}} \left( \frac{1}{\eta_{4}} -\frac{1}{\tau_{N}} \right)  
		\norm{x^{N} -x^{N+1}}^{2}\\
		&- \frac{1}{2 \tau_{N+1}\xi^{N+1}} \norm{x -x^{N+1}}^{2} 
		 -\frac{1}{2\xi^{N+1}} \left( \frac{1}{\sigma_{N+1}}- \xi \eta_{4}\norm{K}^{2} \right) 
		 \norm{y^{N+1}-y}^{2}\\
	\leq & \frac{1}{2 \tau_{0}} \norm{x -x^{0}}^{2} 
	+\frac{1}{2\sigma_{0}} \norm{y -y^{0}}^{2}. 
	\end{align*}
where in the first equality above we use the fact $x^{0} =x^{-1}$, in the second
 inequality above we adopt
 \cref{lemma:inequalityCS}\cref{lemma:inequalityCS:Ksquareeta4},
 and in the last inequality above we use the assumptions 
 $(\forall k \in \mathbf{N})$ 
 $\eta_{4} \geq \tau_{k}$ and $ \frac{1}{\sigma_{k+1}}- \xi \eta_{4}\norm{K}^{2} \geq 0$.
\end{proof}

\begin{proposition}\label{prop:linearconvergeineq}
	Let $(x^{*},y^{*})$ be a saddle-point of $f$.
	Let $\eta_{4}$  be in $\mathbf{R}_{++}$  such that $(\forall k \in \mathbf{N})$
	$\eta_{4} \geq \tau_{k}$ and $ \frac{1}{\sigma_{k+1}}- \xi \eta_{4}\norm{K}^{2} \geq 
	0$.
	Let $\xi$ be in $(0,1)$, 
	and let $\eta_{1}$, $\eta_{2}$, and $\eta_{3}$ be in $\mathbf{R}_{++}$ 
	such that one of the following conditions hold. 
	\begin{enumerate}
		\item[(a)] \label{prop:linearconvergeineq:cK} 
		$(\forall k \in \mathbf{N})$ 
		$\mu +\frac{1}{\tau_{k}} - \eta_{3}\norm{K} \beta_{k}^{2}  \geq \frac{1}{\tau_{k+1} 
			\xi}$,
		$ \nu  +\frac{1}{\sigma_{k}} - \eta_{2} \norm{K}\left(\xi 
		-\alpha_{k-1} \right)^{2} \geq \frac{1}{\sigma_{k+1}\xi}$,
		$\norm{K} \left( \eta_{1} \xi^{2} + \frac{1}{\eta_{2}} 
		\right)  \leq \frac{\xi}{\tau_{k-1}}$, and
		$\frac{1}{\sigma_{k}} -\frac{\norm{K}}{\eta_{1}} -\frac{\norm{K}}{\eta_{3}}  \geq 0$.
		\item[(b)] \label{prop:linearconvergeineq:cKsquare} $(\forall k \in \mathbf{N})$ 
		$\mu +\frac{1}{\tau_{k}} - \eta_{3}\norm{K}^{2}\beta_{k}^{2}  \geq 
		\frac{1}{\tau_{k+1} 
			\xi}$,
		$ \nu  +\frac{1}{\sigma_{k}} - \eta_{2} \norm{K}^{2}\left(\xi -\alpha_{k-1} 
		\right)^{2} \geq \frac{1}{\sigma_{k+1}\xi}$,
		$\eta_{1} \norm{K}^{2}\xi^{2} + \frac{1}{\eta_{2}}  \leq \frac{\xi}{\tau_{k-1}}$, and
		$\frac{1}{\sigma_{k}} -\frac{1}{\eta_{1}} -\frac{1}{\eta_{3}}  \geq 0$.
	\end{enumerate}

Then we have that for every $ k \in \mathbf{N}$,
	\begin{subequations}
	\begin{align*}
		&f(\hat{x}_{k},\hat{y}_{k}) -f(x^{*},y^{*}) 
		\leq \xi^{k} \left(  \frac{1}{2 \tau_{0}} \norm{x^{*} -x^{0}}^{2} 
		+\frac{1}{2\sigma_{0}} \norm{\hat{y}_{k} -y^{0}}^{2} \right);\\
		&f(x^{*},y^{*}) -f(\hat{x}_{k},\hat{y}_{k})
		\leq \xi^{k}\left(  \frac{1}{2 \tau_{0}} \norm{\hat{x}_{k} -x^{0}}^{2} 
		+\frac{1}{2\sigma_{0}} \norm{y^{*} -y^{0}}^{2} \right).
	\end{align*}
\end{subequations} 
\end{proposition}

\begin{proof}
	According to 
	\cref{prop:flinearconvergKsquare}\cref{prop:flinearconvergKsquare:sum}  under the 
	set of conditions (a)
	or (b) above, we have that 
	for every $(x,y) \in \mathcal{H}_{1}\times \mathcal{H}_{2}$ and every $k \in 
	\mathbf{N}$,
	\begin{align}\label{eq:prop:linearconvergeineq}
	 \frac{1}{\sum^{k}_{j=0}\frac{1}{\xi^{j}}} 	\sum^{k}_{i=0} \frac{1}{\xi^{i}} \left( 
	 f(x^{i+1},y) -f(x,y^{i+1}) \right) 
		\leq  \frac{1}{\sum^{k}_{j=0}\frac{1}{\xi^{j}}}  \left(  \frac{1}{2 \tau_{0}} \norm{x 
		-x^{0}}^{2} 
		+\frac{1}{2\sigma_{0}} \norm{y -y^{0}}^{2} \right). 
	\end{align}	

	Note that because $\xi \in (0,1)$, we have that for every $k \in \mathbf{N}$,
	\begin{align}\label{eq:prop:linearconvergeineq:xi}
	\sum^{k}_{j=0}	\frac{1}{\xi^{j}}=\frac{1}{\xi^{k}} \sum^{k}_{j=0}\xi^{j}
	= \frac{1}{\xi^{k}} \frac{1 -\xi^{k}}{1-\xi} \geq \frac{1}{\xi^{k}}.
	\end{align}
In the following two expressions, applying \cref{fact:fxkykx*y*} with $(\forall i \in 
\mathbf{N})$ $t_{i} =\frac{1}{\xi^{i}}$
in the first inequality below, 
employing \cref{eq:prop:linearconvergeineq}
in the second inequality with $(x,y)=\left(x^{*}, \hat{y}_{k}\right)$ 
and $(x,y)=\left(\hat{x}_{k}, y^{*}\right)$, respectively, we have that
\begin{align*}
	&f(\hat{x}_{k},\hat{y}_{k}) -f(x^{*},y^{*})\\ 
	\leq & \frac{1}{\sum^{k}_{i=0}\frac{1}{\xi^{i}}} \sum^{k}_{j=0}\frac{1}{\xi^{j}} \left( 
	f(x^{j+1}, \hat{y}_{k})  
	-f(x^{*},y^{j+1}) \right)\\
	\leq &  \frac{1}{\sum^{k}_{i=0}\frac{1}{\xi^{i}}}  \left(\frac{1}{2 \tau_{0}} \norm{x^{*} 
	-x^{0}}^{2} 
	+\frac{1}{2\sigma_{0}} \norm{ \hat{y}_{k} -y^{0}}^{2}\right)\\
	\stackrel{\cref{eq:prop:linearconvergeineq:xi}}{\leq} & \xi^{k} \left(  \frac{1}{2 \tau_{0}} 
	\norm{x^{*} -x^{0}}^{2} 
	+\frac{1}{2\sigma_{0}} \norm{\hat{y}_{k} -y^{0}}^{2} \right);
\end{align*}
and 
\begin{align*}
&f(x^{*},y^{*}) -f(\hat{x}_{k},\hat{y}_{k})\\
\leq &\frac{1}{\sum^{k}_{i=0}\frac{1}{\xi^{i}}} \sum^{k}_{j=0}\frac{1}{\xi^{j}} \left( 
f(x^{j+1},y^{*})  
-f(\hat{x}_{k},y^{j+1}) \right)	\\
\leq &  \frac{1}{\sum^{k}_{i=0}\frac{1}{\xi^{i}}}  \left(\frac{1}{2 \tau_{0}} \norm{\hat{x}_{k}
-x^{0}}^{2} 
+\frac{1}{2\sigma_{0}} \norm{y^{*} -y^{0}}^{2}\right)\\
\stackrel{\cref{eq:prop:linearconvergeineq:xi}}{\leq} & 
\xi^{k} \left(\frac{1}{  2 \tau_{0}} \norm{\hat{x}_{k}
	-x^{0}}^{2} 
+\frac{1}{ 2\sigma_{0}} \norm{y^{*} -y^{0}}^{2}\right).
\end{align*}
Therefore, our proof is complete.
\end{proof}

\begin{theorem} \label{theorem:linearconvergeineq}
		Let $(x^{*},y^{*})$ be a saddle-point of $f$.
	Suppose $\inf_{k \in 	\mathbf{N}}\tau_{k} >0$ and $\inf_{k \in 	
		\mathbf{N}}\sigma_{k} >0$. 
	Let $\xi$ be in $(0,1)$ such that 
	$\xi \sup_{k \in \mathbf{N}} \tau_{k} \sup_{k \in \mathbf{N}} \sigma_{k} \norm{K}^{2}< 
	1$.
	Let  $\eta_{4}$ be in 
	$\mathbf{R}_{++}$ such that  $\eta_{4} \geq \sup_{i \in \mathbf{N}} \tau_{i}$,
	and $\frac{1}{\sup_{i \in \mathbf{N}}\sigma_{i}} -\xi \eta_{4} \norm{K}^{2} >0$.

	Let $\eta_{1}$, $\eta_{2}$ and $\eta_{3}$ be in 
	$\mathbf{R}_{++}$ such that  one of the following assumptions hold. 
	\begin{enumerate}
		\item[(a)]  \label{theorem:linearconvergeineq:K} 
		$
		\mu\geq \frac{1}{\xi \inf_{k \in \mathbf{N}} 
			\tau_{k}} -\frac{1}{\sup_{k \in \mathbf{N}}\tau_{k}}+  \eta_{3} \norm{K} \sup_{k 
			\in \mathbf{N}}\beta_{k}^{2}
		$,
		$
		\nu \geq \frac{1}{\xi \inf_{k \in 
				\mathbf{N}}\sigma_{k}} -\frac{1}{\sup_{k \in \mathbf{N}}\sigma_{k}}  
		+ \eta_{2}\norm{K} \sup_{k \in \mathbf{N}}(\xi -\alpha_{k})^{2}
		$,
		$\norm{K} \left( \eta_{1} \xi^{2} 
		+\frac{1}{\eta_{2}}\right) \leq \frac{\xi}{\sup_{i 
		\in 
				\mathbf{N}}\tau_{i}}$, and
		$\norm{K} \left(\frac{1}{\eta_{1}} +\frac{1}{\eta_{3}}\right) \leq \frac{1}{\sup_{i \in 
				\mathbf{N}}\sigma_{i}}$.
		\item[(b)]  \label{theorem:linearconvergeineq:Ksquare} 
		$
		\mu\geq \frac{1}{\xi \inf_{k \in \mathbf{N}} 
			\tau_{k}} -\frac{1}{\sup_{k \in \mathbf{N}}\tau_{k}}+  \eta_{3} \norm{K}^{2}\sup_{k 
			\in \mathbf{N}}\beta_{k}^{2}
		$,
		$
		\nu \geq \frac{1}{\xi \inf_{k \in 
				\mathbf{N}}\sigma_{k}} -\frac{1}{\sup_{k \in \mathbf{N}}\sigma_{k}}  
		+ \eta_{2}\norm{K}^{2} \sup_{k \in \mathbf{N}}(\xi -\alpha_{k})^{2}
		$,
		$\eta_{1} \norm{K}^{2}\xi 
		+\frac{1}{\eta_{2}\xi} \leq \frac{1}{\sup_{k \in 
				\mathbf{N}}\tau_{k}}$,
		and
		$\frac{1}{\eta_{1}} +\frac{1}{\eta_{3}} \leq \frac{1}{\sup_{k \in \mathbf{N}} 
		\sigma_{k}}$.
	\end{enumerate}

	Then there exists $M \in \mathbf{R}_{++}$ such that  for every $ k \in \mathbf{N}$,
	\begin{subequations}
		\begin{align*}
			&f(\hat{x}_{k},\hat{y}_{k}) -f(x^{*},y^{*}) 
			\leq \xi^{k} \left(  \frac{1}{2 \tau_{0}} \norm{x^{*} -x^{0}}^{2} 
			+\frac{1}{2\sigma_{0}} \norm{\hat{y}_{k} -y^{0}}^{2} \right) \leq \xi^{k}M;\\
			&f(x^{*},y^{*}) -f(\hat{x}_{k},\hat{y}_{k})
			\leq \xi^{k}\left(  \frac{1}{2 \tau_{0}} \norm{\hat{x}_{k} -x^{0}}^{2} 
			+\frac{1}{2\sigma_{0}} \norm{y^{*} -y^{0}}^{2} \right)\leq \xi^{k}M.
		\end{align*}
	\end{subequations} 
	Consequently, 
	the sequence $\left(f\left( \hat{x}_{k},\hat{y}_{k} \right)\right)_{k \in \mathbf{N}}$
	converges linearly to the point $f(x^{*},y^{*})$ with a rate $\xi \in (0,1)$.
\end{theorem}

\begin{proof}
Because $\mu >0$ and $\nu >0$, it is clear that 
$(\forall i \in \{1,2\})$
$
\mu\geq \frac{1}{\xi \inf_{k \in \mathbf{N}} 
	\tau_{k}} -\frac{1}{\sup_{k \in \mathbf{N}}\tau_{k}}+  \eta_{3} \norm{K}^{i}\sup_{k 
	\in \mathbf{N}}\beta_{k}^{2}
$
and 
$
\nu \geq \frac{1}{\xi \inf_{k \in 
		\mathbf{N}}\sigma_{k}} -\frac{1}{\sup_{k \in \mathbf{N}}\sigma_{k}}  
+ \eta_{2}\norm{K}^{i} \sup_{k \in \mathbf{N}}(\xi -\alpha_{k})^{2}
$
imply that
$
2\mu\geq \frac{1}{\xi \inf_{k \in \mathbf{N}} 
	\tau_{k}} -\frac{1}{\sup_{k \in \mathbf{N}}\tau_{k}}+  \eta_{3} \norm{K}^{i}\sup_{k 
	\in \mathbf{N}}\beta_{k}^{2}
$
and 
$
2 \nu \geq \frac{1}{\xi \inf_{k \in 
		\mathbf{N}}\sigma_{k}} -\frac{1}{\sup_{k \in \mathbf{N}}\sigma_{k}}  
+ \eta_{2}\norm{K}^{i} \sup_{k \in \mathbf{N}}(\xi -\alpha_{k})^{2}.
$
Combine this result with our rest assumptions and 
\Cref{theorem:linearconverg:K,theorem:linearconverg:Ksquare}
to know that the sequence of iterations $\left((x^{k},y^{k})\right)_{k \in  \mathbf{N}}$ 
is bounded, which, via \cref{eq:x_ky_k}, yields that
 the sequence
$\left( \left(\hat{x}_{k}, \hat{y}_{k} \right)\right)_{k \in \mathbf{N}}$ is bounded, 
and that there $M \in \mathbf{R}_{++}$ such that
\begin{align}\label{eq:theorem:linearconvergeineq:M}
 \frac{1}{2 \tau_{0}} \norm{x^{*} -x^{0}}^{2} 
	+\frac{1}{2\sigma_{0}} \norm{\hat{y}_{k} -y^{0}}^{2}   \leq M
\quad 	\text{and} \quad
 \frac{1}{2 \tau_{0}} \norm{\hat{x}_{k} -x^{0}}^{2} 
	+\frac{1}{2\sigma_{0}} \norm{y^{*} -y^{0}}^{2} \leq  M.
\end{align}

On the other hand, it is easy to see that $(\forall i \in \{1,2\})$
our assumption 
$
\mu\geq \frac{1}{\xi \inf_{k \in \mathbf{N}} 
	\tau_{k}} -\frac{1}{\sup_{k \in \mathbf{N}}\tau_{k}}+  \eta_{3} \norm{K}^{i}\sup_{k 
	\in \mathbf{N}}\beta_{k}^{2}
$
and 
$
\nu \geq \frac{1}{\xi \inf_{k \in 
		\mathbf{N}}\sigma_{k}} -\frac{1}{\sup_{k \in \mathbf{N}}\sigma_{k}}  
+ \eta_{2}\norm{K}^{i} \sup_{k \in \mathbf{N}}(\xi -\alpha_{k})^{2}
$
ensures the condition $(\forall k \in \mathbf{N})$ 
$\mu +\frac{1}{\tau_{k}} - \eta_{3}\norm{K}^{i} \beta_{k}^{2}  \geq \frac{1}{\tau_{k+1} 
	\xi}$ and
$ \nu  +\frac{1}{\sigma_{k}} - \eta_{2} \norm{K}^{i}  \left(\xi 
-\alpha_{k-1} \right)^{2} \geq \frac{1}{\sigma_{k+1}\xi}$ 
required in condition (a) with $i=1$ and in condition $b$ with $i=2$
 in  \cref{prop:linearconvergeineq}.

Moreover, 	because we assume that $\xi \sup_{k \in \mathbf{N}} \tau_{k} \sup_{k \in 
	\mathbf{N}} \sigma_{k} \norm{K}^{2}< 1$, 
we are able to take $\eta_{4} \in \mathbf{R}_{++}$ such that
$0<\inf_{k \in \mathbf{N}}\tau_{k} \leq \sup_{k \in \mathbf{N}}\tau_{k} \leq \eta_{4}$ and 
$\eta_{4}\xi \norm{K}^{2} 
<\frac{1}{ \sup_{k \in \mathbf{N}}\sigma_{k}}$, which leads to the required condition
\begin{align*}
(\forall k \in \mathbf{N}) \quad 	\frac{1}{\sigma_{k+1}}- \xi \eta_{4}\norm{K}^{2} \geq 
0
	\quad	\text{and} \quad
	 \eta_{4} \geq \tau_{k}.
\end{align*}
of  \cref{prop:linearconvergeineq}.

Furthermore, it is clear that our assumptions
$\norm{K} \left( \eta_{1} \xi^{2} +\frac{1}{\eta_{2}}\right) \leq \frac{\xi}{\sup_{i \in 
		\mathbf{N}}\tau_{i}}$, and
$\norm{K} \left(\frac{1}{\eta_{1}} +\frac{1}{\eta_{3}}\right) \leq \frac{1}{\sup_{i \in 
		\mathbf{N}}\sigma_{i}}$
	in condition (a) above
ensures
\begin{align*}
	(\forall k \in \mathbf{N}) \quad     \norm{K} \left(\eta_{1}\xi^{2}
	+ \frac{1}{\eta_{2}}\right) \leq \frac{\xi}{\tau_{k-1}} 
	\quad \text{and} \quad 
	\frac{1}{ \sigma_{k}} - \norm{K} \left(\frac{1}{\eta_{1}} +\frac{1}{\eta_{3}} \right) \geq 0, 
\end{align*}
which are the rest condition required  in 
 \cref{prop:linearconvergeineq}(a).
Additionally, it is immediate to see that our assumption 
$\eta_{1} \norm{K}^{2}\xi +\frac{1}{\eta_{2}\xi} \leq \frac{1}{\sup_{k \in 
		\mathbf{N}}\tau_{k}}$
and
$\frac{1}{\eta_{1}} +\frac{1}{\eta_{3}} \leq \frac{1}{\sup_{k \in \mathbf{N}} \sigma_{k}}$
in our condition (b) above,
guarantees the required condition
\begin{align*}
	(\forall k \in \mathbf{N}) \quad  \eta_{1} \norm{K}^{2}\xi^{2}+\frac{1}{\eta_{2}} \leq 
	\frac{\xi}{\tau_{k-1}} \quad \text{and} \quad 
	\frac{1}{\sigma_{k}}  -\frac{1}{\eta_{1}} -\frac{1}{\eta_{3}} \geq 0. 
\end{align*}
in \cref{prop:linearconvergeineq}(b).

Altogether, all required assumptions of \cref{prop:linearconvergeineq} hold in our case
under condition (a) or (b),
 and we are able to apply  \cref{prop:linearconvergeineq} to derive that for every $k \in 
 \mathbf{N}$,
\begin{subequations}
	\begin{align*}
		&f(\hat{x}_{k},\hat{y}_{k}) -f(x^{*},y^{*}) 
		\leq \xi^{k} \left(  \frac{1}{2 \tau_{0}} \norm{x^{*} -x^{0}}^{2} 
		+\frac{1}{2\sigma_{0}} \norm{\hat{y}_{k} -y^{0}}^{2} \right)
		 \stackrel{\cref{eq:theorem:linearconvergeineq:M}}{\leq }  \xi^{k}M;\\
		&f(x^{*},y^{*}) -f(\hat{x}_{k},\hat{y}_{k})
		\leq \xi^{k}\left(  \frac{1}{2 \tau_{0}} \norm{\hat{x}_{k} -x^{0}}^{2} 
		+\frac{1}{2\sigma_{0}} \norm{y^{*} -y^{0}}^{2} \right)
		\stackrel{\cref{eq:theorem:linearconvergeineq:M}}{\leq }  \xi^{k}M.
	\end{align*}
\end{subequations} 
Therefore, we conclude that 
the sequence $\left(f\left( \hat{x}_{k},\hat{y}_{k} \right)\right)_{k \in \mathbf{N}}$
converges linearly to the point $f(x^{*},y^{*})$ with a rate $\xi \in (0,1)$.
\end{proof}

To apply the linear convergence results stated in 
\cref{theorem:linearconvergeineq} easily in practice , 
we provide the following results with clear assumptions and with all involved parameters
in our algorithm \cref{eq:fact:algorithmsymplity} being constants.

\begin{corollary} \label{corollary:linearconvergeineq}
Let $(x^{*},y^{*})$ be a saddle-point of $f$. Suppose that 
\begin{align}  \label{eq:corollary:linearconvergeineq:constant}
	(\forall k \in \mathbf{N} \cup \{-1\}) \quad \tau_{k} \equiv \tau \in \mathbf{R}_{++}, 
	\sigma_{k} \equiv  \sigma \in \mathbf{R}_{++}, 
	\alpha_{k} \equiv \alpha \in \mathbf{R}_{++}, \text{ and } 
	\beta_{k} \equiv \beta \in \mathbf{R}_{+}.
\end{align}
Recall that $\mu>0$, $\nu>0$, $\norm{K} \geq 0$ are known. 
Let $\tau >0$ and $\sigma >0$ be small enough such that 
$\norm{K}^{2} <\frac{\mu +\frac{1}{\tau}}{\sigma}$ and 
$\norm{K}^{2} < \frac{ \nu +\frac{1}{\sigma}}{\tau}$. 
Let $\alpha$ be in $\left( \frac{1}{  \mu \tau +1} , 1\right)$ such that $\alpha \geq 
\frac{1}{\nu \sigma +1}$, and $\alpha \tau \sigma \norm{K}^{2} <1$.
Let $\beta$ satisfy that  
$\norm{K}^{2} \beta^{2} \sigma < \left(\mu 
+\frac{1}{\tau}-\frac{1}{\alpha 
	\tau} \right)\left( 1-\alpha \tau \sigma \norm{K}^{2}\right)
$.
Let $\eta_{4}$ be in $\mathbf{R}_{++}$ such that 
$\tau \leq \eta_{4}$ and $\alpha \sigma \eta_{4}\norm{K}^{2} <1$.

Suppose that one of the following assumptions hold. 
\begin{enumerate}
	\item \label{corollary:linearconvergeineq:K} 
		Let $\eta_{3}$ be in $\mathbf{R}_{++}$ such that 
	$\frac{\sigma \norm{K}}{1 -\alpha \tau \sigma \norm{K}^{2}} < \eta_{3} $ and
	$\eta_{3}\norm{K} \beta^{2} \leq  \mu +\frac{1}{\tau} 
	-\frac{1}{\alpha \tau}$.
	Let $\eta_{1}$ be in $\mathbf{R}_{++}$ such that 
	$\eta_{1} \geq \frac{\sigma \norm{K}\eta_{3}}{\eta_{3} -\sigma \norm{K}}$ and
	$\alpha \tau \eta_{1}\norm{K} <1$.
	Let $\eta_{2}$ be in $\mathbf{R}_{++}$ such that 
 $\eta_{2} \geq \frac{\tau \norm{K}}{\alpha - \tau \eta_{1} \norm{K}\alpha^{2}}$. 	
	\item  \label{corollary:linearconvergeineq:Ksquare} 
	Let $\eta_{3}$ be in $\mathbf{R}_{++}$ such that 
	$\eta_{3} >\frac{\sigma}{1-\alpha \tau \sigma \norm{K}^{2}}$
	and $\norm{K}^{2}\beta^{2}\eta_{3} \leq \mu +\frac{1}{\tau} -\frac{1}{\alpha \tau}$.
	Let $\eta_{1}$ be in $\mathbf{R}_{++}$ such that 
	$\eta_{1} \geq \frac{\sigma \eta_{3}}{\eta_{3} -\sigma}$
	and $\alpha \tau \eta_{1}\norm{K}^{2} <1$.
	Let $\eta_{2}$ be in $\mathbf{R}_{++}$ such that 
	$\eta_{2} \geq \frac{\tau}{\alpha(1-\alpha \tau \eta_{1}\norm{K}^{2})}$. 
\end{enumerate}
Then there exists $M \in \mathbf{R}_{++}$ such that  for every $ k \in \mathbf{N}$,
\begin{subequations}
	\begin{align*}
		&f(\hat{x}_{k},\hat{y}_{k}) -f(x^{*},y^{*}) 
		\leq  \alpha^{k} \left(  \frac{1}{2 \tau_{0}} \norm{x^{*} -x^{0}}^{2} 
		+\frac{1}{2\sigma_{0}} \norm{\hat{y}_{k} -y^{0}}^{2} \right) \leq \alpha^{k}M;\\
		&f(x^{*},y^{*}) -f(\hat{x}_{k},\hat{y}_{k})
		\leq  \alpha^{k}\left(  \frac{1}{2 \tau_{0}} \norm{\hat{x}_{k} -x^{0}}^{2} 
		+\frac{1}{2\sigma_{0}} \norm{y^{*} -y^{0}}^{2} \right)\leq \alpha^{k}M.
	\end{align*}
\end{subequations} 
Consequently, 
the sequence $\left(f\left( \hat{x}_{k},\hat{y}_{k} \right)\right)_{k \in \mathbf{N}}$
converges linearly to the point $f(x^{*},y^{*})$ with a rate $\xi \in (0,1)$.
\end{corollary}

\begin{proof}
Clearly, it suffices to show that all requirements of \cref{theorem:linearconvergeineq}
under our assumptions on the constant involved parameters.

Set $\xi =\alpha$ in \cref{theorem:linearconvergeineq}. 

Applying  
\cref{lemma:infsupconstants}\cref{lemma:infsupconstants:munu}$\&$\cref{lemma:infsupconstants:alpha}$\&$\cref{lemma:infsupconstants:beta}
with $\zeta =1$, we are able to take 
$\tau >0$ and $\sigma >0$  small enough satisfying 
$\norm{K}^{2} <\frac{\mu +\frac{1}{\tau}}{\sigma}$ and 
$\norm{K}^{2} < \frac{\nu +\frac{1}{\sigma}}{\tau}$, 
take $\alpha\in \left( \frac{1}{  \mu \tau +1} , 1\right) \subseteq (0,1)$ satisfying 
$\alpha \geq  \frac{1}{\nu \sigma +1}$ and $\alpha \tau \sigma \norm{K}^{2} <1$,
take $\beta$ satisfying
$\sigma \norm{K}^{2}\beta^{2} \leq 
\left(1-\alpha \tau \sigma \norm{K}^{2}\right)
\left(\mu +\frac{1}{\tau} -\frac{1}{\alpha \tau}\right)$.
Moreover, because $\alpha =\xi$, 
the required condition 
$\xi \in (0,1)$ and
$\xi \sup_{k \in \mathbf{N}} \tau_{k} \sup_{k \in \mathbf{N}} \sigma_{k} 
\norm{K}^{2}< 
1$ in \cref{theorem:linearconvergeineq} hold in our case. 

Employing \cref{lemma:infsupconstants}\cref{lemma:infsupconstants:nu} 
with $\zeta =1$ and with $i=1$ under condition (a) 
(resp.\,with $i=2$ under condition (b)),
and using $\xi 
=\alpha$,
we satisfy
$
 \nu \geq \frac{1}{\xi \inf_{k \in 
		\mathbf{N}}\sigma_{k}} -\frac{1}{\sup_{k \in \mathbf{N}}\sigma_{k}}  
+ \eta_{2}\norm{K}^{i} \sup_{k \in \mathbf{N}}(\xi -\alpha_{k})^{2}
$
in  \cref{theorem:linearconvergeineq}.

Suppose condition (a) holds. 
Applying 
\cref{lemma:infsupconstants}\cref{lemma:infsupconstants:eta4}$\&$\cref{lemma:infsupconstants:eta3:a}$\&$\cref{lemma:infsupconstants:eta1a}$\&$\cref{lemma:infsupconstants:eta2a}
with $\zeta =1$, 
we are able to take $\eta_{4}$, $\eta_{3}$, $\eta_{1}$, and $\eta_{2}$ in
$\mathbf{R}_{++}$  such that 
$\eta_{4} \geq \tau$, $\alpha \sigma \eta_{4} \norm{K}^{2} < 1$,  
$\eta_{3} > \frac{\sigma \norm{K}}{ 1 -\alpha \tau \sigma \norm{K}^{2}}$,
$\eta_{3}\norm{K} \beta^{2} <  \mu +\frac{1}{\tau} 
-\frac{1}{\alpha \tau}$,
$\eta_{1} \geq \frac{\sigma \norm{K}\eta_{3}}{\eta_{3} -\sigma \norm{K}}$, 
$\alpha \tau \eta_{1}\norm{K} <1$, and
$\eta_{2} \geq \frac{\tau \norm{K}}{\alpha - \tau \eta_{1} \norm{K}\alpha^{2}}$;
moreover, our choices of $\eta_{4}$, $\eta_{3}$, $\eta_{1}$, and $\eta_{2}$ in
$\mathbf{R}_{++}$  satisfy the conditions 
$\eta_{4} \geq \sup_{i \in \mathbf{N}} \tau_{i}$,
$\frac{1}{\sup_{i \in \mathbf{N}}\sigma_{i}} -\xi \eta_{4} \norm{K}^{2} >0$,
$
 \mu\geq \frac{1}{\xi \inf_{k \in \mathbf{N}} 
	\tau_{k}} -\frac{1}{\sup_{k \in \mathbf{N}}\tau_{k}}+  \eta_{3} \norm{K} \sup_{k 
	\in \mathbf{N}}\beta_{k}^{2}
$,
$\norm{K} \left(\frac{1}{\eta_{1}} +\frac{1}{\eta_{3}}\right) \leq \frac{1}{\sup_{i \in 
		\mathbf{N}}\sigma_{i}}$, and
$\norm{K} \left( \eta_{1} \xi^{2} +\frac{1}{\eta_{2}}\right) \leq \frac{\xi}{\sup_{i \in 
		\mathbf{N}}\tau_{i}}$, 
	which are the rest conditions required in the condition (a) of   
		\cref{theorem:linearconvergeineq}.

Suppose condition (b) holds. 
Employing 
	\cref{lemma:infsupconstants}\cref{lemma:infsupconstants:eta4}$\&$\cref{lemma:infsupconstants:eta3:b}$\&$\cref{lemma:infsupconstants:eta1b}$\&$\cref{lemma:infsupconstants:eta2b},
with $\zeta =a$, we are able to take 
$\eta_{4}$, $\eta_{3}$, $\eta_{1}$, and $\eta_{2}$ in
$\mathbf{R}_{++}$  such that 
$\tau \leq \eta_{4}$, $\alpha \sigma \eta_{4}\norm{K}^{2} <1$,
$\eta_{3} >\frac{\sigma}{1-\alpha \tau \sigma \norm{K}^{2}}$,
$\norm{K}^{2}\beta^{2}\eta_{3} \leq \mu +\frac{1}{\tau} -\frac{1}{\alpha \tau}$,
$\eta_{1} \geq \frac{\sigma \eta_{3}}{\eta_{3} -\sigma}$,
$\alpha \tau \eta_{1}\norm{K}^{2} <1$,
and
$\eta_{2} \geq \frac{\tau}{\alpha(1-\alpha \tau \eta_{1}\norm{K}^{2})}$,
which ensures the conditions
$ \sup_{k \in \mathbf{N}}\tau_{k} \leq \eta_{4}$,
$\eta_{4}\xi \norm{K}^{2} <\frac{1}{ \sup_{k \in \mathbf{N}}\sigma_{k}}$,
$
 \mu\geq \frac{1}{\xi \inf_{k \in \mathbf{N}} 
	\tau_{k}} -\frac{1}{\sup_{k \in \mathbf{N}}\tau_{k}}+  \eta_{3} \norm{K}^{2}\sup_{k 
	\in \mathbf{N}}\beta_{k}^{2}
$,
and 
$\eta_{1} \norm{K}^{2}\xi +\frac{1}{\eta_{2}\xi} \leq \frac{1}{\sup_{k \in 
		\mathbf{N}}\tau_{k}}$, which are the rest conditions
required in  the condition (b) of   
\cref{theorem:linearconvergeineq}.

 Altogether, the  proof is complete.
\end{proof}

  \section*{Acknowledgments}
  Hui Ouyang thanks Professor Boyd Stephen for his insight and expertise comments on
   the topic of saddle-point problems
  and all unselfish support. 
 Hui Ouyang acknowledges the support of
 the Natural Sciences and Engineering Research Council of Canada (NSERC), 
 [funding reference number PDF – 567644 – 2022]. 
 
 
 \addcontentsline{toc}{section}{References}
 \bibliographystyle{abbrv}
 \bibliography{ccspp_proximity_stronglyconvex}

\begin{thebibliography}{1}

\bibitem{BC2017}
H.~H. Bauschke and P.~L. Combettes.
\newblock {\em Convex analysis and monotone operator theory in {H}ilbert
  spaces}.
\newblock CMS Books in Mathematics/Ouvrages de Math\'{e}matiques de la SMC.
  Springer, Cham, second edition, 2017.
\newblock With a foreword by H\'{e}dy Attouch.

\bibitem{BoctCsetnekSedlmayer2022accelerated}
R.~I. Bo{\c{t}}, E.~R. Csetnek, and M.~Sedlmayer.
\newblock An accelerated minimax algorithm for convex-concave saddle point
  problems with nonsmooth coupling function.
\newblock {\em Computational Optimization and Applications}, pages 1--42, 2022.

\bibitem{ChambollePock2011}
A.~Chambolle and T.~Pock.
\newblock A first-order primal-dual algorithm for convex problems with
  applications to imaging.
\newblock {\em Journal of Mathematical Imaging and Vision}, 40(1):120--145,
  2011.

\bibitem{Oy2023ccspp}
H.~Ouyang.
\newblock Alternating proximal point algorithm with gradient descent and ascent
  steps for convex-concave saddle-point problems.
\newblock {\em arXiv preprint arXiv:}, pages 1--25, 2023.

\bibitem{Oy2023Proximity}
H.~Ouyang.
\newblock Alternating proximity mapping method for convex-concave saddle-point
  problems.
\newblock {\em arXiv preprint arXiv:}, pages 1--40, 2023.

\bibitem{Oy2023subgradient}
H.~Ouyang.
\newblock Alternating subgradient methods for convex concave saddle-point
  problems.
\newblock {\em arXiv preprint arXiv:2305.15653}, pages 1--22, 2023.

\bibitem{Oy2023fccspp}
H.~Ouyang.
\newblock A note on the convergence of the ogaprox.
\newblock {\em arXiv preprint arXiv:}, pages 1--15, 2023.

\end{thebibliography}

\end{document}